\documentclass[11pt,b5paper,notitlepage]{article}
\usepackage[b5paper, margin={0.5in,0.65in}]{geometry}

\usepackage{amsmath,amscd,amssymb,amsthm,mathrsfs,amsfonts,layout,indentfirst,graphicx,caption,mathabx, stmaryrd,appendix,calc,imakeidx,upgreek,appendix} % mathabx for \widecheck
\usepackage[dvipsnames]{xcolor}
\usepackage{palatino}  %template
\usepackage{slashed} % Dirac operator
\usepackage{mathrsfs} % Enable using \mathscr
\usepackage{extarrows} % long equal sign, \xlongequal{blablabla}
\usepackage{enumitem} % enumerate label change e.g. [label=(\alph*)]  shows (a) (b) 
\usepackage[all]{xy}
\usepackage{fancyhdr} % date in footer
\usepackage{verbatim}

\usepackage{tikz-cd}
\usepackage[nottoc]{tocbibind}   % Add  reference to ToC

\usepackage{float}

\usepackage{lipsum}
\let\OLDthebibliography\thebibliography
\renewcommand\thebibliography[1]{
	\OLDthebibliography{#1}
	\setlength{\parskip}{0pt}
	\setlength{\itemsep}{2pt} 
}

\allowdisplaybreaks  %allow aligns to break between pages
\usepackage{latexsym}
\usepackage{chngcntr}
\usepackage[colorlinks,linkcolor=blue,anchorcolor=blue, linktocpage,
]{hyperref}
\hypersetup{ urlcolor=cyan,
	citecolor=[rgb]{0,0.5,0}}

\usepackage{simpler-wick}

\DeclareFontFamily{U}{rcjhbltx}{}
\DeclareFontShape{U}{rcjhbltx}{m}{n}{<->s*[1.15]rcjhbltx}{}   % replace <->s*[1.2]rcjhbltx with <->rcjhbltx for the normal size
\DeclareSymbolFont{hebrewletters}{U}{rcjhbltx}{m}{n}

\DeclareMathSymbol{\lamed}{\mathord}{hebrewletters}{108}
\DeclareMathSymbol{\mem}{\mathord}{hebrewletters}{109}
\DeclareMathSymbol{\ayin}{\mathord}{hebrewletters}{96}
\DeclareMathSymbol{\tsadi}{\mathord}{hebrewletters}{118}
\DeclareMathSymbol{\qof}{\mathord}{hebrewletters}{113}
\DeclareMathSymbol{\shin}{\mathord}{hebrewletters}{152}

\counterwithin{figure}{section}

\theoremstyle{definition}
\newtheorem{df}{Definition}[section]

\newtheorem{rem}[df]{Remark}

\newtheorem{cv}[df]{Convention}

\theoremstyle{plain}
\newtheorem{thm}[df]{Theorem}

\newtheorem{pp}[df]{Proposition}
\newtheorem{co}[df]{Corollary}
\newtheorem{lm}[df]{Lemma}

\newcommand{\fk}{\mathfrak}
\newcommand{\mc}{\mathcal}
\newcommand{\wtd}{\widetilde}

\newcommand{\wch}{\widecheck}
\newcommand{\ovl}{\overline}

\newcommand{\tr}{\mathrm{t}} %transpose
\newcommand{\Tr}{\mathrm{Tr}}
\newcommand{\End}{\mathrm{End}} %endomorphism
\newcommand{\idt}{\mathbf{1}}
\newcommand{\Hom}{\mathrm{Hom}}

\newcommand{\Res}{\mathrm{Res}}

\newcommand{\opp}{\mathrm{op}}

\newcommand{\diag}{\mathrm{diag}}

\newcommand{\SV}{\mathscr{V}}
\newcommand{\Span}{\mathrm{Span}}

\newcommand{\Vect}{\mathcal Vect}
\newcommand{\scr}{\mathscr}

\newcommand{\im}{\mathbf{i}}

\newcommand{\SLF}{\mathrm{SLF}}

\newcommand{\mbb}{\mathbb}

\newcommand{\blt}{\bullet}

\newcommand{\Vbb}{\mathbb V}

\newcommand{\Xbb}{\mathbb X}

\newcommand{\Wbb}{\mathbb W}
\newcommand{\Mbb}{\mathbb M}
\newcommand{\Gbb}{\mathbb G}
\newcommand{\Cbb}{\mathbb C}
\newcommand{\Nbb}{\mathbb N}
\newcommand{\Zbb}{\mathbb Z}
\newcommand{\Pbb}{\mathbb P}

\newcommand{\Ebb}{\mathbb E}
\newcommand{\cbf}{\mathbf c}

\newcommand{\MO}{\mathcal{O}}
\newcommand{\MU}{\mathcal{U}}

\newcommand{\fx}{\mathfrak{X}}

\newcommand{\ST}{\mathscr{T}}

\newcommand{\SF}{\mathscr{F}}

\newcommand{\MG}{\mathcal G}
\newcommand{\MD}{\mathcal{D}}

\newcommand{\fc}{\mathfrak{C}}

\newcommand{\bk}[1]{\langle {#1}\rangle}
\newcommand{\bigbk}[1]{\big\langle {#1}\big\rangle}

\newcommand{\bbs}{\boxbackslash}
\newcommand{\fq}{{\mathfrak Q}}
\newcommand{\ft}{{\mathfrak T}}
\newcommand{\Mod}{\mathrm{Mod}}
\newcommand{\ModfL}{\mathrm{Mod}^{\mathrm f}_{\mathrm L}}
\newcommand{\id}{\mathrm{id}}

\newcommand{\eps}{\varepsilon}

\newcommand{\fn}{\mathfrak{N}}
\newcommand{\fy}{\mathfrak{Y}}

\newcommand{\ff}{\mathfrak{F}}
\newcommand{\fg}{\mathfrak{G}}

\newcommand{\Lan}{{\big\langle}}
\newcommand{\Ran}{{\big\rangle}}
\newcommand{\Coh}{{\mathrm{Coh}_{\mathrm L}}}
\newcommand{\Abb}{{\mathbb{A}}}
\newcommand{\hqed}{\hfill\qedsymbol}
\newcommand{\Aut}{\mathrm{Aut}}

\newcommand{\Del}{\mathrm{Del}}

\usepackage{tipa} % wierd symboles e.g. \textturnh

\newcommand{\tipaz}{\text{\textctyogh}}

\newcommand{\tipae}{\text{\textrhookrevepsilon}}
\newcommand{\tipath}{\text{\textbullseye}}

\usepackage{tipx}

\numberwithin{equation}{section}

\title{How are pseudo-$q$-traces related to (co)ends?}
\author{{\sc Bin Gui, Hao Zhang}
}
\date{}
\begin{document}\sloppy % avoid stretch into margins
	\pagenumbering{arabic}
	%\pagenumbering{gobble}
	%\newpage
	%\setcounter{page}{1}
	\setcounter{section}{-1}

	\maketitle

%%%%%%%%%%%%%%%%%%%%%%%%%%%%%%%%%%%%%%%%%%%%%%%5
\newcommand\blfootnote[1]{%
	\begingroup
	\renewcommand\thefootnote{}\footnote{#1}%
	\addtocounter{footnote}{-1}%
	\endgroup
}
% Footnote without marker

%%%%%%%%%%%%%%%%%%%%%%%%%%%%%%%%%%%%%%%%%%%%%

%\hyperlink{currentwriting}{Current writing progress}~~~~~~ 

%\hypertarget{currentwriting}{}

\begin{abstract}
Let $\Vbb$ be an $\Nbb$-graded $C_2$-cofinite vertex operator algebra (VOA), not necessarily rational or self-dual. Using a special case of the sewing-factorization theorem from \cite{GZ3}, illustrated in \hyperlink{figsfend}{Fig. 0} below, we show that the end
\begin{align*}
\Ebb=\int_{\Mbb\in\Mod(\Vbb)}\Mbb\otimes_\Cbb\Mbb'
\end{align*}
in $\Mod(\Vbb^{\otimes2})$ (where $\Mbb'$ is the contragredient module of $\Mbb$) admits a natural structure of associative $\Cbb$-algebra compatible with its $\Vbb^{\otimes2}$-module structure. Moreover, we show that a suitable category $\Coh(\Ebb)$ of left $\Ebb$-modules is isomorphic, as a linear category, to $\Mod(\Vbb)$, and that the space of vacuum torus conformal blocks is isomorphic to the space $\SLF(\Ebb)$ of symmetric linear functionals on $\Ebb$.

Combining these results with the main theorem of \cite{GZ4}, we prove a conjecture of Gainutdinov-Runkel \cite{GR-Verlinde}: For any projective generator $\Gbb$ in $\Mod(\Vbb)$, the pseudo-$q$-trace construction yields a linear isomorphism from $\SLF(\End_\Vbb(\Gbb)^\opp)$ to the space of vacuum torus conformal blocks of $\Vbb$. 

In particular, if $A$ is a unital finite-dimensional $\Cbb$-algebra such that the category of finite-dimensional left $A$-modules is equivalent to $\Mod(\Vbb)$, then $\SLF(A)$ is linearly isomorphic to the space of vacuum torus conformal blocks of $\Vbb$. This confirms a conjecture of Arike-Nagatomo \cite{AN-pseudo-trace}.
\end{abstract}

\hypertarget{figsfend}{}
\begin{figure}[H]
	\centering
\begin{gather*}
\ST^*\Bigg(\vcenter{\hbox{{
		\includegraphics[height=1cm]{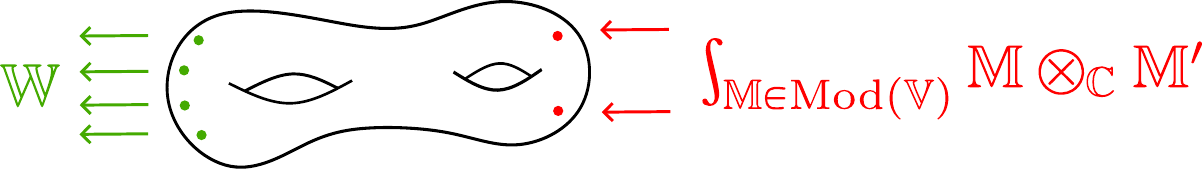}}}}\Bigg)\simeq\ST^*\Bigg(\vcenter{\hbox{{
		\includegraphics[height=1cm]{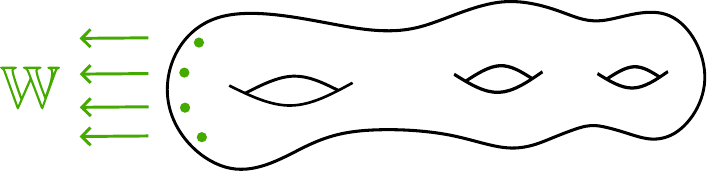}}}}\Bigg)
   \end{gather*}
\caption*{Figure 0. A pictorial illustration of the sewing-factorization isomorphism \eqref{eq138}.}
\end{figure}
\tableofcontents

%\vspace{-0.5cm}
%\blfootnote{Last major revision:  2021.6}

%%%%%%%%%%%%%%%%%%%%%%%%%%%%%%%%5
%\makeatletter
%\newcommand*{\toccontents}{\@starttoc{toc}}
%\makeatother
%\toccontents

% title and table of contents same page

%%%%%%%%%%%%%%%%%%%%%%%%%%%%%

	%%%%%%%%%%%%%%%%%%%%%%%%%%%%%%%%%%%%%%%%%%%%%%%%%%%%%%%%%
	
	%\newpage
	%$~$
	%\renewcommand\contentsname{} % the empty name
	
	%\begingroup
	%\let\clearpage\relax
	%\vspace{-2cm} % the removed space. Set as appropriate

	% Remove header of table of contents
	
	%%%%%%%%%%%%%%%%%%%%%%%%%%%%%%%%%%%%%%%%%%%%%%%%%%%%%%%

\section{Introduction}

\subsection{(Co)ends in finite logarithmic CFT}

\nocite{HLZ1,HLZ2,HLZ3,HLZ4,HLZ5,HLZ6,HLZ7,HLZ8}

A fundamental feature of rational conformal field theory (rational CFT) is the \textbf{factorization property}, which states, roughly speaking, that when a (possibly disconnected) compact Riemann surface is sewn along several pairs of marked points using prescribed local coordinates, the resulting space of conformal blocks is isomorphic to a direct sum of spaces of conformal blocks associated to the pre-sewing configuration. 

In the theory of rational and $C_2$-cofinite vertex operator algebras (VOAs), proving the factorization property in low-genus cases is a central topic in the literature. The key results in this direction include the modular invariance property \cite{Zhu-modular-invariance, Hua-differential-genus-1} and the associativity of intertwining operators \cite{Hua-tensor-4, Hua-differential-genus-0}. These are analytic in nature, meaning that the isomorphisms appearing in the factorization property are derived using Segal's formalism of sewing \cite{Segal-CFT1, Segal-CFT2}. More recently, the factorization property in arbitrary genus has been established for \textit{rational} $C_2$-cofinite VOAs using formal (algebraic) sewing, as in \cite{DGT2}; for genus zero, this formal approach was developed in \cite{NT-P1_conformal_blocks}.

Now we focus on a $C_2$-cofinite, but not necessarily rational, VOA $\Vbb=\bigoplus_{n\in\Nbb}\Vbb(n)$. The category $\Mod(\Vbb)$ of grading-restricted generalized $\Vbb$-modules is a finite abelian category as a linear category \cite{MNT10, Hua-projectivecover}. As a monoidal category, its tensor structure---defined via the formalism of Huang-Lepowsky-Zhang \cite{HLZ1}, \cite{HLZ2}-\cite{HLZ8}---forms a (possibly non-rigid) ribbon Grothendieck-Verdier category \cite{ALSW21}. 

In this setting, the study of factorization originates from the construction of non-semisimple modular functors in topological field theory (TFT), especially the work of Lyubashenko \cite{Lyu95-Invariants, Lyu96-Ribbon}. From the TFT perspective, factorization should naturally be formulated in terms of ends and coends, the definitions of which will be recalled in Def. \ref{lb75}. There are several approaches to expressing factorization via ends and coends. One is the left exact coend formulation \cite{Lyu95-Invariants, Lyu96-Ribbon,FS-coends-CFT}. Another employs the horizontal composition of profunctors, as in \cite{HR24-MF}. In fact, the Huang-Lepowsky-Zhang tensor category theory can be viewed as realizing a genus-zero sewing-factorization theorem in the language of horizontal composition of profunctors---a perspective first emphasized in \cite{Moriwaki22-CB}. 

On the other hand, the modular invariance property is typically regarded as a genus-one sewing-factorization theorem. In the non-rational case, one well-established formulation of modular invariance is expressed in terms of pseudo-$q$-traces \cite{Miy-modular-invariance,AN-pseudo-trace,Fio-genus-1,Hua-modular-C2}. However, the connection between this formulation and the end/coend perspective remains unclear. The aim of this paper is to clarify that relationship.

\subsection{The sewing-factorization (SF) theorem}

In \cite{GZ3}, we established several equivalent formulations of the sewing-factorization theorem for any $C_2$-cofinite VOA $\Vbb=\bigoplus_{n\in\Nbb}\Vbb(n)$. One such formulation appears in the language of horizontal composition of profunctors; see \cite[Sec. 3.2]{GZ3}. While this is a coend-based expression of the factorization property, its relation to pseudo-$q$-traces is not immediately transparent. In the following, we recall the version of the sewing-factorization theorem stated in terms of (dual) fusion products, as proved in \cite[Sec. 3.1]{GZ3} and reviewed in detail in Sec. \ref{lb76}.

Let $\fg$ be a (possibly disconnected) compact Riemann surface with disjoint sets $G'$ and $G$ of outgoing and incoming marked points, respectively, such that each connected component of $\fg$ intersects $G'\cup G$. Suppose each point in $G'\cup G$ is equipped with a local coordinate. Let $N = |G'|$ and $R = |G|$. Fix orderings of $G'$ and $G$, that is, bijections
\begin{align*}
\{1,\dots,N\}\xlongrightarrow{\simeq}G'\qquad \{1,\dots,R\}\xlongrightarrow{\simeq}G
\end{align*}
Let $\Vect$ be the category of finite-dimensional $\Cbb$-vector spaces. Then we have a left exact (cf. Rem. \ref{lb38}) profunctor
\begin{gather*}
\Mod(\Vbb^{\otimes N})\times\Mod(\Vbb^{\otimes R})\rightarrow\Vect\\
(\Mbb,\Xbb)\quad\mapsto\quad\ST^*\Bigg(\vcenter{\hbox{{
		\includegraphics[height=1.2cm]{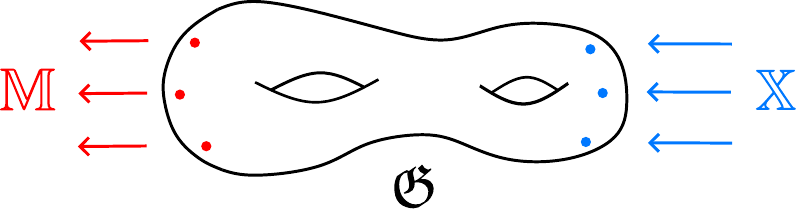}}}}\Bigg)
\end{gather*}
where the notation $\ST^*(\cdots)$ on the RHS denotes the space of conformal blocks over $\fg$ with input module $\Xbb$ and output module $\Mbb$. (The roles of input and output modules can be interchanged by taking contragredient modules.) This profunctor is covariant in $\Mbb$ and contravariant in $\Xbb$. For a detailed definition and interpretation of the figure, see Sec. \ref{lb77}.

By \cite{DSPS19-balanced}, every linear functor from a finite $\Cbb$-linear category to $\Vect$ is representable. Therefore, for each fixed $\Xbb\in\Mod(\Vbb^{\otimes R})$, there exists an $\Mbb$-natural linear isomorphism
\begin{align}\label{eq134}
\Hom_{\Vbb^{\otimes N}}(\boxtimes_\fg\Xbb,\Mbb)\xlongrightarrow{\simeq}\ST^*\bigg(\vcenter{\hbox{{
		\includegraphics[height=1.2cm]{fig27a.pdf}}}}\bigg)
\end{align}
Fix such a natural isomorphism, and let
\begin{align*}
\daleth\in \ST^*\bigg(\vcenter{\hbox{{
		\includegraphics[height=1.2cm]{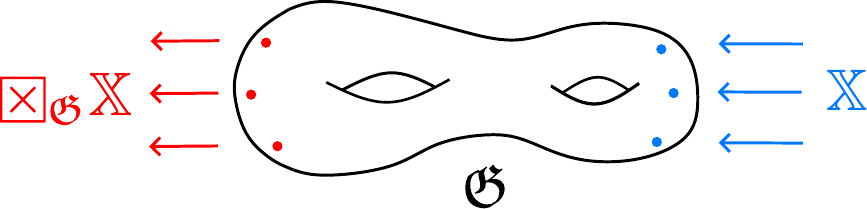}}}}\bigg)
\end{align*}
be the element corresponding to $\id_{\boxtimes_\fg\Xbb}\in\End_{\Vbb^{\otimes N}}(\boxtimes_\fg\Xbb)$ under the isomorphism \eqref{eq134}. The pair $(\boxtimes_\fg\Xbb,\daleth)$ is called a \textbf{fusion product of $\Xbb$ along $\fg$}, and $\daleth$ is called the \textbf{canonical conformal block}. Note that this construction depends on the chosen orderings of the incoming and outgoing marked points $G$ and $G'$. For simplicity, we suppress this dependence in the introduction, but it will be made explicit in the main body of the paper.

Now suppose we are given, similarly to $\fg$, a compact Riemann surface $\ff$ with disjoint sets $F',F$ of outgoing and incoming marked points. Let $K=|F'|$ and assume that $|F|=N=|G'|$. Fix orderings of $F'$ and $F$. Assume that each component of $\ff$ intersects $F'\cup F$. Then, using the chosen orderings, we can analytically sew $\ff$ and $\fg$ along $F$ and $G'$, producing a new surface $\ff \# \fg$:
\begin{align*}
\vcenter{\hbox{{
		\includegraphics[height=1.6cm]{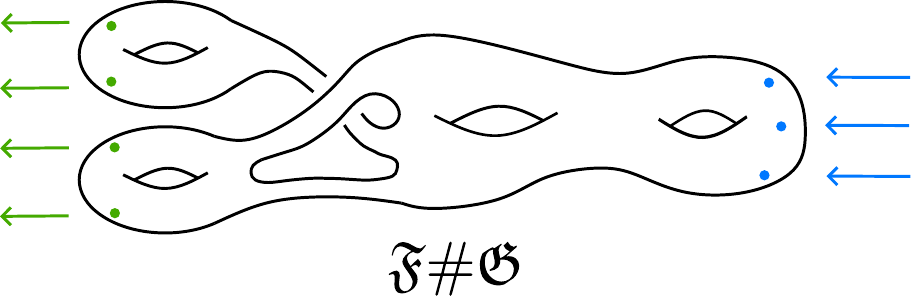}}}}\quad=\quad\vcenter{\hbox{{
		\includegraphics[height=1.6cm]{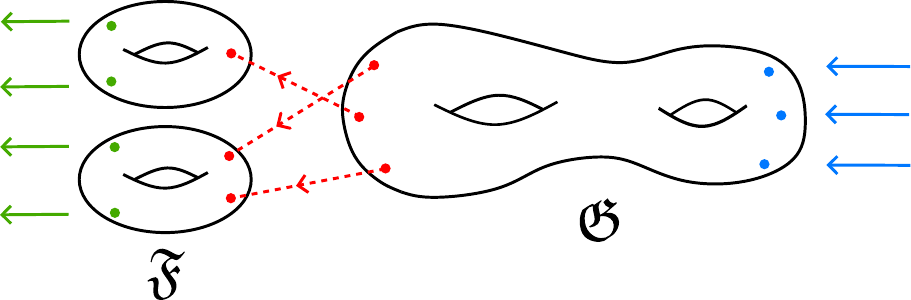}}}}
\end{align*}
(Since we are performing analytic sewing, we must choose sewing moduli. In the introduction, as well as in many parts of this paper, we fix all sewing moduli to be 1; see Subsec. \ref{lb46} for details.)

Assume that each component of $\ff\#\fg$ intersects $F'\cup G$. The \textbf{sewing-factorization (SF) theorem} says that for each $\Wbb\in\Mod(\Vbb^{\otimes K})$ and $\Xbb\in\Mod(\Vbb^{\otimes R})$, we have a linear isomorphism (called the \textbf{sewing-factorization isomorphism})
\begin{gather}\label{eq135}
\begin{gathered}
\ST^*\Bigg(\vcenter{\hbox{{
		\includegraphics[height=1.6cm]{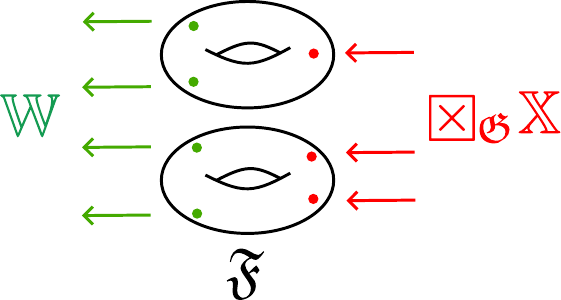}}}}\Bigg)\quad\xlongrightarrow{\simeq}\quad\ST^*\Bigg(\vcenter{\hbox{{
		\includegraphics[height=1.6cm]{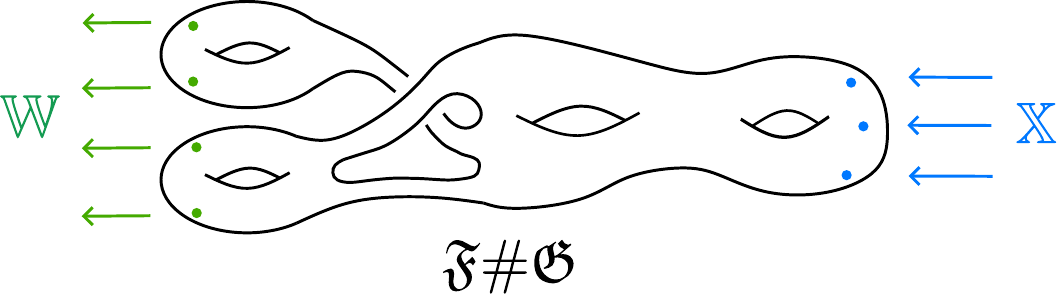}}}}\Bigg)\\ \upphi\mapsto\upphi\circ\daleth
\end{gathered}
\end{gather}
where $\upphi\circ\daleth:\Wbb'\otimes\Xbb\rightarrow\Cbb$ is the contraction of $\upphi:\Wbb'\otimes\boxtimes_\fg\Xbb\rightarrow\Cbb$ and $\daleth:\bbs_\fg\Xbb\otimes\Xbb\rightarrow\Cbb$, with $\bbs_\fg\Xbb$ being the contragredient module of $\boxtimes_\fg\Xbb$ (called the \textbf{dual fusion product}). 

We remark that in \cite{GZ3}, the SF theorem is stated under the assumption that each component of $\fg$ intersects $G$. This condition can be removed by invoking the propagation of conformal blocks; see Thm. \ref{SF}. (In particular, when $\fg$ is connected, it is allowed to have no incoming marked points. In that case, since $\fg$ intersects $G'\cup G$, the set $G'$ of outgoing marked points must be non-empty.)

\subsection{The SF theorem in terms of the end $\int_{\Mbb\in\Mod(\Vbb)}\Mbb\otimes_\Cbb\Mbb'$}

An important special case of the above SF theorem is when $\ff$ has two incoming points, and $\fg$ is the sphere $\fn$ with no incoming and two outgoing points $\infty,0$, equipped with the local coordinates $1/\zeta,\zeta$ respectively, where $\zeta$ is the standard coordinate of $\Cbb$. We refer to $\fn$ as the \textbf{default $2$-pointed sphere}. Let $(\boxtimes_\fn\Cbb,\upomega)$ be a fusion product of the complex field $\Cbb\in\Mod(\Vbb^{\otimes0})$ along $\fn$. Assume that
\begin{align*}
\ff\quad=\quad\vcenter{\hbox{{
		\includegraphics[height=1cm]{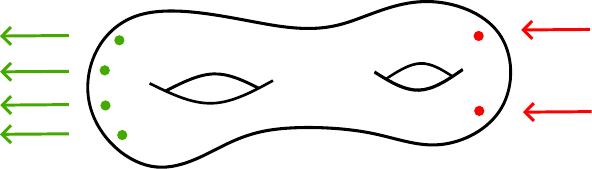}}}}
\end{align*}
Then, since
\begin{align}
\vcenter{\hbox{{
		\includegraphics[height=1.1cm]{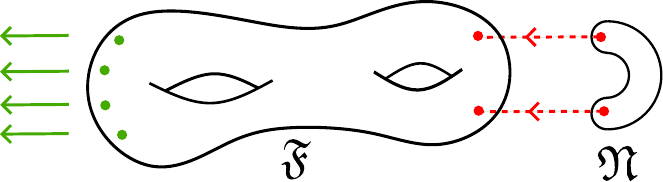}}}}\quad=\quad\vcenter{\hbox{{
		\includegraphics[height=1cm]{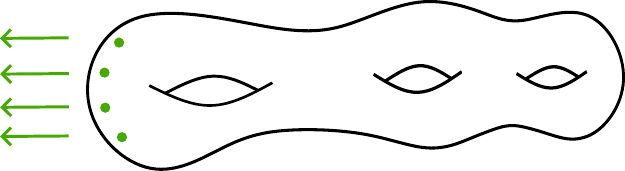}}}}
\end{align}
the SF isomorphism \eqref{eq135} becomes
\begin{subequations}\label{eq138}
\begin{gather}\label{eq136}
\begin{gathered}
\ST^*\Bigg(\vcenter{\hbox{{
		\includegraphics[height=1cm]{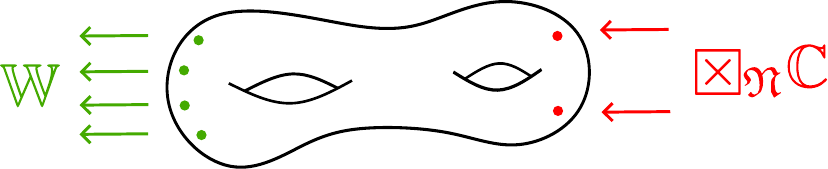}}}}\Bigg)\quad\xlongrightarrow{\simeq}\quad\ST^*\Bigg(\vcenter{\hbox{{
		\includegraphics[height=1cm]{fig28c.pdf}}}}\Bigg)\\ \upphi\mapsto\upphi\circ\upomega
\end{gathered}
\end{gather}
where $\Wbb\in\Mod(\Vbb^{\otimes K})$, and $K$ is (again) the number of outgoing points of $\ff$.

By the propagation of conformal blocks, $\boxtimes_\fn\Cbb$ is isomorphic to the fusion product $\boxtimes_\fq\Vbb$ of $\Vbb$ along a sphere with one input and two outputs; see Thm. \ref{lb36} for details.\footnote{The contragredient of $\boxtimes_\fq\Vbb$ was first considered by Li in \cite{Li-regular-rep}, where it was referred to as the regular representation of $\Vbb$.} In \cite[Sec. 0.6]{GZ3}, we explained why $\boxtimes_\fq\Vbb$ can be viewed as the end $\int_{\Mbb\in\Mod(\Vbb)}\Mbb\otimes\Mbb'$. Therefore, we have an isomorphism of $\Vbb^{\otimes 2}$ modules
\begin{align}\label{eq137}
\boxtimes_\fn\Cbb\simeq\int_{\Mbb\in\Mod(\Vbb)}\Mbb\otimes_\Cbb\Mbb'
\end{align}
\end{subequations}
(An alternative proof of \eqref{eq137} will be given in this paper; see Thm. \ref{end}.) Therefore, \eqref{eq136} yields the isomorphism shown in \hyperlink{figsfend}{Figure 0} following the abstract. See Rem. \ref{SF1} for further explanation.

\subsection{The Arike-Nagatomo conjecture is an easy consequence of the SF theorem}\label{lb80}

As a special case of the isomorphism in \hyperlink{figsfend}{Figure 0}, we have
\begin{align*}
&\ST^*\bigg(\vcenter{\hbox{{
		\includegraphics[height=1cm]{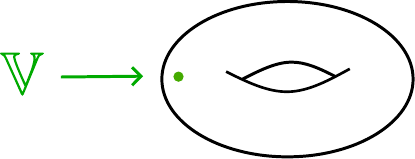}}}}\bigg)\quad\simeq\quad\ST^*\bigg(\vcenter{\hbox{{
		\includegraphics[height=1cm]{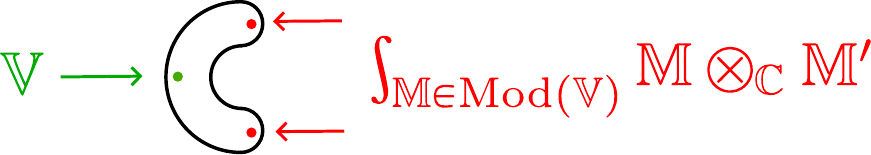}}}}\bigg)\\[1ex]
\simeq\quad&\ST^*\bigg(\vcenter{\hbox{{
		\includegraphics[height=1cm]{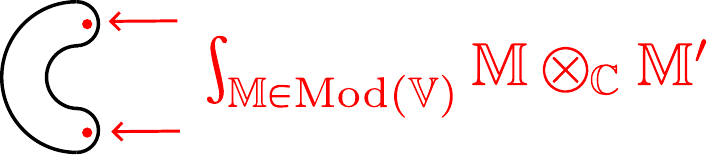}}}}\bigg)
\end{align*}
where the last equality is due to the propagation of conformal blocks. Since the transpose of the dinatural transform $\int_{\Mbb\in\Mod(\Vbb)}\Mbb\otimes\Mbb'\rightarrow\Mbb\otimes\Mbb'$ is $\Mbb'\otimes\Mbb\rightarrow\big(\int_{\Mbb\in\Mod(\Vbb)}\Mbb\otimes\Mbb'\big)'$, the latter must satisfy the universal property required for a coend. Thus
\begin{align*}
\Big(\int_{\Mbb\in\Mod(\Vbb)}\Mbb\otimes\Mbb'\Big)'\simeq\int^{\Mbb\in\Mod(\Vbb)}\Mbb'\otimes\Mbb
\end{align*}
Therefore, we have
\begin{align*}
\ST^*\bigg(\vcenter{\hbox{{
		\includegraphics[height=1cm]{fig29a.pdf}}}}\bigg)\quad\simeq\quad\ST^*\bigg(\vcenter{\hbox{{
		\includegraphics[height=1cm]{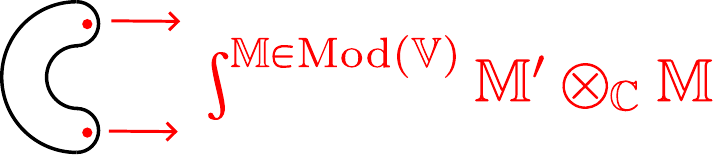}}}}\bigg)
\end{align*}
By \eqref{eq134} and \eqref{eq137}, the RHS above is isomorphic to
\begin{align}\label{eq139}
\Hom_{\Vbb^{\otimes2}}\Big(\int_{\Mbb\in\Mod(\Vbb)}\Mbb\otimes\Mbb',\int^{\Mbb\in\Mod(\Vbb)}\Mbb'\otimes\Mbb\Big)
\end{align}
Therefore, the space of vacuum torus conformal blocks is isomorphic to \eqref{eq139}. This immediately implies the following theorem, originally conjectured by Arike-Nagatomo in the Introduction of \cite{AN03-finite-dimensional}.

\begin{thm}\label{lb78}
Let $A$ be a unital finite-dimensional $\Cbb$-algebra such that $\Mod(\Vbb)$ is linearly isomorphic to the category $\ModfL(A)$ of finite dimensional left $A$-modules. Then we have a linear isomorphism
\begin{align}\label{eq140}
\ST^*\bigg(\vcenter{\hbox{{
		\includegraphics[height=1cm]{fig29a.pdf}}}}\bigg)\quad\simeq\quad\SLF(A)
\end{align}
where $\SLF(A)$ denotes the space of symmetric linear functionals on $A$.
\end{thm}
By a symmetric linear functional on $A$, we mean a linear map $\phi:A\rightarrow\Cbb$ satisfying $\phi(ab)=\phi(ba)$ for all $a,b\in A$.

\begin{proof}
We have used the SF theorem to prove that the LHS of \eqref{eq140} is linearly isomorphic to \eqref{eq139}. Note that
\begin{align*}
\ModfL(A)^\opp\simeq\Mod(\Vbb)^\opp\simeq\Mod(\Vbb)
\end{align*}
where the isomorphism $\Mod(\Vbb)^\opp\simeq\Mod(\Vbb)$ is defined by sending each $\Mbb^\opp$ to the contragredient $\Mbb'$ of $\Mbb$. By \cite{McR-deligne}, 
\begin{align*}
\Mod(\Vbb)\times\Mod(\Vbb)\rightarrow\Mod(\Vbb^{\otimes2})\qquad(\Xbb,\mbb Y)\mapsto\Xbb\otimes_\Cbb\mbb Y
\end{align*}
is a Deligne product. Therefore, \eqref{eq139} can be written as
\begin{align*}
&\Hom_{\Mod(\Vbb)\otimes^\Del\Mod(\Vbb)}\Big(\int_{\Mbb\in\Mod(\Vbb)}\Mbb\otimes^\Del\Mbb',\int^{\Mbb\in\Mod(\Vbb)}\Mbb'\otimes^\Del\Mbb\Big)\\
\simeq&\Hom_{\ModfL(A)\otimes^\Del\ModfL(A)^\opp}\Big(\int_{M\in\ModfL(A)}M\otimes^\Del M^\opp,\int^{M\in\ModfL(A)}M^\opp\otimes^\Del M\Big)
\end{align*}
where $\otimes^\Del$ denotes the Deligne product. By \cite[Cor. 2.9]{FSS20}, if we identify $\ModfL(A)\otimes^\Del\ModfL(A)^\opp$ with the category $\mathrm{Bim}^{\mathrm f}(A)$ of finite-dimensional $A$-bimodules, then the last Hom space above is equivalent to
\begin{align*}
\Hom_{\mathrm{Bim}^{\mathrm f}(A)}(A,A^*)\simeq\SLF(A)
\end{align*}
where each $T\in \Hom_{\mathrm{Bim}^{\mathrm f}(A)}(A,A^*)$ corresponds to $1_A^\tr\circ T\in\SLF(A)$ and $1_A^\tr:A^*\rightarrow\Cbb$ is the transpose of $\lambda\in\Cbb\mapsto \lambda\cdot 1_A\in A$. This establishes the isomorphism \eqref{eq140}.
\end{proof}

\subsection{The Gainutdinov-Runkel conjecture on pseudo-$q$-traces}

The isomorphism \eqref{eq140} established in Thm. \ref{lb78} is fairly abstract, and it is natural to seek an explicit linear map that realizes this isomorphism. Such a map was proposed by Gainutdinov and Runkel in \cite{GR-Verlinde}. Specifically, Conjecture 5.8 of \cite{GR-Verlinde} asserts that if $\Gbb$ is a projective generator in the abelian category $\Mod(\Vbb)$, then the pseudo-$q$-trace construction (in the sense of \cite{AN-pseudo-trace}) yields a linear isomorphism
\begin{align}\label{eq141}
\SLF(\End_\Vbb(\Gbb)^\opp)\xlongrightarrow{\simeq}\ST^*\bigg(\vcenter{\hbox{{
		\includegraphics[height=1cm]{fig29a.pdf}}}}\bigg)
\end{align}

We do not recall the definition of the pseudo-$q$-trace construction here; see Sec. \ref{lb79} for details, or the Introduction of \cite{GZ4} for a brief overview. However, let us explain why the Arike-Nagatomo conjecture is a special case of the Gainutdinov-Runkel conjecture: Suppose that $\Mod(\Vbb)\simeq\ModfL(A)$ as linear categories, where $A$ is a unital finite-dimensional algebra. Then $A$, as a left $A$-module, is a projective generator of $\ModfL(A)$, and $\End_{\ModfL(A)}(A)^\opp\simeq A$. By choosing any $\Gbb\in\Mod(\Vbb)$ corresponding to the object $A$ of $\ModfL(A)$, we recover the isomorphism \eqref{eq140} from \eqref{eq141}.

We emphasize that in the main body of this paper, our proof of the Gainutdinov-Runkel conjecture, formally stated in Thm. \ref{lb68}, does not rely on first establishing the Arike-Nagatomo conjecture. In fact, our argument provides an independent proof of the Arike-Nagatomo conjecture, separate from the one given in Sec. \ref{lb80}. Furthermore, our proof does not assume the isomorphism $\boxtimes_\fn\Cbb\simeq\int_{\Mbb\in\Mod(\Vbb)}\Mbb\otimes\Mbb'$; rather, the alternative proof of this isomorphism is a consequence of the techniques developed for proving the Gainutdinov-Runkel conjecture, as we will see in Sec. \ref{lb81}.

To prove the Gainutdinov-Runkel conjecture, we must relate pseudo-$q$-traces to ends and coends---in other words, to answer the question posed in the title of this paper. Our answer, in brief, is as follows:
\begin{enumerate}[label=(\arabic*)]
\item The vector space $\boxtimes_\fn\Cbb$ can be equipped with an associative $\Cbb$-algebra structure that is compatible with its $\Vbb^{\otimes2}$-structure (Cor. \ref{lb82}). Moreover, the algebra $\boxtimes_\fn\Cbb$ is \textbf{almost unital and finite-dimensional (AUF)} in the sense of \cite{GZ4} (Cor. \ref{lb54}). 
\item There is a canonical linear isomorphism from the category $\Coh(\boxtimes_\fn\Cbb)$ of coherent left $\boxtimes_\fn\Cbb$-modules (cf. Def. \ref{lb83}) to the abelian category $\Mod(\Vbb)$ (Thm. \ref{lb55}). Therefore, the projective generators of the two categories can be identified.
\item By \eqref{eq136} and the propagation of conformal blocks, the sewing-factorization isomorphism implements an isomorphism from the following space of conformal blocks
\begin{align}
\ST^*\bigg(\vcenter{\hbox{{
		\includegraphics[height=1cm]{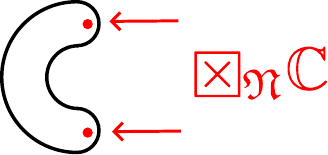}}}}\bigg)
\end{align}
(which can be identified with $\SLF(\boxtimes_\fn\Cbb)$, cf. Thm. \ref{lb48}) to the space of vacuum torus conformal blocks.
\item By the main result (Thm. 9.4) of \cite{GZ4} on symmetric linear functionals of AUF algebras, the pseudotrace construction yields a linear isomorphism
\begin{align}\label{eq142}
\SLF(\End_\Vbb(\Gbb))\xlongequal{\text{by (2)}}\SLF(\End_{\boxtimes_\fn\Cbb}(\Gbb))\xlongrightarrow{\simeq}\SLF(\boxtimes_\fn\Cbb)
\end{align}
Due to (3), the sewing-factorization isomorphism yields an isomorphism from $\SLF(\boxtimes_\fn\Cbb)$ to the space of vacuum torus conformal blocks. One can show that the composition of the pseudotrace construction \eqref{eq142} and the SF isomorphism equals the pseudo-$q$-trace construction. The proof of the Gainutdinov-Runkel conjecture is finished.
\end{enumerate}

\subsection{The cobordism geometry of associative $\Cbb$-algebras}

According to the discussion above, the key to answering the question ``how are pseudo-$q$-traces related to (co)ends" lies in the fact that the end $\boxtimes_\fn\Cbb$ naturally carries the structure of an AUF algebra. Our approach to studying torus conformal blocks via infinite-dimensional associative algebras is partly inspired by Huang's construction of the algebra $A^\infty(\Vbb)$ in \cite{Hua-associative,Hua22-Ass-IO} and his use of this algebra in \cite{Hua-modular-C2} to establish modular invariance of intertwining operators for $C_2$-cofinite VOAs. Another type of infinite-dimensional algebra, the so called mode transition algebra, was considered by Damiolini-Gibney-Krashen \cite{DGK2,DGK3-morita}, and was conjectured in \cite{DW-modular-functor} to be closely related to the end $\int_{\Mbb\in\Mod(\Vbb)}\Mbb\otimes\Mbb'$, although this connection remains unclear in the absence of rationality assumption on the $C_2$-cofinite VOA $\Vbb$.

Our construction of the associative algebra structure on $\boxtimes_\fn\Cbb$ differs fundamentally from all previous approaches to associative algebras in the VOA context: it is purely geometric, in the sense of Segal's CFT and cobordism categories. 

It is well-known that the vertex operator $Y(-,z)$, associated to any $\Vbb$-module, gives a conformal block associated to a sphere with two inputs and one output. In contrast, the geometric realization of the Zhu algebra $A(\Vbb)$ and the higher Zhu algebras $A_n(\Vbb)$, is far less transparent. However, the geometry for the algebra structure of $\boxtimes_\fn\Cbb$ is much clearer: the multiplication map $\diamond:\boxtimes_\fn\Cbb\otimes\boxtimes_\fn\Cbb\rightarrow\boxtimes_\fn\Cbb$ belongs to the following space of conformal blocks:
\begin{align}\label{eq143}
\ST^*\Bigg(\vcenter{\hbox{{
		\includegraphics[height=2.6cm]{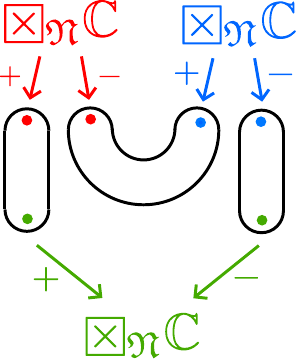}}}}~\Bigg)
\end{align}
where the signs $+$ and $-$ indicate the ordering of marked points; see Sec. \ref{lb77} for the precise meaning of the graphical notation for conformal blocks. Moreover, for any $\Mbb\in\Mod(\Vbb)$, the corresponding left $\boxtimes_\fn\Cbb$-module structure on $\Mbb$, given by a linear map $\boxtimes_\fn\Cbb\otimes\Mbb\rightarrow\Mbb$, belongs to the following space of conformal blocks:
\begin{align}
\ST^*\Bigg(\vcenter{\hbox{{
		\includegraphics[height=2.4cm]{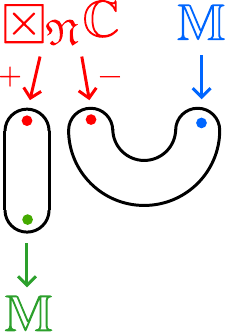}}}}~\Bigg)
\end{align}
These two conformal blocks will be defined precisely in Def. \ref{lb49} in a uniform way using the SF theorem. This provides yet another illustration of the power of the SF theorem established in \cite{GZ3}.

According to \cite{Li-regular-Zhu,Li-regular-AnV} (see also the appendex Chapter B of the arXiv version of \cite{GZ1}), the Zhu algebra $A(\Vbb)$ and the higher Zhu algebra $A_n(\Vbb)$ can be realized as quotient spaces of $\boxtimes_\fn\Cbb$, and their algebra structures can be defined via the $\Vbb^{\otimes2}$-module structure of $\boxtimes_\fn\Cbb$. Since our associative algebra structure on $\boxtimes_\fn\Cbb$ is compatible with this $\Vbb^{\otimes2}$-module structure, we expect that the algebra structure on $\boxtimes_\fn\Cbb$ descends to those of $A(\Vbb)$ and $A_n(\Vbb)$. This, in turn, provides a cobordism-geometric interpretation of $A(\Vbb)$ and $A_n(\Vbb)$: the geometries of the associative algebras $A(\Vbb)$ and $A_n(\Vbb)$ should be viewed as the zero-level and finite-level truncations of a distinguished conformal block in \eqref{eq143}, namely, the conformal block defining the multiplication map $\diamond:\boxtimes_\fn\Cbb\otimes\boxtimes_\fn\Cbb\rightarrow\boxtimes_\fn\Cbb$.

\subsection*{Acknowledgment}

We would like to thank Jurgen Fuchs, Robert McRae, Shuang Ming, Ingo Runkel, Christoph Schweigert, Yilong Wang,  Lukas Woike, Baojun Wu, and Jinwei Yang for helpful discussions. Special thanks go to Yi-Zhi Huang for encouraging us to explore the connection between our series of works \cite{GZ1,GZ2,GZ3} and the pseudo-$q$-traces, a relationship that was not at all apparent to us when the outlines of these three articles were initially conceived.

\section{Conformal blocks and their graphical calculus}

\subsection{Notation}\label{lb32}

Throughout this paper, we use the following notation.

\begin{itemize}
\item $\Nbb=\{0,1,2,\dots\}$, $\Zbb_+=\{1,2,\dots\}$. Neighborhoods are assumed to be open.
\item Let $\Cbb^\times=\Cbb\setminus\{0\}$. For each $r\in[0,+\infty]$, we let
\begin{align*}
\MD_r=\{z\in\Cbb:|z|<r\}\qquad \MD_r^\times=\{z\in\Cbb:0<|z|<r\}
\end{align*}
\item Let $\Vect$ be the abelian category of finite-dimensional $\Cbb$-vector spaces.
\item Let $\zeta$ be the \textbf{standard coordinate of $\pmb\Cbb$}, namely, the identity map $\id:\Cbb\rightarrow\Cbb$. 
\item For each complex manifold $X$, $\MO_X$ denotes the sheaf of germs of holomorphic functions on $X$. Therefore, $\MO_X(X)=\MO(X)$ is the space of holomorphic functions $X\rightarrow\Cbb$. We let $\omega_X$ be the sheaf of germs of holomorphic $1$-forms on $X$.
	\item Throughout this paper, we fix an $\Nbb$-graded $C_2$-cofinite vertex operator algebra (VOA) $\Vbb=\bigoplus_{n\in \Nbb}\Vbb(n)$ with conformal vector $\cbf$ and vacuum vector $\idt$. For each $n\in\Nbb$, we let
\begin{align*}
\Vbb^{\leq n}=\bigoplus_{k\leq n}\Vbb(k)
\end{align*}
For each $N\in\Nbb$,  we let $\pmb{\Mod(\Vbb^{\otimes N})}$ be the category of grading-restricted generalized $\Vbb^{\otimes N}$-modules, which is an abelian category by \cite{Hua-projectivecover} (see also \cite{MNT10}). We refer the reader to \cite{Hua-projectivecover} for the general properties of grading-restricted generalized modules of $C_2$-cofinite VOAs.

	\item Recall from \cite[Sec. 1.1]{GZ2} that if $\Wbb\in\Mod(\Vbb^{\otimes N})$ and $v\in\Vbb$, the $i$-th vertex operator
\begin{align*}
Y_{\Wbb,i}(v,z)=\sum_{n\in\Zbb}Y_{\Wbb,i}(v)_nz^{-n-1}
\end{align*}
is $Y(\idt\otimes\cdots\otimes v\otimes\cdots\otimes\idt,z)$ where $v$ is at the $i$-th component. We abbreviate $Y_{\Wbb,i}$ to $Y_i$ when no confusion arises. We also write
	\begin{subequations}\label{eq1}
	\begin{align}
	Y_i'(v,z)=Y_i(\MU(\upgamma_z)v,z^{-1})
	\end{align}
	where $\MU(\upgamma_z)=e^{zL(1)}(-z^{-2})^{L(0)}$ is the change-of-coordinate operator (cf. Subsec. \ref{lb52}) associated to
\begin{align}\label{eq132}
\upgamma_z:t\mapsto\frac{1}{z+t}-\frac 1z
\end{align}
Clearly $\MU(\upgamma_z)^{-1}=\MU(\upgamma_{1/z})$, and hence
	\begin{align}
	Y_i(v,z)=Y_i'(\MU(\upgamma_z)v,z^{-1})
	\end{align}
	\end{subequations}
We write
\begin{align}\label{eq2}
Y_i'(v,z)=\sum_{n\in\Zbb}Y_i'(v)_nz^{-n-1}
\end{align}
We also write
\begin{align}
Y_+=Y_1\quad Y_+'=Y_1'\quad Y_-=Y_2\quad Y_-'=Y_2'\qquad\text{for $\Vbb^{\otimes 2}$-modules}
\end{align}
Let $L_i(n)=Y_i(\cbf)_{n-1}$. 
\item 	If $\Wbb\in\Mod(\Vbb^{\otimes N})$ and $\lambda_1,\dots,\lambda_N\in\Cbb$, then $\Wbb_{[\lambda_\blt]}$ is the subspace of all $w\in\Wbb$ such that for all $1\leq i\leq N$, $w$ is a generalized eigenvector of $L_i(0)$ with eigenvalue $\lambda_i$. The finite-dimensional subspace $\Wbb_{[\leq\lambda_\blt]}$ is defined to be the direct sum of all $\Wbb_{[\mu_\blt]}$ where $\Re(\mu_i)\leq \Re(\lambda_i)$ for all $1\leq i\leq N$. Then the contragredient $\Vbb^{\otimes N}$-module of $\Wbb$, as a vector space, is
\begin{align*}
\Wbb'=\bigoplus_{\lambda_\blt\in\Cbb^N}(\Wbb_{[\lambda_\blt]})^*
\end{align*}
Then for each $w\in\Wbb,w'\in\Wbb$ we clearly have
\begin{align}\label{eq3}
\bk{Y_i(v,z)w,w'}=\bk{w,Y_i'(v,z)w'}
\end{align}
The algebraic completion of $\Wbb$ is 
\begin{align*}
\ovl\Wbb=(\Wbb')^*=\prod_{\lambda_\blt\in\Cbb^N}\Wbb_{[\lambda_\blt]}
\end{align*}
We let 
\begin{gather*}
P(\lambda_\blt)=\text{the projection of $\ovl\Wbb$ onto $\Wbb_{[\lambda_\blt]}$}\\
P({\leq\lambda_\blt})=\text{the projection of $\ovl\Wbb$ onto $\Wbb_{[\leq\lambda_\blt]}$}
\end{gather*}
Fix $1\leq i\leq N$ and $\lambda\in \Cbb$, then 
\begin{gather}\label{eq115}
\begin{gathered}
\text{$P_i(\lambda)$ resp. $P_i({\leq\lambda})$ denotes the projection of $\ovl{\Wbb}$ onto}\\
	\bigoplus_{\mu_\blt\in\Cbb^N,\mu_i=\lambda}\Wbb_{[\mu_\blt]}\qquad\text{resp.}\qquad \bigoplus_{\mu_\blt\in\Cbb^N,\Re(\mu_i)\leq \Re(\lambda)}\Wbb_{[\mu_\blt]}
\end{gathered}
\end{gather}
If $N=2$, we write
\begin{align}\label{eq130}
P_+(\lambda)=P_1(\lambda)\quad P_+({\leq\lambda})=P_1({\leq\lambda}) \quad P_-(\lambda)=P_2(\lambda)\quad P_-({\leq\lambda})=P_2({\leq\lambda})
\end{align}
\item Let $E$ be a finite set such that $|E|=N$. An \textbf{ordering} of $E$ is a bijection $\eps:\{1,\cdots,N\}\rightarrow E$. Suppose we have two orderings
\begin{align*}
	\eps:\{1,\cdots,N\}\rightarrow E,\quad \tipae:\{1,\cdots, M\}\rightarrow F
\end{align*}
The \textbf{composition of orderings} is defined as 
\begin{gather*}
	\eps*\tipae:\{1,\cdots,N+M\}\rightarrow E\sqcup F\\
	\eps*\tipae(i)=\eps(i),1\leq i\leq N\qquad \eps*\tipae(N+j)=\tipae(j),1\leq j \leq M
\end{gather*}
Then $\eps*\tipae$ is an ordering of $E\sqcup F$. It is easy to check that composition of orderings satisfies the associativity. Thus, the composition of $l$ orderings $\eps_1*\eps_2*\cdots *\eps_l$ is well-defined for $\eps_i:\{1,\cdots,N_i\}\rightarrow E_i$.
\item Let $N\in \Zbb_+$. The permutation group of $\{1,\cdots,N\}$ is denoted as $\fk S_N$.
\item Let $\Mbb\in \Mod(\Vbb)$. Let $\End(\Mbb)$ be the set of linear operators on $\Mbb$. Let
\begin{align}\label{eq87}
\begin{aligned}
	&\End^0(\Mbb):=\bigcup_{\lambda\in\Cbb} \End(\Mbb_{[\leq\lambda]})\\
=&\{T\in\End(\Mbb):T=P(\leq\lambda)TP(\leq\lambda)\text{ for some }\lambda\in\Cbb\}
\end{aligned}
\end{align}
$\End^0(\Mbb)$ is a grading-restricted $\Vbb^{\otimes 2}$-module whose module structure is determined by the fact that for each $v\in\Vbb,T\in\End^0(\Mbb)$, the following relation holds in $\End^0(\Mbb)[[z^{\pm1}]]$:
\begin{align*}
Y(v\otimes \idt,z) T=Y_\Mbb(v,z)\circ T\qquad Y(\idt\otimes v,z)T=T\circ Y_\Mbb (\mc U(\upgamma_z)v,z^{-1})
\end{align*}
Under this structure, the linear isomorphism
\begin{align*}
\Mbb\otimes\Mbb'\xlongrightarrow{\simeq}\End^0(\Mbb) \qquad m\otimes m'\mapsto m\cdot \bk{m',-}
\end{align*}
is an isomorphism in $\Mod(\Vbb^{\otimes 2})$.
\item Let $\Mbb\in\Mod(\Vbb)$. If $\Mbb$ is a right module over an associative $\Cbb$-algebra $B$, we let
\begin{align}\label{eq74}
\End_B^0(\Mbb)=\{T\in\End^0(\Mbb):(Tm)b=T(mb)\text{for each $m\in\Mbb$ and $b\in B$}\}
\end{align}
\item If $A$ is an associative $\Cbb$-algebra, we let 
\begin{align*}
\SLF(A)=\{\text{symmetric linear functionals on $A$}\}
\end{align*}
where a \textbf{symmetric linear functional} on $A$ denotes a linear map $\phi:A\rightarrow\Cbb$ satisfying $\phi(ab)=\phi(ba)$ for all $a,b\in A$.
\end{itemize}

\subsection{Conformal blocks for unordered $N$-pointed compact Riemann surfaces}\label{lb77}

In this section, we introduce the (unordered) $N$-pointed compact Riemann surfaces with local coordinates. 

\subsubsection{Sheaf of VOA}\label{lb52}

Let us recall the definition of sheaf of VOA. See \cite[Subsec. 1.3.1]{GZ2} for details.

Let $\mathcal G$ be the group of all $f(z)=\sum_{n>0}a_n z^n$, where $a_n\in \Cbb$ and $a_1\ne 0$. The group product of $f,g\in\MG$ is defined to be the composition $f\circ g$. For each $\alpha\in\MG$, $\MU(\alpha)$ is an invertible linear operator on $\Vbb$ defined by
	 \begin{align*}
\MU(\alpha)=\alpha'(0)^{L(0)}\exp\Big(\sum_{n>0}c_n L(n)\Big)
	 \end{align*}
where $c_n\in\Cbb$ and $0\neq\alpha'(0)\in\Cbb$ are the constants defined by $\alpha(z)=\alpha'(0)\cdot \exp\big(\sum_{n>0}c_n z^{n+1}\partial_z\big)z$.

More generally, if $\Wbb\in\Mod(\Vbb^{\otimes N})$, for each $1\leq i\leq N$ we define a linear operator
\begin{align}\label{eq93}
\MU_i(\alpha)=\alpha'(0)^{L_i(0)}\exp\Big(\sum_{n>0}c_n L_i(n)\Big)
\end{align}
on $\Wbb$, which depends on the choice of $\arg\alpha'(0)$.

If $X$ is a complex manifold, a map $\rho:X\rightarrow\MG$ (sending $x\in X$ to $\rho_x\in\MG$) is called a \textbf{holomorphic family of transformations} if for each $x\in X$, there exists a neighborhood $V\subset X$ of $x$ and a neighborhood $U\subset \Cbb$ of $0$ such that the map $(z,y)\in U\times V\mapsto\rho_y(z)\in\Cbb$ is holomorphic. In that case, for each $n\in\Nbb$, we have an $\End(\Vbb^{\leq n})$-valued holomorphic map
\begin{align*}
\MU(\rho):X\rightarrow \End(\Vbb^{\leq n})\qquad x\mapsto\MU(\rho_x)
\end{align*}
which induces an isomorphism of holomorphic vector bundles
\begin{align*}
\MU(\rho):\Vbb^{\leq n}\otimes_\Cbb\MO_X\xlongrightarrow{\simeq}\Vbb^{\leq n}\otimes_\Cbb\MO_X
\end{align*}

Let $C$ be a Riemann surface (without boundary). The \textbf{sheaf of VOA} 
	 \begin{align*}
		\scr V_C:=\varinjlim_{n\in\Nbb}\scr V_C^{\leq n}=\bigcup_{n\in\Nbb}\scr V_C^{\leq n}
	 \end{align*}
which relies on $\Vbb$ and $C$, is defined as follows.  $\SV_C^{\leq n}$ is a (finite rank) locally free $\MO_C$-module defined by the transition functions provided below. For each open subset $U\subset C$ and each univalent (i.e. holomorphic injective) function $\eta\in \MO(U)$, we have a trivialization, i.e., an isomorphism of holomorphic vector bundles
	 \begin{align*}
		\MU_\varrho(\eta):\SV_C^{\leq n}|_U\xrightarrow{\simeq} \Vbb^{\leq n}\otimes \MO_U
	 \end{align*}
	 If $\mu\in \MO(U)$ is another univalent function, the \textbf{transition function} is given by the isomorphism
	 \begin{align}\label{eq109}
		\MU_\varrho(\eta)\MU_\varrho(\mu)^{-1}=\MU(\varrho(\eta|\mu)):\Vbb^{\leq n}\otimes_\Cbb \MO_U \xrightarrow{\simeq} \Vbb^{\leq n}\otimes_\Cbb \MO_U
	 \end{align}
	 Here $\varrho(\eta|\mu):U\rightarrow \mathcal G$ is the holomorphic family such that for each $p\in U$, 
	 \begin{align*}
    \varrho(\eta|\mu)_p(z)=\eta\circ\mu^{-1}\big(z+\mu(p)\big)-\eta(p)
	 \end{align*}
Equivalently, $\varrho(\eta|\mu)_p$ is the unique element of $\MG$ transforming the local coordinate $\mu-\mu(p)$ at $p$ to $\eta-\eta(p)$, i.e.,
\begin{align*}
\varrho(\eta|\mu)_p\circ(\mu-\mu(p))=\eta-\eta(p)
\end{align*}

Recall that $\omega_C$ is the sheaf of germs of holomorphic 1-forms on $C$. We view $\omega_C$ as a holomorphic line bundle. Consider the tensor product bundle $\SV_C^{\leq n}\otimes\omega_C$, formed from $\SV_C^{\leq n}$ and $\omega_C$. Then for each univalent $\eta\in\MO(U)$ as above, we have a trivialization
\begin{align*}
\MU_\varrho(\eta)\otimes \id:\SV_C^{\leq n}\otimes\omega_C|_U\xrightarrow{\simeq} \Vbb^{\leq n}\otimes \omega_U
\end{align*}
abbreviated to $\MU_\varrho(\eta)$ for convenience.

\subsubsection{Definition of conformal blocks}

Let $N\in\Nbb$.

\begin{df}\label{lb34}
An  \textbf{(unordered) \pmb{$N$}-pointed compact Riemann surface with local coordinates} denotes the data
	\begin{align}\label{eq4}
		\fx=(C\big|E;\eta_E)=(C\big| E;(\eta_x)_{x\in E}),
	 \end{align}
	 where $C$ is a compact Riemann surface; $E\subset C$ is a set of marked points such that $|E|=N$; for each $x\in E$, $\eta_x$ is a \textbf{local coordinate at $x$}---that is, $\eta_x$ is an injective holomorphic function on a neighborhood $U_x$ of $x$ satisfying $\eta_x(x)=0$. We also assume that
\begin{align}\label{eq113}
\text{the intersection of $E$ with each connected component of $C$ is non-empty.}
\end{align}
An \textbf{ordering} of $E$ is defined to be a bijection
	 \begin{align}\label{eq5}
		\eps:\{1,2,\cdots,N\}\xlongrightarrow{\simeq} E
	 \end{align}
Given an ordering $\eps$, the set $E$ can be written as $E=\{\eps(1),\cdots,\eps(N)\}$. 
\end{df}

Consider the sheaf $\SV_\fx\otimes \omega_C(\blt E)$, where
\begin{align*}
\SV_\fx:=\SV_C
\end{align*}
The sheaf $\SV_\fx\otimes \omega_C(\blt E)$ consists of sections of $\SV_\fx\otimes \omega_C\big|_{C\setminus E}$ with finite poles at $E$. Then $H^0\big(C,\SV_\fx\otimes \omega_C(\blt E)\big)$ is the space of global sections of $\SV_\fx\otimes \omega_C$ on $C\setminus E$ with finite poles at $E$.

Let $\eps$ be an ordering of $E$. Let $\Wbb\in \Mod(\Vbb^{\otimes N})$. We will use the terminology of \textbf{associating \pmb{$\Wbb$} to \pmb{$E$} via \pmb{$\eps$}} to describe the assignment of $\Vbb^{\otimes N}$-modules to marked points. This terminology allows us to define the \textbf{\pmb{$\eps$}-residue action $\pmb{*^\eps}$} of $H^0\big(C,\SV_\fx\otimes \omega_C(\blt  E)\big)$ on $\Wbb$, which is a linear action defined by
\begin{subequations}\label{eq6}
	 \begin{gather}
		\sigma*^\eps_i w=\Res_{\eps(i)}~Y_i(\MU_\varrho(\eta_{\eps(i)})\sigma,\eta_{\eps(i)})w\label{eq6a}\\
        \sigma*^\eps w=\sum_{i=1}^N \sigma*^\eps_i w
	 \end{gather}
\end{subequations}
for each $\sigma\in H^0\big(C,\SV_\fx\otimes \omega_C(\blt  E)\big)$, $w\in \Wbb$, and $1\leq i\leq N$. 

Eq. \eqref{eq6a} is interpreted in the following way: Let $U_{\eps(i)}$ be a connected neighborhood of $\eps(i)$ on which $\eta_{\eps(i)}$ is defined, and let $U^\times_{\eps(i)}=U_{\eps(i)}\setminus\{\eps(i)\}$. Then the trivialization
\begin{align*}
\MU_\varrho(\eta_{\eps(i)}):\SV_\fx\otimes \omega_C\big|_{U^\times_{\eps(i)}}\xlongrightarrow{\simeq}\Vbb\otimes_\Cbb\omega_{U^\times_{\eps(i)}}=\Vbb\otimes_\Cbb\MO_{U^\times_{\eps(i)}}d\eta_{\eps(i)}
\end{align*}
maps $\sigma$ (more precisely, the restriction $\sigma|_{U^\times_{\eps(i)}}$) to a finite sum
\begin{align*}
\sum_k v_k\otimes f_k d\eta_{\eps(i)}
\end{align*}
where $v_k\in\Vbb$, and $f_k=f_k(\eta_{\eps(i)})\in\MO(U^\times_{\eps(i)})$ admits a Laurent series expansion at $\eps(i)$:
\begin{align*}
f_k=\sum_{n\in\Zbb}a_{k,n}(\eta_{\eps(i)})^n
\end{align*} 
with $a_{k,n}\in\Cbb$, and $a_{k,n}=0$ for $n\ll0$. Then
\begin{align}\label{eq7}
\sigma*^\eps_i w=\sum_k\Res_{\eps(i)}~Y_i(v_k,\eta_{\eps(i)})w\cdot f_kd\eta_{\eps(i)}=\sum_k\sum_{n\in\Zbb}a_{k,n}Y_i(v_k)_nw
\end{align}

The space of coinvariants is given by 
	  \begin{align*}
		\ST_{\fx,\eps}(\Wbb)=\frac{\Wbb}{H^0\big(C,\SV_\fx\otimes \omega_C(\blt E)\big)*^\eps \Wbb}
	  \end{align*}
where $\Span_\Cbb$ has been omited in the denominator. Its dual space is called the \textbf{space of conformal blocks} and is denoted by \pmb{$\ST^*_{\fx,\eps}(\Wbb)$}. The elements of $\ST^*_{\fx,\eps}(\Wbb)$ are called \textbf{conformal blocks} associated to the family $\fx$, the module $\Wbb$, and the ordering $\eps$. In other words:
\begin{itemize}
\item The space $\ST^*_{\fx,\eps}(\Wbb)$ consists of all conformal blocks, i.e., linear functionals $\upphi:\Wbb\rightarrow\Cbb$ satisfying
\begin{align*}
\bk{\upphi,\sigma*^\eps w}=0
\end{align*}
for each $w\in\Wbb,\sigma\in H^0(C,\SV_\fx\otimes\omega_C(\blt E))$.
\end{itemize}

\begin{rem}
Let $\Wbb\in \Mod(\Vbb^{\otimes N})$. A pictorial illustration of $\ST_{\fx,\eps}^*(\Wbb)$ is 
\begin{align*}
	\ST^*\bigg(~\vcenter{\hbox{{
		 \includegraphics[height=2.2cm]{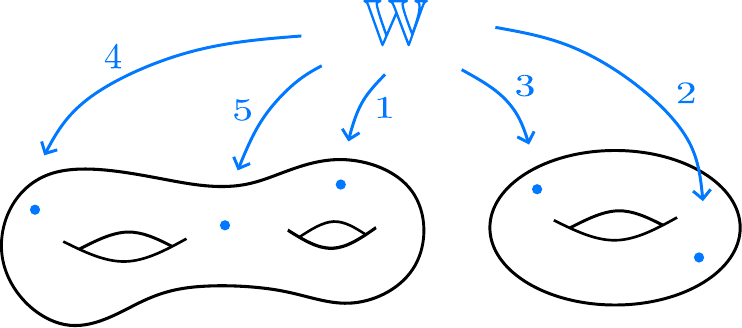}}}}~~\bigg)
 \end{align*}
In this picture, $\fx$ has two connected components, and $N=5$. Thus, the set $E$ of marked points has five elements, corresponding to $\eps(4),\eps(5),\eps(1),\eps(3),\eps(2)$, listed from left to right. (Note that each marked point $\eps(i)$ is equipped with a local coordinate $\eta_{\eps(i)}$.) And $\Mbb\in\Mod(\Vbb^{\otimes 5})$.
\end{rem} 

\begin{rem}
The setting of conformal blocks laid out in this article is slightly different from that of \cite{GZ1,GZ2,GZ3}, where we study conformal blocks associated to $\Wbb\in\Vbb^{\otimes N}$ and an \textbf{\textit{ordered} $N$-pointed compact Riemann surface with local coordinates}
\begin{align*}
\fy=(C|x_1,\dots,x_N;\eta_1,\dots,\eta_N)
\end{align*}
The definition used there agrees with the present definition of $\ST^*_{\fx,\eps}(\Wbb)$ provided that, given the above $\fy$, we define  the unordered data
\begin{align*}
\fx=(C|E;(\eta_x)_{x\in E})\qquad\text{where }E=\{x_1,\dots,x_N\}\text{ and }\eta_{x_i}=\eta_i
\end{align*}
and define the ordering $\eps$ by $\eps(i)=x_i$. Conversely, given the unordered data $\fx$ and the ordering $\eps$ as in \eqref{eq4} and \eqref{eq5}, we define the ordered data
\begin{align*}
\fy=(C|\eps(1),\dots,\eps(N);\eta_{\eps(1)},\dots,\eta_{\eps(N)})
\end{align*}
\end{rem}

\begin{rem}\label{lb38}
For each $\Mbb,\Wbb\in\Mod(\Vbb^{\otimes N})$ and $T\in\Hom_{\Vbb^{\otimes N}}(\Mbb,\Wbb)$, the map $T:\Mbb\rightarrow\Wbb$ clearly descends to a linear map of the corresponding quotient spaces
\begin{align}
T:\ST_{\fx,\eps}(\Mbb)\longrightarrow\ST_{\fx,\eps}(\Wbb)
\end{align}
Its transpose is
\begin{align}
T^\tr:\ST^*_{\fx,\eps}(\Wbb)\longrightarrow\ST^*_{\fx,\eps}(\Mbb)
\end{align}
Therefore, we have a contravariant functor
\begin{gather*}
\ST^*_{\fx,\eps}:\Mod(\Vbb^{\otimes N})\longrightarrow \Vect\\
\Wbb\longmapsto \ST^*_{\fx,\eps}(\Wbb)\\
T\in\Hom_{\Vbb^{\otimes N}}(\Mbb,\Wbb)\longmapsto\big(T^\tr:\ST^*_{\fx,\eps}(\Wbb)\rightarrow\ST^*_{\fx,\eps}(\Mbb)\big)
\end{gather*}
called the \textbf{conformal block functor}. In fact, this contravariant functor is left exact; see \cite[Thm. 1.22]{GZ3}.
\end{rem}

 \subsubsection{Grouping marked points}

Let $\fx$ be as in Def. \ref{lb34}. In practice, the set $E$ of marked points of \eqref{eq4} is often divided into $T$ subsets (where $T\in\Zbb_+$), written as $E=E_1\sqcup\cdots\sqcup E_T$. In this case, for each $i=1,\dots,T$, we let
\begin{align}
\eta_{E_i}=(\eta_x)_{x\in E_i}
\end{align}
so that $\fx$ can be written as
 \begin{align}
	\fx=(C\big|E_1\sqcup \cdots \sqcup E_T;\eta_{E_1},\cdots,\eta_{E_T})
 \end{align}

For each $1\leq i\leq T$, set $N_i=|E_i|$, and choose an ordering
\begin{align*}
	\eps_i:\{1,\cdots,N_i\}\xlongrightarrow{\simeq} E_i
\end{align*}
Choose $\Wbb_i\in \Mod(\Vbb^{\otimes N_i})$, and associate $\Wbb_i$ to $E_i$ via $\eps_i$.

Recall that $\eps_1*\cdots*\eps_T$ denotes the composition of the orderings $\eps_1,\dots,\eps_T$, cf. Sec. \ref{lb32}.

\begin{pp}\label{lb10}
For each $\alpha\in \fk S_T$, we associate $\Wbb_{\alpha(1)}\otimes \cdots \otimes \Wbb_{\alpha(T)}\in \Mod(\Vbb^{\otimes N})$ to $E$ via the ordering $\eps_{\alpha(1)}*\cdots*\eps_{\alpha(T)}$. Then the linear isomorphism
\begin{gather*}
\pi_\alpha:\Wbb_1\otimes \cdots \otimes \Wbb_T\rightarrow \Wbb_{\alpha(1)}\otimes \cdots \otimes \Wbb_{\alpha(T)}\\
w_1\otimes \cdots \otimes w_T\mapsto w_{\alpha(1)}\otimes \cdots \otimes w_{\alpha(T)}
\end{gather*}
descends to a linear isomorphism
\begin{align*}
\pi_\alpha:\ST_{\fx,\eps_{1}*\cdots*\eps_{T}}(\Wbb_{1}\otimes \cdots \otimes \Wbb_{T})\xlongrightarrow{\simeq}\ST_{\fx,\eps_{\alpha(1)}*\cdots*\eps_{\alpha(T)}}(\Wbb_{\alpha(1)}\otimes \cdots \otimes \Wbb_{\alpha(T)})
\end{align*}
Therefore, its transpose gives a linear isomorphism
	\begin{align*}
\pi_\alpha^\tr:		\ST_{\fx,\eps_{\alpha(1)}*\cdots*\eps_{\alpha(T)}}^*(\Wbb_{\alpha(1)}\otimes \cdots \otimes \Wbb_{\alpha(T)})\xlongrightarrow{\simeq} \ST_{\fx,\eps_{1}*\cdots*\eps_{T}}^*(\Wbb_{1}\otimes \cdots \otimes \Wbb_{T})
	\end{align*}
\end{pp}

\begin{proof}
By \eqref{eq6} and \eqref{eq7}, for each $\sigma\in H^0\big(C,\SV_\fx\otimes \omega_C(\blt E)\big)$ we have
\begin{align*}
\pi_\alpha(\sigma *^{\eps_1*\cdots*\eps_T}(w_1\otimes\cdots\otimes w_T))=\sigma *^{\eps_{\alpha(1)}*\cdots*\eps_{\alpha(T)}}(w_{\alpha(1)}\otimes\cdots\otimes w_{\alpha(T)})
\end{align*}
for each $w_1\in\Wbb_1,\dots,w_T\in\Wbb_T$, and hence 
\begin{align}
\pi_\alpha(\sigma *^{\eps_1*\cdots*\eps_T}w)=\sigma *^{\eps_{\alpha(1)}*\cdots*\eps_{\alpha(T)}}\pi_\alpha(w)
\end{align}
for each $w\in\Wbb_1\otimes\cdots\otimes\Wbb_T$. Therefore
\begin{align*}
&\pi_\alpha\Big(H^0\big(C,\SV_\fx\otimes \omega_C(\blt E)\big)*^{\eps_1*\cdots*\eps_T} (\Wbb_1\otimes\cdots\otimes\Wbb_T)\Big)\\
=&H^0\big(C,\SV_\fx\otimes \omega_C(\blt E)\big)*^{\eps_{\alpha(1)}*\cdots*\eps_{\alpha(T)}} (\Wbb_{\alpha(1)}\otimes \cdots \otimes \Wbb_{\alpha(T)})
\end{align*}
This proves that $\pi_\alpha$ descends to a bijective linear map between spaces of coinvariants.
\end{proof}

\begin{df}\label{lb35}
By Prop. \ref{lb10}, for each $\alpha\in\fk S_T$, we can identify the spaces of conformal blocks $\ST_{\fx,\eps_{\alpha(1)}*\cdots*\eps_{\alpha(T)}}^*(\Wbb_{\alpha(1)}\otimes \cdots \otimes \Wbb_{\alpha(T)})$ and $\ST_{\fx,\eps_{1}*\cdots*\eps_{T}}^*(\Wbb_{1}\otimes \cdots \otimes \Wbb_{T})$ via $\pi_\alpha^\tr$. We denote this identified space by
\begin{align*}
\pmb{\ST_{\fx,\eps_\blt}^*(\Wbb_\blt)}\equiv\ST^*_{\fx,\eps_1,\dots,\eps_T}(\Wbb_1,\dots,\Wbb_T)\equiv\ST^*_{\fx,\eps_{\alpha(1)},\dots,\eps_{\alpha(T)}}(\Wbb_{\alpha(1)},\dots,\Wbb_{\alpha(T)})
\end{align*}
and call it the \textbf{space of conformal blocks} associated to $\fx$ and $\Wbb_1,\dots,\Wbb_T$ via $\eps_1,\dots,\eps_T$.
\end{df}

\begin{rem}\label{lb33}
A pictorial illustration of $\ST^*_{\fx,\eps_\blt}(\Wbb_\blt)$ is
\begin{align*}
\ST^*_{\fx,\eps_\blt}(\Wbb_\blt)=\ST^*\bigg(\vcenter{\hbox{{
		 \includegraphics[height=2.7cm]{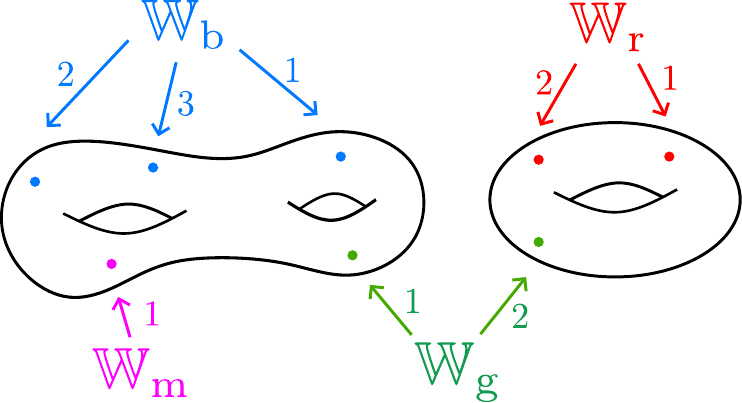}}}}~  \bigg)
\end{align*}
In this example, $\fx$ has two components, with a total of $N=8$ marked points, divided into $T=4$ subsets. The set of marked points is $E=E_{\mathrm b}\sqcup E_{\mathrm r}\sqcup E_{\mathrm m}\sqcup E_{\mathrm g}$ where the subscripts $\mathrm b,\mathrm r,\mathrm m,\mathrm g$ stand for blue, red, magenta, and green, respectively. Let $\eps_{\mathrm b},\eps_{\mathrm r},\eps_{\mathrm m},\eps_{\mathrm g}$ be the corresponding orderings.
\begin{itemize}
\item $E_{\mathrm b}$ consists of the top three blue points on the left component, ordered from left to right as $\eps_{\mathrm b}(2),\eps_{\mathrm b}(3),\eps_{\mathrm b}(1)$.
\item  $E_{\mathrm r}$ consists of the top two red points on the right component, ordered from left to right as $\eps_{\mathrm r}(2),\eps_{\mathrm r}(1)$.
\item $E_{\mathrm m}$ consists of the single magenta point, $\eps_{\mathrm m}(1)$, located at the bottom of the left component.
\item $E_{\mathrm g}$ consists of the bottom two green points: $\eps_{\mathrm g}(1)$ lies on the left component, and $\eps_{\mathrm g}(2)$ on the right.
\end{itemize}
The modules $\Wbb_{\mathrm b}\in\Mod(\Vbb^{\otimes 3})$, $\Wbb_{\mathrm r}\in\Mod(\Vbb^{\otimes 2})$, $\Wbb_{\mathrm m}\in\Mod(\Vbb)$, and $\Wbb_{\mathrm g}\in\Mod(\Vbb^{\otimes 2})$ are associated to $E_{\mathrm b},E_{\mathrm r},E_{\mathrm m},E_{\mathrm g}$ via the orderings $\eps_{\mathrm b},\eps_{\mathrm r},\eps_{\mathrm m},\eps_{\mathrm g}$ respectively. 
\end{rem}

\begin{rem}\label{lb84}
When $|E_i|=1$, we typically omit the index of the arrow in the picture. When $|E_i|=2$, we 
\begin{subequations}
\begin{align}
\text{write $1$ and $2$ as $+$ and $-$, respectively}
\end{align} 
Accordingly, we
\begin{align}
\text{write $Y_1$ and $Y_2$ as $Y_+$ and $Y_-$, respectively, for $\Vbb^{\otimes 2}$-modules}
\end{align}
\end{subequations}
Moreover, for any $E_i$, reversing the directions of the arrows for $\Wbb_i$ and simultaneously replacing $\Wbb_i$ with its contragredient $\Wbb_i'$ represent the same space of conformal blocks. For example, the space $\ST^*_{\fx,\eps_\blt}(\Wbb_\blt)$ in Rem. \ref{lb33} can be represented by
\begin{align*}
&\ST^*_{\fx,\eps_\blt}(\Wbb_\blt)=\ST^*\bigg(\vcenter{\hbox{{
		 \includegraphics[height=2.7cm]{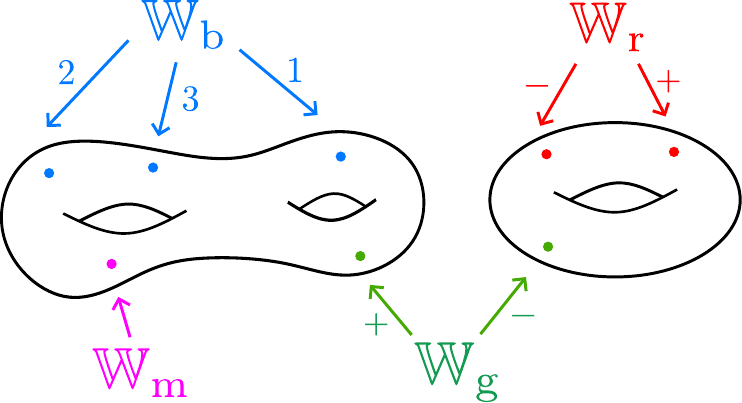}}}}~  \bigg)\\[0.5ex]
=&\ST^*\bigg(\vcenter{\hbox{{
		 \includegraphics[height=2.7cm]{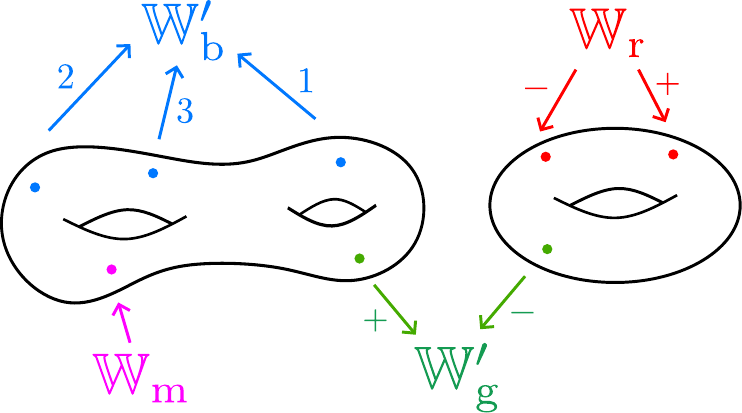}}}}~  \bigg)=\ST^*\bigg(\vcenter{\hbox{{
		 \includegraphics[height=2.7cm]{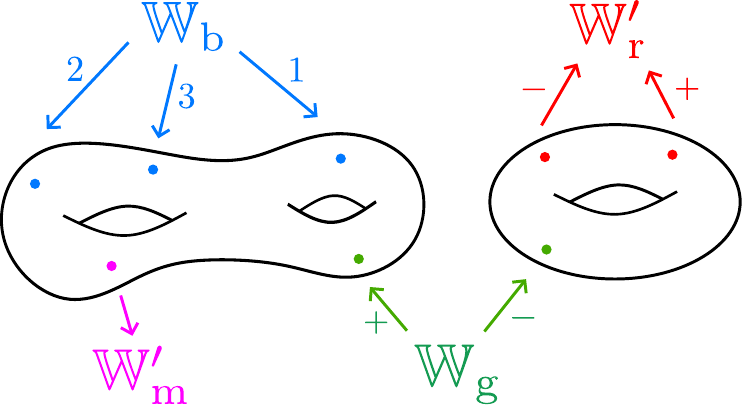}}}}~  \bigg)
\end{align*}
\end{rem}

	\subsection{(Dual) fusion products}

In this section, we let $N,R\in\Nbb$. In order to define (dual) fusion products, we divide the set of marked points $E$ into $T=2$ parts $E',E$, and place one of them---say $E'$---together with its local coordinates in front of $C$. More precisely:

\begin{df}\label{lb1}
Let $\ff$ be an unordered $(R+N)$-pointed compact Riemann surfaces with local coordinates, where the set of marked points is divided into two groups $E'$ and $E$. We let $\theta_x$ denote the local coordinate at $x\in E$, and let $\theta'_{x'}$ denote the local coordinate at $x'\in E'$. We write this data as
\begin{align}\label{eq13}
\fk F=\big(E'; \theta'_{E'}\big|C\big|E;\theta_E\big)\qquad\text{where }\theta'_{E'}=(\theta'_{x'})_{x'\in E'}\text{ and }\theta_E=(\theta_x)_{x\in E}
\end{align}
and call it an \textbf{unordered \pmb{$(R,N)$}-pointed compact Riemann surface with local coordinates}. We call $E$ the \textbf{incoming marked points} (or simply \textbf{inputs}) of $\fk F$, and we call $E'$ the \textbf{outgoing marked points} (or simply \textbf{outputs}) of $\fk F$.
\end{df}

Note that according to Def. \ref{lb34}, we have:
\begin{align}\label{eq8}
\text{Each connected component of $C$ intersects $E\cup E'$}.
\end{align}

When we view $\ff$ as an $(R+N)$-pointed compact Riemann surface with local coordinate, we write it as
\begin{align}\label{eq14}
	\fk F=\big(C\big|E'\sqcup E;\theta'_{E'},\theta_E).
	\end{align}

	%If we choose orderings $\eps:\{1,\cdots,N\}\rightarrow E,\eps':\{1,\cdots ,R\}\rightarrow E'$, then $\eps'*\eps$ and $\eps*\eps'$ are both orderings of $E'\cup E$.

\begin{df}\label{lb2}
	Let $\ff=\eqref{eq13}$ be as in Def. \ref{lb1}. Choose orderings $\eps:\{1,\cdots,N\}\rightarrow E$ and $\eps':\{1,\cdots ,R\}\rightarrow E'$. Associate $\Wbb\in\Mod(\Vbb^{\otimes N})$ to $E$ via $\eps$. An \textbf{\pmb{$(\eps',\eps)$}-dual fusion product} of $\Wbb$ along $\fk F$ denotes a pair \pmb{$(\bbs_{\fk F,\eps',\eps} (\Wbb),\gimel)$}, where $\bbs_{\fk F,\eps',\eps} (\Wbb)\in\Mod(\Vbb^{\otimes R})$ is associated to $E'$ via $\eps'$,
\begin{align*}
\gimel\in\ST_{\fk F,\eps',\eps}^*(\bbs_{\fk F,\eps',\eps} (\Wbb),\Wbb)
\end{align*}
where $\ST_{\fk F,\eps',\eps}^*(\bbs_{\fk F,\eps',\eps} (\Wbb),\Wbb)$ is the space of conformal blocks associated to $\ff$ and $\bbs_{\fk F,\eps',\eps} (\Wbb),\Wbb$ via $\eps',\eps$ (cf. Def. \ref{lb35}), and the following \textbf{universal property} is satisfied:
\begin{itemize}
\item For each $\Mbb\in\Mod(\Vbb^{\otimes R})$ associated to $E'$ via $\eps'$, the map
\begin{align}\label{eq15}
\Hom_{\Vbb^{\otimes R}}(\Mbb,\bbs_{\fk F,\eps',\eps} (\Wbb))\rightarrow \scr T_{\fk F,\eps',\eps}^*(\Mbb,\Wbb)\qquad T\mapsto \gimel\circ(T\otimes\id_\Wbb)
\end{align}
is a linear isomorphism. 
\end{itemize}
We call $\gimel$ the \textbf{canonical conformal block}. We abbreviate $(\bbs_{\fk F,\eps',\eps} (\Wbb),\gimel)$ to $\bbs_{\fk F,\eps',\eps}(\Wbb)$ when no confusion arises. 

The contragredient $\Vbb^{\otimes R}$-module of $\bbs_{\fk F,\eps',\eps} (\Wbb)$ is denoted by \pmb{$\boxtimes_{\fk F,\eps',\eps}(\Wbb)$}, that is,
\begin{align*}
\boxtimes_{\fk F,\eps',\eps}(\Wbb):=\bbs_{\fk F,\eps',\eps}(\Wbb)'
\end{align*}
We call the pair $(\boxtimes_{\fk F,\eps',\eps} (\Wbb),\gimel)$---or simply $\boxtimes_{\fk F,\eps',\eps} (\Wbb)$---an \textbf{\pmb{$(\eps',\eps)$}-fusion product} of $\Wbb$ along $\ff$.  \hqed
\end{df}

\begin{rem}
In Def. \ref{lb2}, we are viewing
\begin{gather*}
\ST_{\fk F,\eps',\eps}^*(\bbs_{\fk F,\eps',\eps}(\Wbb),\Wbb)=\ST_{\fk F,\eps'*\eps}^*(\bbs_{\fk F,\eps',\eps} (\Wbb)\otimes\Wbb)\\
\scr T_{\fk F,\eps',\eps}^*(\Mbb,\Wbb)=\scr T_{\fk F,\eps'*\eps}^*(\Mbb\otimes\Wbb)
\end{gather*}
cf. Def. \ref{lb35}. Then $\gimel$ is a linear functional
\begin{align*}
\gimel:\bbs_{\fk F,\eps',\eps}(\Wbb)\otimes\Wbb\rightarrow\Cbb
\end{align*}
By Def. \ref{lb35}, one can also view
\begin{gather*}
\ST_{\fk F,\eps',\eps}^*(\bbs_{\fk F,\eps',\eps}(\Wbb),\Wbb)=\ST_{\fk F,\eps*\eps'}^*(\Wbb\otimes\bbs_{\fk F,\eps',\eps} (\Wbb))\\
\scr T_{\fk F,\eps',\eps}^*(\Mbb,\Wbb)=\scr T_{\fk F,\eps*\eps'}^*(\Wbb\otimes\Mbb)
\end{gather*}
In that case, $\gimel$ is a linear functional
\begin{align*}
\gimel:\Wbb\otimes\bbs_{\fk F,\eps',\eps}(\Wbb)\rightarrow\Cbb
\end{align*}
and the expression $ T\mapsto \gimel\circ(T\otimes\id_\Wbb)$ in \eqref{eq15} should be replaced by
\begin{align*}
T\mapsto\gimel\circ(\id_\Wbb\otimes T)
\end{align*}
\end{rem}

\begin{rem}\label{lb39}
When $E=\emptyset$ (and hence $N=0$), Def. \ref{lb2} is interpreted in the following way. Since
\begin{align*}
\Mod(\Vbb^{\otimes 0})=\Vect
\end{align*}
It suffices to consider the scalar field $\Cbb\in\Mod(\Vbb^{\otimes 0})$. Then, an $\eps'$-dual fusion product of $\Cbb$ along $\ff$ denotes a pair $(\bbs_{\ff,\eps'}(\Cbb),\gimel)$, where $\bbs_{\ff,\eps'}(\Cbb)\in\Mod(\Vbb^{\otimes R})$, and
\begin{align*}
\gimel\in\ST^*_{\ff,\eps'}\big(\bbs_{\ff,\eps'}(\Cbb)\big)
\end{align*}
satisfies the following universal property that for each $\Mbb\in\Mod(\Vbb^{\otimes R})$ associated to $E'$ via $\eps'$, the following map is a linear isomorphism:
\begin{align}
\Hom_{\Vbb^{\otimes R}}(\Mbb,\bbs_{\fk F,\eps'} (\Cbb))\rightarrow \scr T_{\fk F,\eps'}^*(\Mbb)\qquad T\mapsto \gimel\circ T
\end{align}
\end{rem}

In \cite{GZ1}, the existence of (dual) fusion products is established under the additional assumption that each component of $C$ intersects $E$ (not merely $E\cup E'$). This assumption, however, can be removed with the help of propagation of conformal blocks:

\begin{thm}\label{lb3}
Let $\ff=\eqref{eq13}$ be as in Def. \ref{lb1}. Choose orderings $\eps:\{1,\cdots,N\}\rightarrow E$ and $\eps':\{1,\cdots ,R\}\rightarrow E'$. Associate $\Wbb\in\Mod(\Vbb^{\otimes N})$ to $E$ via $\eps$. Then there exists an $(\eps',\eps)$-dual fusion products of $\Wbb$ along $\ff$.
\end{thm}

Note that by the universal property in Def. \ref{lb2}, $(\eps',\eps)$-dual fusion products are unique up to unique $\Vbb^{\otimes R}$-module isomorphisms. 

\begin{proof}
By enlarging the set of incoming marked points of $\ff$, we get an $(R,N+L)$-pointed compact Riemann surface with local coordinates
	\begin{align*}
		\wtd{\fk F}=\big(E'; \theta'_{E'}\big|C\big|E\sqcup \{z_1,\dots,z_L\};\theta_E,\theta_{z_1},\cdots \theta_{z_L}\big)
	\end{align*}
(where $z_1,\dots,z_L$ are distinct points of $C\setminus (E\cup E')$, and $\theta_{z_i}$ is an arbitrary local coordinate at $z_i$) such that
	\begin{align}\label{eq9}
		\text{each component of $C$ intersects $E\cup \{z_1,\cdots,z_L\}$}.
		\end{align}
Then $\wtd{\fk F}$ satisfies \cite[Asmp. 2.2]{GZ1}, so that the dual fusion product exists. More precisely, let 
		\begin{align*}
			\iota_L:\{1,\cdots,L\}\xlongrightarrow{\simeq} \{z_1,\cdots,z_L\}\qquad i\mapsto z_i
		\end{align*}
Associate $\Wbb\otimes \Vbb^{\otimes L}$ to $E\cup \{z_1,\cdots,z_L\}$ via $\eps*\iota_L$. By \cite[Thm. 3.31]{GZ1}, there exists an $(\eps',\eps*\iota_L)$-dual fusion product $\big(\bbs_{\wtd\ff,\eps',\eps*\iota_L} (\Wbb\otimes \Vbb^{\otimes L}),\wtd \gimel\big)$ of $\Wbb\otimes\Vbb^{\otimes L}$ along $\wtd\ff$. We view $\wtd\gimel$ as a linear functional
\begin{align*}
\wtd\gimel:\bbs_{\wtd\ff,\eps',\eps*\iota_L} (\Wbb\otimes \Vbb^{\otimes L})\otimes\Wbb\otimes\Vbb^{\otimes L}\rightarrow\Cbb
\end{align*}

By \cite[Cor. 2.44]{GZ1}, for each $\Mbb\in\Mod(\Vbb^{\otimes R})$ we have an linear isomorphism 
		\begin{align}\label{eq11}
			\scr T_{\wtd{\fk F},\eps'*\eps*\iota_L}^*(\Mbb\otimes\Wbb\otimes \Vbb^{\otimes L})\xlongrightarrow{\simeq} \scr T_{\fk F,\eps'*\eps}^*(\Mbb\otimes\Wbb)\qquad \uppsi\mapsto\uppsi(-\otimes \idt^{\otimes L})
		\end{align}
Therefore, the linear functional
\begin{align*}
\gimel:\bbs_{\wtd\ff,\eps',\eps*\iota_L} (\Wbb\otimes \Vbb^{\otimes L})\otimes\Wbb\rightarrow\Cbb \qquad \gimel(-)=\wtd\gimel(-\otimes\idt^{\otimes L})
\end{align*}
belongs to
\begin{align*}
\scr T_{\fk F,\eps',\eps}^*\big(\bbs_{\wtd\ff,\eps',\eps*\iota_L} (\Wbb\otimes \Vbb^{\otimes L}),\Wbb\big)=\scr T_{\fk F,\eps'*\eps}^*\big(\bbs_{\wtd\ff,\eps',\eps*\iota_L} (\Wbb\otimes \Vbb^{\otimes L})\otimes\Wbb\big)
\end{align*}
By the universal property of $\wtd\gimel$, 
\begin{equation}\label{eq10}
		\begin{gathered}
\Hom_{\Vbb^{\otimes R}}\big(\Mbb,\bbs_{\wtd{\fk F},\eps',\eps*\iota_L} (\Wbb\otimes \Vbb^{\otimes L})\big)\longrightarrow \scr T_{\wtd{\fk F},\eps'*\eps*\iota_L}^*\big(\Mbb\otimes\Wbb\otimes \Vbb^{\otimes L}\big)\\
			T\mapsto \wtd\gimel\circ(T\otimes\id_{\Wbb}\otimes \id_{\Vbb^{\otimes L}})
			\end{gathered}
		\end{equation}
is a linear isomorphism. Its composition with \eqref{eq11}, namely
\begin{gather*}
\Hom_{\Vbb^{\otimes R}}\big(\Mbb,\bbs_{\wtd{\fk F},\eps',\eps*\iota_L} (\Wbb\otimes \Vbb^{\otimes L})\big)\longrightarrow\scr T_{\fk F,\eps'*\eps}^*\big(\Mbb\otimes\Wbb\big)\\
T\mapsto\gimel\circ(T\otimes\id_\Wbb)
\end{gather*}
is also a linear isomorphism. This proves that $\big(\bbs_{\wtd\ff,\eps',\eps*\iota_L} (\Wbb\otimes \Vbb^{\otimes L}),\gimel\big)$ is an $(\eps',\eps)$-dual fusion products of $\Wbb$ along $\ff$.
\end{proof}

The proof of Thm. \ref{lb3} implies a result that is important enough to be stated separately.  Recall that $\idt\in\Vbb$ denotes the vacuum vector of $\Vbb$.

\begin{thm}\label{lb36}
Let $\ff=\eqref{eq13}$ be as in Def. \ref{lb1}. Choose orderings $\eps:\{1,\cdots,N\}\rightarrow E$ and $\eps':\{1,\cdots ,R\}\rightarrow E'$. Associate $\Wbb\in\Mod(\Vbb^{\otimes N})$ to $E$ via $\eps$. Let $(\bbs_{\ff,\eps',\eps}(\Wbb),\gimel)$ be an $(\eps',\eps)$-dual fusion product of $\Wbb$ along $\ff$.

Choose distinct points $z_1,\dots,z_L\in C\setminus(E\cup E')$ and local coordinates $\theta_{z_1},\dots,\theta_{z_L}$ at these points. Let
\begin{align}\label{eq111}
\wtd{\fk F}=\big(E'; \theta'_{E'}\big|C\big|E\sqcup \{z_1,\dots,z_L\};\theta_E,\theta_{z_1},\cdots \theta_{z_L}\big)
\end{align}
Let $\iota_L$ be the ordering of $\{z_1,\dots,z_L\}$ defined by $\iota_L(i)=z_i$. Then there exists a unique
\begin{subequations}\label{eq110}
\begin{align}
\wtd\gimel\in \ST^*_{\wtd\ff,\eps',\eps*\iota_L}\big(\bbs_{\ff,\eps',\eps}(\Wbb),\Wbb\otimes\Vbb^{\otimes L}\big)
\end{align}
such that $\wtd\gimel$, as a linear functional $\bbs_{\ff,\eps',\eps}(\Wbb)\otimes\Wbb\otimes\Vbb^{\otimes L}\rightarrow\Cbb$, satisfies
\begin{align}\label{eq110b}
\gimel(-)=\wtd\gimel(-\otimes\idt^{\otimes L})
\end{align}
\end{subequations}
Moreover, $\big(\bbs_{\ff,\eps',\eps}(\Wbb),\wtd\gimel\big)$ is an $(\eps',\eps*\iota_L)$-dual fusion product of $\Wbb\otimes\Vbb^{\otimes L}$ along $\wtd\ff$.
\end{thm}

By viewing $\wtd\gimel$ as a linear functional on $\bbs_{\ff,\eps',\eps}(\Wbb)\otimes\Wbb\otimes\Vbb^{\otimes L}$, we regard
\begin{align*}
\ST^*_{\wtd\ff,\eps',\eps*\iota_L}\big(\bbs_{\ff,\eps',\eps}(\Wbb),\Wbb\otimes\Vbb^{\otimes L}\big)=\ST^*_{\wtd\ff,\eps'*\eps*\iota_L}\big(\bbs_{\ff,\eps',\eps}(\Wbb)\otimes\Wbb\otimes\Vbb^{\otimes L}\big)
\end{align*}

\begin{df}\label{lb37}
The data $\wtd\ff=\eqref{eq111}$ is called the \textbf{propagation} of $\ff$ at $z_1,\dots,z_L$ with local coordinates $\theta_{z_1},\dots,\theta_{z_L}$. The pair
\begin{align*}
\big(\bbs_{\ff,\eps',\eps}(\Wbb),\wtd\gimel\big)
\end{align*}
defined in Thm. \ref{lb36} is called the  \textbf{propagation of the dual fusion product}  $(\bbs_{\ff,\eps',\eps}(\Wbb),\gimel)$ at $z_1,\dots,z_L$ with local coordinates $\theta_{z_1},\dots,\theta_{z_L}$.
\end{df}

Note that unlike the proof of Thm. \ref{lb3}, in Thm. \ref{lb36} and Def. \ref{lb37} we do not assume that each component of $C$ intersects $E\cup\{z_1,\dots,z_L\}$. However, since we are assuming that each component of $C$ intersects $E\cup E'$ (cf. Def. \ref{lb34}), the results in \cite{GZ1} on the propagation of conformal blocks still apply.

\begin{proof}[\textbf{Proof of Thm. \ref{lb36}}]
Let $(\bbs_{\ff,\eps',\eps}(\Wbb),\gimel)$ be an $(\eps',\eps)$-dual fusion product of $\Wbb$ along $\ff$. The existence and uniqueness of $\wtd\gimel$ satisfying \eqref{eq110} follow from the linear isomorphism \eqref{eq11}. 

Let $\big(\bbs_{\wtd\ff,\eps',\eps*\iota_L} (\Wbb\otimes \Vbb^{\otimes L}),\wtd \daleth\big)$ be any $(\eps',\eps*\iota_L)$-dual fusion product of $\Wbb\otimes\Vbb^{\otimes L}$ along $\wtd\ff$. By the proof of Thm. \ref{lb3}, $\big(\bbs_{\wtd\ff,\eps',\eps*\iota_L} (\Wbb\otimes \Vbb^{\otimes L}),\daleth\big)$ is an $(\eps',\eps)$-dual fusion product of $\Wbb$ along $\ff$, provided that we set $\daleth(-)=\wtd\daleth(-\otimes\idt^{\otimes L})$. Therefore, by the uniqueness of dual fusion products, there exists a unique $\Vbb^{\otimes R}$-module isomorphism
\begin{align*}
T:\bbs_{\wtd\ff,\eps',\eps*\iota_L} (\Wbb\otimes \Vbb^{\otimes L})\xlongrightarrow{\simeq}\bbs_{\ff,\eps',\eps}(\Wbb)
\end{align*}
such that $\daleth=\gimel\circ(T\otimes\id_\Wbb)$. With the help of the isomorphism \eqref{eq11}, one easily checks that $\wtd\daleth=\wtd\gimel\circ(T\otimes\id_\Wbb\otimes\id_{\Vbb^{\otimes L}})$. This proves that $\big(\bbs_{\ff,\eps',\eps}(\Wbb),\wtd\gimel\big)$ is an $(\eps',\eps*\iota_L)$-dual fusion product of $\Wbb\otimes\Vbb^{\otimes L}$ along $\wtd\ff$.
\end{proof}

\subsection{Basic properties}

\subsubsection{The action of $\fk S_N$}

Let $\fx=(C\big|E;\eta_E)$ be an $N$-pointed compact Riemann surface with local coordinates. 

\begin{df}\label{lb61}
Choose $\alpha\in\fk S_N$. Then $\alpha$ induces an automorphism 
	  \begin{align*}
		\ovl{\alpha}:\Mod(\Vbb^{\otimes N})\xrightarrow{\simeq} \Mod(\Vbb^{\otimes N})
	  \end{align*}
	  defined as follows. For each $\Wbb\in \Mod(\Vbb^{\otimes N})$
\begin{subequations}
\begin{gather}
\text{$\ovl{\alpha}(\Wbb):=\Wbb$ as vector spaces}
\end{gather}
The $i$-th vertex operator on $\ovl{\alpha}(\Wbb)$ is defined by
	  \begin{align}
		Y_{\ovl{\alpha}(\Wbb),i}:=Y_{\Wbb,\alpha^{-1}(i)}
	  \end{align}
\end{subequations}
The operator $\ovl\alpha$ acts as the identity on the Hom spaces of $\Mod(\Vbb^{\otimes N})$. 
\end{df}

\begin{rem}
Note that for each $\alpha,\beta\in\fk S_N$ we have
\begin{align*}
\ovl\alpha\circ\ovl\beta=\ovl{\alpha\circ\beta}
\end{align*}
since, for each $\Wbb\in\Mod(\Vbb^{\otimes N})$, we have
\begin{align*}
Y_{\ovl\alpha\circ\ovl\beta(\Wbb),i}=Y_{\ovl\beta(\Wbb),\alpha^{-1}(i)}=Y_{\Wbb,\beta^{-1}\circ\alpha^{-1}(i)}=Y_{\Wbb,(\alpha\circ\beta)^{-1}(i)}
\end{align*}
Therefore, we have a group homomorphism
\begin{align*}
\fk S_N\rightarrow\mathrm{Aut}\big(\Mod(\Vbb^{\otimes N})\big)\qquad\alpha\mapsto\ovl\alpha
\end{align*}
\end{rem}

\begin{pp}\label{lb12}
Let $\eps:\{1,2,\cdots,N\}\rightarrow E$ be an ordering of $E$. Let $\Wbb\in\Mod(\Vbb^{\otimes N})$. Let $\alpha\in \fk S_N$. Then the identity map
\begin{align*}
\id_\Wbb:\Wbb\rightarrow\ovl\alpha(\Wbb)
\end{align*}
descends to a linear isomorphism
\begin{align*}
\ST_{\fx,\eps}(\Wbb)\xlongrightarrow{\simeq} \ST_{\fx,\eps\circ \alpha^{-1}}(\ovl{\alpha}(\Wbb))
\end{align*}
Therefore, its transpose is a linear isomorphism
\begin{align*}
\ST_{\fx,\eps\circ \alpha^{-1}}^*(\ovl{\alpha}(\Wbb))\xlongrightarrow{\simeq}  \ST_{\fx,\eps}^*(\Wbb)
\end{align*}
\end{pp}

	\begin{proof}
		Choose $\sigma\in H^0\big(C,\SV_\fx\otimes \omega_C(\blt E)\big)$ and $w\in \Wbb$. If we view $w$ as an element of the module $\Wbb$, then 
		\begin{gather*}
			\sigma*^\eps w=\sum_{i=1}^N\Res_{\eps(i)}~Y_{\Wbb,i}(\MU_\varrho(\eta_{\eps(i)})\sigma,\eta_{\eps(i)})w
		\end{gather*}
		If we view $w$ as an element of the module $\ovl{\alpha}(\Wbb)$, then
		\begin{align*}
&\sigma*^{\eps\circ \alpha^{-1}} w=\sum_{i=1}^N\Res_{\eps\circ\alpha^{-1}(i)}~Y_{\ovl\alpha(\Wbb),i}(\MU_\varrho(\eta_{\eps\circ\alpha^{-1}(i)})\sigma,\eta_{\eps\circ\alpha^{-1}(i)})w\\
=&\sum_{i=1}^N\Res_{{\eps\circ\alpha^{-1}(i)}}~Y_{\Wbb,\alpha^{-1}(i)}(\MU_\varrho(\eta_{\eps\circ\alpha^{-1}(i)})\sigma,\eta_{\eps\circ\alpha^{-1}(i)})w
		\end{align*}
The above two expressions are clearly equal. This proves our result.
	\end{proof}

\begin{rem}
A pictorial illustration of Prop. \ref{lb12} is given by
\begin{align}\label{eq112}
	\ST^*\bigg(~\vcenter{\hbox{{
		 \includegraphics[height=2cm]{fig2.pdf}}}}~~\bigg)\simeq\ST^*\bigg(~\vcenter{\hbox{{
		 \includegraphics[height=2cm]{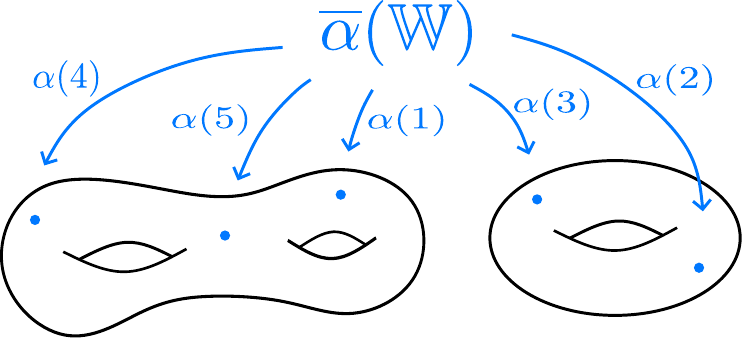}}}}~~\bigg)
\end{align}
where the isomorphism is induced by $\id:\Wbb\rightarrow\ovl\alpha(\Wbb)$. On the LHS of \eqref{eq112}, the number for each $x\in E$ is $\eps^{-1}(x)$. On the RHS, the number for $x$ is $\alpha\circ\eps^{-1}(x)$.
\end{rem}

\subsubsection{Isomorphisms of pointed surfaces}

\begin{df}\label{lb5}
	Suppose we have two unordered $N$-pointed compact Riemann surfaces with local coordinates 
\begin{gather*}
\fx=(C\big|E;\eta_E)\qquad \fy=(D\big|F;\tau_F)
\end{gather*}
An \textbf{isomorphism of \pmb{$N$}-pointed compact Riemann surfaces with local coordinates} $f:\fx\rightarrow\fy$ denotes a biholomorphism $f:C\rightarrow D$ satisfying the following conditions:
	\begin{itemize}
		\item $f(E)=F$.
		\item The pullpack of $\tau_F$ along $f$ is $\eta_E$. More precisely, for each $x\in E$, the relation
\begin{align*}
\tau_{f(x)}\circ f=\eta_x
\end{align*}
holds on a neighborhood $U_x$ of $x$ such that $\eta_x$ is defined on $U_x$ and $\tau_{f(x)}$ is defined on $f(U_x)$.
	\end{itemize}
\end{df}

\begin{pp}\label{lb7}
Let $\fx,\fy$ be as in Def. \ref{lb5}, and let $f:\fx\rightarrow\fy$ be an isomorphism of $N$-pointed compact Riemann surfaces with local coordinates. Choose an ordering $\eps:\{1,\cdots, N\}\rightarrow E$. Let $\Wbb\in \Mod(\Vbb^{\otimes N})$. Then we have
\begin{align*}
\ST_{\fx,\eps}(\Wbb)=\ST_{\fy,f\circ \eps}(\Wbb)\qquad \ST_{\fx,\eps}^*(\Wbb)=\ST_{\fy,f\circ \eps}^*(\Wbb)
\end{align*}
\end{pp}

In other words, the space of coinvariants (resp. space of conformal blocks) associated to $\Wbb$ and $\fx$ via $\eps$ is identical to the one associated to $\Wbb$ and $\fy$ via $f\circ\eps$. (Note that $f\circ\eps$ is an ordering of $F$.)

\begin{proof}
It is clear that
		\begin{gather*}
			H^0\big(C,\SV_{\fx}\otimes \omega_C(\blt E)\big)\cdot^\eps \Wbb=H^0\big(D,\SV_{\fy}\otimes \omega_D(\blt F)\big)\cdot^{f\circ \eps} \Wbb
		\end{gather*}
\end{proof}

\begin{df}\label{lb6}
	Suppose that we are given two unordered $(R,N)$-pointed compact Riemann surfaces with local coordinates
	\begin{gather*}
		\ff=(E';\theta'_{E'}\big|C\big|E;\theta_E)\qquad
		\fk K=(F';\vartheta'_{F'}\big|D\big|F;\vartheta_F)
	\end{gather*}
If $f:\ff\rightarrow\fg$ is an isomorphism of $(R+N)$-pointed compact Riemann surfaces with local coordinates satisfying $f(E)=F$ and $f(E')=F'$, we call $f$ an \textbf{isomorphism of \pmb{$(R,N)$}-pointed compact Riemann surfaces with local coordinates}.
\end{df}

\begin{pp}\label{lb8}
Let $\ff,\fg$ be as in Def. \ref{lb6}, and let $f:\ff\rightarrow\fg$ be an isomorphism of $(R,N)$-pointed compact Riemann surfaces with local coordinates. Choose orderings $\eps:\{1,\cdots,N\}\rightarrow E$ and  $\eps':\{1,\cdots,R\}\rightarrow E'$. Let $\Wbb\in\Mod(\Vbb^{\otimes N})$. Let $(\bbs_{\ff,\eps',\eps}(\Wbb),\gimel)$ be an $(\eps',\eps)$-dual fusion product of $\Wbb$ along $\ff$. Then $(\bbs_{\ff,\eps',\eps}(\Wbb),\gimel)$ is also an $(f\circ\eps',f\circ\eps)$-dual fusion product of $\Wbb$ along $\fg$.
\end{pp}

\begin{proof}
This is clear by Prop. \ref{lb7}.
\end{proof}

\subsubsection{Change of coordinates}

\begin{pp}\label{lb65}
Let $\fx=(C|E;\eta_E)$ be an unordered $N$-pointed compact Riemann surface with local coordinates. Choose an ordering $\eps:\{1,\dots,N\}\rightarrow E$. Let $\Wbb\in\Mod(\Vbb^{\otimes N})$. For each $1\leq i\leq N$, let $\alpha_i\in\MG$, and choose an argument $\arg\alpha'_i(0)$. Let $\wtd\fx=(C|E;\wtd\eta_E)$ where
\begin{align*}
\wtd\eta_{\eps(i)}=\alpha_i\circ\eta_{\eps(i)}\qquad\text{for all }i
\end{align*}
Then the (invertible) operator $\MU_1(\alpha_1)\circ\cdots\circ\MU_N(\alpha_N)$ on $\Wbb$ descends to a linear isomorphism
\begin{gather*}
\ST_{\fx,\eps}(\Wbb)\xlongrightarrow{\simeq}\ST_{\wtd\fx,\eps}(\Wbb)
\end{gather*}
Therefore, its transpose is a linear isomorphism:
\begin{gather*}
\ST^*_{\wtd\fx,\eps}(\Wbb)\xlongrightarrow{\simeq}\ST^*_{\fx,\eps}(\Wbb)\qquad\upphi\mapsto\upphi\circ \MU_1(\alpha_1)\circ\cdots\circ\MU_N(\alpha_N)
\end{gather*}
\end{pp}

\begin{proof}
This follows from the coordinate-free definition of conformal blocks. See \cite[Sec. 2.1]{GZ1}, \cite[Sec. 6.5]{FB04}, or \cite[Thm. 3.2]{Gui-sewingconvergence}. Note that in \cite{GZ1}, a diagonal operator $\wtd L_i(0)$ is used instead of $L_i(0)$ to define \eqref{eq93}. The operator $\wtd L_i(0)$ satisfies that $L_i(0)-\wtd L_i(0)$ commutes with the action of $\Vbb^{\otimes N}$. Therefore, the operator $\MU_i(\alpha)$ defined in this paper equals the composition of the corresponding operator in \cite{GZ1} with an automorphism of the $\Vbb^{\otimes N}$-module $\Wbb$. Thus, the results in \cite[Sec. 2.1]{GZ1} remain applicable in the present setting.
\end{proof}

\begin{pp}\label{lb66}
Let $\fk F=\big(E'; \theta'_{E'}\big|C\big|E;\theta_E\big)$ be an unordered $(R,N)$-pointed compact Riemann surfaces with local coordinates. Choose orderings $\eps:\{1,\cdots,N\}\rightarrow E$ and  $\eps':\{1,\cdots,R\}\rightarrow E'$. Let $\Wbb\in\Mod(\Vbb^{\otimes N})$. For each $1\leq i\leq N$ and $1\leq j\leq R$, choose $\alpha_i,\beta_j\in\MG$ with prescribed $\arg\alpha'_i(0)$ and $\arg\beta'_j(0)$, and let
\begin{align*}
\wtd\theta_{\eps(i)}=\alpha_i\circ\theta_{\eps(i)}\qquad\wtd\theta'_{\eps'(j)}=\beta_j\circ\theta'_{\eps'(j)}
\end{align*}
Let $\wtd\ff=\big(E'; \wtd\theta'_{E'}\big|C\big|E;\wtd\theta_E\big)$. Let $\big(\bbs_{\wtd\ff,\eps',\eps}(\Wbb),\wtd\gimel\big)$ be an $(\eps',\eps)$-dual fusion product of $\Wbb$ along $\wtd\ff$. Then $\big(\bbs_{\wtd\ff,\eps',\eps}(\Wbb),\gimel\big)$ is an $(\eps',\eps)$-dual fusion product of $\Wbb$ along $\ff$, where
\begin{align*}
\gimel=\wtd\gimel\circ(\MU_1(\alpha_1)\cdots\MU_N(\alpha_N)\otimes\MU_1(\beta_1)\cdots\MU_R(\beta_R) ):\Wbb\otimes\bbs_{\wtd\ff,\eps',\eps}(\Wbb)\rightarrow\Cbb
\end{align*}
\end{pp}

\begin{proof}
This follows immediately from Prop. \ref{lb65}.
\end{proof}

\subsection{Standard $2$-pointed spheres and the default fusion product $(\boxtimes_\fn\Cbb,\upomega)$}

\subsubsection{The $2$-pointed sphere $\fn$ and the default fusion product $(\boxtimes_\fn\Cbb,\upomega)$}

Recall from Sec. \ref{lb32} that $\zeta$ denotes the standard coordinate of $\Cbb$. 

\begin{df}\label{lb51}
Throughout this paper, we let $\fn$ denote the (unordered) $2$-pointed sphere with local coordinates:
\begin{align}\label{eq12}
		\fn=\big(\Pbb^1\big| \{\infty,0\};(\eta_x)_{x\in\{\infty,0\}}\big)\qquad\text{where }\eta_\infty=1/\zeta\text{ and }\eta_0=\zeta
\end{align}
Then the automorphism group $\Aut(\fn)\simeq\Zbb_2$ is generated by
\begin{align*}
\tipath:\Pbb^1\rightarrow\Pbb^1\qquad z\mapsto 1/z
\end{align*}
The \textbf{default ordering of \pmb{$\{\infty,0\}$}} is defined to be
\begin{align}
\epsilon:\{+,-\}\xlongrightarrow{\simeq}\{\infty,0\}\qquad \epsilon(+)=\infty\qquad \epsilon(-)=0
\end{align} 
\end{df}

\begin{rem}\label{lb40}
Choose any $\Wbb\in\Mod(\Vbb^{\otimes 2})$. By Prop. \ref{lb7}, we have $\ST_{\fn,\epsilon}^*(\Wbb)=\ST_{\fn,\tipath\circ\epsilon}^*(\Wbb)$. In fact, both spaces consist of linear functionals $\upphi:\Wbb\rightarrow\Cbb$ satisfying the relation
\begin{subequations}\label{eq69}
\begin{align}
\bk{\upphi,Y_+'(v,z)w}=\bk{\upphi,Y_-(v,z)w}\qquad\text{for all }v\in\Vbb,w\in\Wbb
\end{align}
in $\Cbb[[z^{\pm 1}]]$. Equivalently, $\upphi$ satisfies 
\begin{align}
\bk{\upphi,Y_+(v,z)w}=\bk{\upphi,Y'_-(v,z)w}\qquad\text{for all }v\in\Vbb,w\in\Wbb
\end{align}
\end{subequations}
See \cite[Rem. 2.1]{GZ3} for more explanations.
\end{rem}

\begin{df}\label{lb29}
By viewing $\fn=\eqref{eq12}$ as an $(2,0)$-pointed sphere with local coordinates
	\begin{align}\label{eq18}
\fn=\big( \{\infty,0\};(\eta_x)_{x\in\{\infty,0\}}\big|\Pbb^1\big)\qquad\text{where }\eta_\infty=1/\zeta\text{ and }\eta_0=\zeta
		\end{align}
we fix, throughout this article, an $\epsilon$-dual fusion product $(\bbs_{\fn,\epsilon}(\Cbb),\upomega)$ of $\Cbb$ along $\fn$, cf. Rem. \ref{lb39}. We use the abbreviations
\begin{align*}
{\bbs_\fn\Cbb}:=\bbs_{\fn,\epsilon}(\Cbb)\qquad {\boxtimes_\fn\Cbb}:=\boxtimes_{\fn,\epsilon}(\Cbb)
\end{align*}
and call $(\bbs_\fn\Cbb,\upomega)$ the \textbf{default dual fusion product of $\pmb\Cbb$ along $\pmb\fn$}. Accordingly, $(\boxtimes_\fn\Cbb,\upomega)$---or simply $\boxtimes_\fn\Cbb$---is called the \textbf{default fusion product of $\pmb\Cbb$ along $\pmb\fn$}.
	\end{df}

\begin{rem}
Note that the canonical conformal block
\begin{align*}
\upomega\in\ST^*_{\fn,\epsilon}(\bbs_\fn\Cbb)=\ST^*_{\fn,\tipath\circ\epsilon}(\bbs_\fn\Cbb)
\end{align*}
is a linear functional
\begin{align*}
\upomega:\bbs_\fn\Cbb\rightarrow\Cbb
\end{align*}
satisfying the same condition as $\upphi$ in Rem. \ref{lb40}.
\end{rem}

\subsubsection{Standard $2$-pointed spheres}

\begin{df}\label{lb43}
A \textbf{standard $2$-pointed sphere} is defined to be a $2$-pointed compact Riemann surface with local coordinates $\fk C$ that is isomorphic to $\fn$. Equivalently, it is defined to be
\begin{align}\label{eq22}
\fk C=(C\big| \{z,\tipaz\};\eta_z,\eta_\tipaz)
\end{align}
where $C$ is a compact Riemann surface biholomorphic to $\Pbb^1$, and the local coordinates $\eta_z,\eta_\tipaz$ are linear fractional transformations satisfying $\eta_z\cdot\eta_\tipaz=1$.
\end{df}

\begin{rem}\label{lb59}
Let $\Wbb\in\Mod(\Vbb^{\otimes2})$. According to Rem. \ref{lb40}, the space of conformal blocks associated to $\Wbb$ and a standard $2$-pointed sphere $\fk C$ is independent of the ordering of the marked points of $\fk C$. Therefore, we denote this space by $\pmb{\ST^*_{\fk C}(\Wbb)}$, whose elements are precisely the linear functionals on $\Wbb$ satisfying \eqref{eq69}.
\end{rem}

\begin{pp}\label{lb13}
Let $\fc=\eqref{eq22}$ be a standard $2$-pointed sphere. Let $\eps:\{+,-\}\rightarrow\{z,\tipaz\}$ be any ordering of $\{z,\tipaz\}$. Then $(\boxtimes_\fn\Cbb,\upomega)$ is an $\eps$-fusion product of $\Cbb$ along $\fc$.
\end{pp}

\begin{proof}
This is clear by Prop. \ref{lb8}.
\end{proof}

The following two figures represent the fusion products of $\Cbb$ along $\fc$ with respect to the two orderings of $\{z,\tipaz\}$.
\begin{align}\label{eq29}
\vcenter{\hbox{{\includegraphics[height=1.8cm]{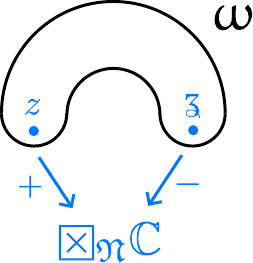}}}}
\qquad\text{and}\qquad
\vcenter{\hbox{{\includegraphics[height=1.8cm]{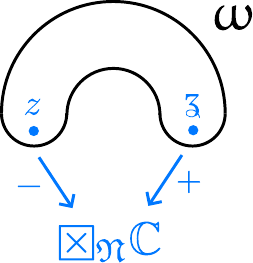}}}}
\end{align}

\subsection{Composition of conformal blocks and the sewing-factorization theorem}\label{lb76}

Let $K,N,R\in\Nbb$. In this section, we let 
\begin{gather*}
\ff=\big(F'; \theta'_{F'}\big|C_1\big|F;\theta_F\big)\qquad
\fg=\big(G';\mu'_{G'}\big|C_2\big| G;\mu_G \big)
    \end{gather*}
be respectively $(K,N)$-pointed and $(N,R)$-pointed compact Riemann surfaces with local coordinates, cf. Def. \ref{lb1}. 
Fix orderings
	\begin{gather*}
\eps':\{1,\cdots,K\}\xlongrightarrow{\simeq}F'\qquad	\eps:\{1,\cdots,N\}\xlongrightarrow{\simeq}F\\ \tipae':\{1,\cdots,N\}\xlongrightarrow{\simeq}G'\qquad \tipae:\{1,\cdots,R\}\xlongrightarrow{\simeq}G
	\end{gather*}

\subsubsection{The sewing $\ff\#^{\eps,\tipae'}_{p_\blt}\fg$ and the composition $\ff\#^{\eps,\tipae'}\fg$}\label{lb46}

For each $x\in F$ (resp. $y'\in G'$), choose a neighborhood $V_x$ (resp. $W'_{y'}$) such that $\theta_x$ (resp. $\mu'_{y'}$) is defined. Assume that $V_x$ and $W'_{y'}$ are open disks, i.e.
\begin{align}
\theta_x(V_x)=\MD_{r_x}\qquad \mu'_{y'}(W'_{y'})=\MD_{\rho_{y'}}\qquad\text{where }r_x,\rho_{y'}\in(0,+\infty]
\end{align} 
The numbers $r_x,\rho_{y'}$ are called \textbf{sewing radii}. Assume that $V_{x_1}\cap V_{x_2}=\emptyset$ if $x_1\neq x_2$, and that $W'_{y_1'}\cap W'_{y_2'}=\emptyset$ if $y_1'\neq y_2'$.

\begin{df}\label{lb11}
	Set
\begin{gather*}
	\MD_{r_\blt\rho_\blt}=\prod_{i=1}^N \MD_{r_{\eps(i)}\rho_{\tipae'(i)}}\quad 
	\MD^\times_{r_\blt\rho_\blt}=\prod_{i=1}^N \MD^\times_{r_{\eps(i)}\rho_{\tipae'(i)}}
\end{gather*}
Choose $p_\blt=(p_1,\dots,p_N)\in\MD_{r_\blt\rho_\blt}^\times$. In other words, $p_\blt\in\Cbb^N$ satisfies
\begin{align}\label{eq121}
0<|p_i|<r_{\eps(i)}\rho_{\tipae'(i)}\qquad\text{for each }1\leq i\leq N
\end{align}
The \textbf{sewing} $\ff\#^{\eps,\tipae'}_{p_\blt}\fg$ of $\ff$ and $\fg$ via $(\eps,\tipae')$ with \textbf{sewing moduli} $p_\blt$ is defined as follows. (Note that if \eqref{eq121} is satisfied, we say that \textbf{the sewing radii $r_\blt,\rho_\blt$ are admissible for the sewing moduli $p_\blt$}.)

For each $1\leq i\leq N$, 
\begin{align*}
\Gamma_i=\Big\{z\in V_{\eps(i)}:|\theta_{\eps(i)}(z)|\leq \frac{|p_i|}{\rho_{\tipae'(i)}} \Big\}
\qquad
\Delta_i=\Big\{z\in W'_{\tipae'(i)}:|\mu'_{\tipae'(i)}(z)|\leq \frac{|p_i|}{r_{\eps(i)}} \Big\}
\end{align*}
are compact subsets of $V_{\eps(i)}$ and $W'_{\tipae'(i)}$ respectively. We have a biholomorphism
\begin{align*}
\scr S_i: V_{\eps(i)}\setminus\Gamma_i\xlongrightarrow{\simeq} W'_{\tipae'(i)}\setminus\Delta_i\qquad z\mapsto (\mu'_{\tipae'(i)})^{-1}\Big(\frac{p_i}{\theta_{\eps(i)}(z)}\Big)
\end{align*}
The compact Riemann surface $C_1\#^{\eps,\tipae'}_{p_\blt}C_2$ is defined by removing $\Gamma_i$ and $\Delta_i$ (for all $1\leq i\leq N$) and gluing $V_{\eps(i)}\setminus\Gamma_i$ and $W'_{\tipae'(i)}\setminus\Delta_i$ via the biholomorphism $\scr S_i$. In other words, for each $z_1\in V_{\eps(i)}\setminus\Gamma_i$ and $z_2\in W'_{\tipae'(i)}\setminus\Delta_i$,
\begin{align}
\text{$z_1$ is identified with $z_2$}\qquad\Longleftrightarrow\qquad \theta_{\eps(i)}(z_1)\cdot \mu'_{\tipae'(i)}(z_2)=p_i
\end{align}
Note that after gluing, $F$ and $G'$ are removed, but $F',G$ and their local coordinates $\theta'_{F'},\mu_G$ remain. We let
\begin{align}
\ff\#^{\eps,\tipae'}_{p_\blt}\fg=\big(F';\theta'_{F'}\big|C_1\#^{\eps,\tipae'}_{p_\blt}C_2\big|G;\mu_G\big)
\end{align}
Then $\ff\#^{\eps,\tipae'}_{p_\blt}\fg$ satisfies the definition of an $(K,R)$-pointed compact Riemann surface with local coordinates, except that \eqref{eq8} in Def. \ref{lb1} is not necessarily satisfied---that is, it is \textit{not necessarily true that}
\begin{align}\label{eq114}
\text{each component of $C_1\#^{\eps,\tipae'}_{p_\blt}C_2$ intersects $F'\cup G$}
\end{align}
The superscript $\eps,\tipae'$ in $\#^{\eps,\tipae'}$ and $\#^{\eps,\tipae'}_{p_\blt}$ will be omitted when the context is clear.

In the case where $p_1=\dots=p_N=1$ (note that this requires $r_{\eps(i)}\rho_{\tipae'(i)}>1$ for each $i$, due to \eqref{eq121}), we suppress the subscript $p_\blt$, that is, we write
\begin{gather*}
\ff\#^{\eps,\tipae'}\fg:=\ff\#^{\eps,\tipae'}_{(1,\dots,1)}\fg
\end{gather*}
We call $\ff\#^{\eps,\tipae'}\fg$ the \textbf{composition} of $\ff$ and $\fg$ via $(\eps,\tipae')$. \hqed
	\end{df}

\subsubsection{Sewing and composition of conformal blocks}

Assume that condition \eqref{eq114} holds, namely, each component of $\ff\#^{\eps,\tipae'}_{p_\blt}\fg$ contains at least one incoming or outgoing marked point. (Note that this condition is independent of the choice of $p_\blt$.)

Choose $\Wbb\in\Mod(\Vbb^{\otimes K})$, $\Mbb\in\Mod(\Vbb^{\otimes N})$, $\Xbb\in\Mod(\Vbb^{\otimes R})$. Choose
\begin{align}\label{eq27}
\upphi\in\ST_{\ff,\eps'*\eps}^*(\Wbb'\otimes\Mbb)\qquad\uppsi\in\ST^*_{\fg,\tipae'*\tipae}(\Mbb'\otimes\Xbb)
\end{align}
Since $\upphi$ and $\uppsi$ are linear functionals $\Wbb'\otimes\Mbb\rightarrow\Cbb$ and $\Mbb'\otimes\Xbb\rightarrow\Cbb$, respectively, they can be viewed as linear maps
\begin{gather}
\upphi^\sharp:\Mbb\rightarrow\ovl\Wbb\qquad\uppsi^\sharp:\Xbb\rightarrow\ovl\Mbb
\end{gather}
Choose $p_\blt\in\MD_{r_\blt\rho_\blt}^\times$ with fixed arguments $\arg p_1,\dots,\arg p_R$. 

In the following, we let
\begin{align}
p_\blt^{L_\blt(0)}:=p_1^{L_1(0)}\cdots p_N^{L_N(0)}
\end{align}

\begin{thm}\label{lb47}
The linear functional
\begin{gather*}
\upphi\circ_{p_\blt}\uppsi: \Wbb'\otimes\Xbb\rightarrow\Cbb\\
w'\otimes w\mapsto  \sum_{\lambda_\blt\in\Cbb^N} \bigbk{w',\upphi^\sharp\circ p_\blt^{L_\blt(0)}\circ P(\lambda_\blt)\circ\uppsi^\sharp(w)}   
\end{gather*}
converges absolutely for each $w'\in\Wbb'$ and $w\in\Xbb$. Moreover, $\upphi\circ_{p_\blt}\uppsi$ belongs to $\ST^*_{\ff\#^{\eps,\tipae'}_{p_\blt}\fg,\eps',\tipae}(\Wbb',\Xbb)$, that is, it is a conformal block associated to $\ff\#^{\eps,\tipae'}_{p_\blt}\fg$ and $\Wbb',\Xbb$ via the orderings $\eps',\tipae$.
\end{thm}

\begin{proof}
This is due to Thm. 4.9 and Rem. 4.10 of \cite{GZ2}.
\end{proof}

\begin{df}\label{lb71}
We call $\upphi\circ_{p_\blt}\uppsi$ the \textbf{sewing of $\upphi$ and $\uppsi$ with moduli $p_\blt$}. Note that $\upphi\circ_{p_\blt}\uppsi$ can also be defined by the \textbf{contraction}
\begin{align}\label{eq57}
\begin{aligned}
\bk{\upphi\circ_{p_\blt}\uppsi,w'\otimes w}&=\wick{\Lan\upphi,w'\otimes p_\blt^{L_\blt(0)}\c1-\Ran\cdot \Lan\uppsi,\c1 -\otimes w\Ran}\\
&:=\sum_{\lambda_\blt\in\Cbb^N}\sum_{\alpha\in\fk A_{\lambda_\blt}}\bk{\uppsi,w'\otimes p_\blt^{L_\blt(0)}e_{\lambda_\blt}(\alpha)}\cdot\bk{\uppsi,\wch e_{\lambda_\blt}(\alpha)\otimes w}
\end{aligned}
\end{align}
where $(e_{\lambda_\blt}(\alpha))_{\alpha\in\fk A_{\lambda_\blt}}$ is a (finite) basis of $\Mbb_{[\lambda_\blt]}$ with dual basis $(\wch e_{\lambda_\blt}(\alpha))_{\alpha\in\fk A_{\lambda_\blt}}$. In the case that $p_i=1$ and $\arg p_i=0$ for each $i$, we write
\begin{align*}
\upphi\circ\uppsi:=\upphi\circ_{1,\dots,1}\uppsi
\end{align*}
and call $\upphi\circ\uppsi$ the \textbf{composition of $\upphi$ and $\uppsi$}. 
\end{df}

\begin{rem}\label{lb72}
Suppose that $p_i=1$ and $\arg p_i=0$ for each $i$. The terminology of composing conformal blocks is due to the obvious fact that the linear map
\begin{subequations}
\begin{align}
(\upphi\circ\uppsi)^\sharp:\Xbb\rightarrow\ovl\Wbb
\end{align}
determined by $\upphi\circ\uppsi$ equals the composition of $\upphi^\sharp$ and $\uppsi^\sharp$, that is, for each $w\in\Xbb$,
\begin{align}\label{eq104}
(\upphi\circ\uppsi)^\sharp(w)=\upphi^\#\circ\uppsi^\#(w):=\sum_{\lambda_\blt\in\Cbb^N} \upphi^\sharp\circ P(\lambda_\blt)\circ\uppsi^\sharp(w)
\end{align}
\end{subequations}
where the RHS converges absolutely when evaluated with each element of $\Wbb'$. 
\end{rem}

\begin{rem}\label{lb41}
Here, we give a pictorial illustration of \textit{composing} conformal blocks: Let
  \begin{gather*}
	\upphi\in \ST^*\Bigg(\vcenter{\hbox{{
		\includegraphics[height=2cm]{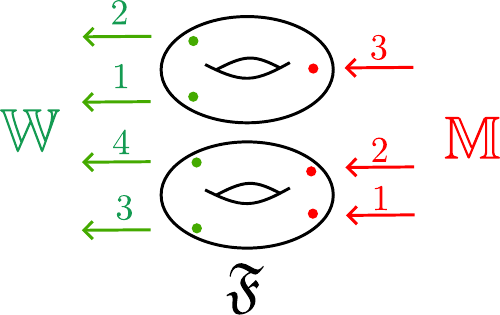}}}}\Bigg)
  \end{gather*} 
In this figure, $\ff$ has two components and $7$ marked points. The four green points on the left, listed from top to bottom as $\eps'(2),\eps'(1),\eps'(4),\eps'(3)$, are ordered by $\eps'$. The three red points on the right, listed from top to bottom as $\eps(3),\eps(2),\eps(1)$, are ordered by $\eps$. Let
\begin{gather*}
 \uppsi\in \ST^*\bigg(\vcenter{\hbox{{
			\includegraphics[height=1.3cm]{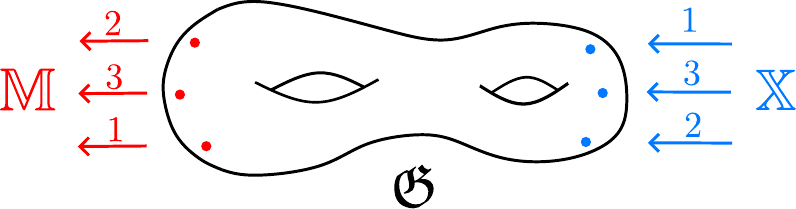}}}}~\bigg)
\end{gather*}
In this figure, $\fg$ is connected and has $6$ marked points. The three red points on the left, listed from top to bottom as $\tipae'(2),\tipae'(3),\tipae'(1)$, are ordered by $\tipae'$. The three blue points on the right, listed from top to bottom as $\tipae(1),\tipae(3),\tipae(2)$, are ordered by $\tipae$.

Let $\upchi$ be a conformal block associated to $\ff\#^{\eps,\tipae'}\fg$ and $\Wbb',\Xbb$ via $\eps',\tipae$, that is, $\upchi\in\ST^*_{\ff\#^{\eps,\tipae'}\fg,\eps',\tipae}(\Wbb',\Xbb)$. Then the relation $\upchi=\upphi\circ\uppsi$ is represented by the graphical equation
\begin{align*}
&\vcenter{\hbox{{
		 \includegraphics[height=1.8cm]{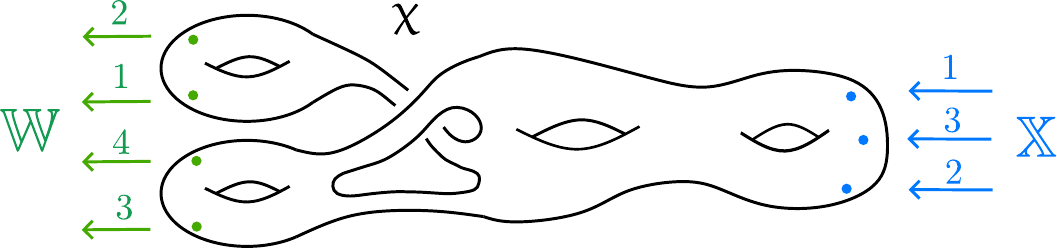}}}}\\[1ex]
=~&\vcenter{\hbox{{
			\includegraphics[height=2.2cm]{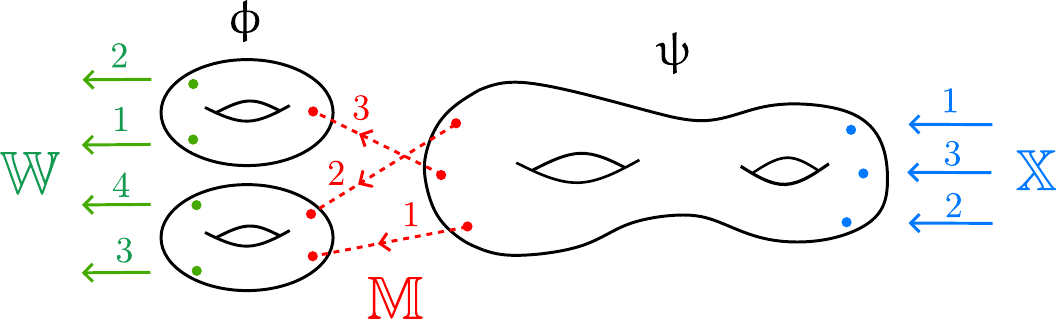}}}}
  \end{align*}
The pointed compact Riemann surface with local coordinates on the LHS of the above equation represents the composition $\ff\#^{\eps,\tipae'}\fg$ of $\ff,\fg$ via $\eps,\tipae'$.  \hqed
\end{rem}

\subsubsection{The sewing-factorization theorem}

We continue to assume that \eqref{eq114} holds, that is, each component of $\ff\#^{\eps,\tipae'}_{p_\blt}\fg$ contains at least one incoming or outgoing marked point. Choose $\Wbb\in\Mod(\Vbb^{\otimes K})$ and $\Xbb\in\Mod(\Vbb^{\otimes R})$. Choose $p_\blt\in\MD_{r_\blt\rho_\blt}^\times$ with fixed arguments $\arg p_1,\dots,\arg p_R$.

\begin{thm}\label{SF}
Let $(\boxtimes_{\fg,\tipae',\tipae}(\Xbb),\daleth)$ be an $(\tipae',\tipae)$-fusion product of $\Xbb$ along $\fg$. Then the following map is a linear isomorphism
\begin{gather}\label{eq19}
\begin{gathered}
\ST_{\ff,\eps'*\eps}^*\big(\Wbb'\otimes \boxtimes_{\fg,\tipae',\tipae}(\Xbb)\big)\xlongrightarrow{\simeq} \ST_{\ff\#^{\eps,\tipae'}_{p_\blt}\fg,\eps'*\tipae}^*(\Wbb'\otimes \Xbb)\\[0.5ex]
\upphi \mapsto \upphi\circ_{p_\blt}\daleth
\end{gathered}
\end{gather}
called the \textbf{sewing-factorization isomorphism}.
\end{thm}

Note that $\daleth\in\ST^*_{\fg,\tipae'*\tipae}(\bbs_{\fg,\tipae',\tipae}(\Xbb)\otimes\Xbb)$ is a linear functional $\bbs_{\fg,\tipae',\tipae}(\Xbb)\otimes\Xbb\rightarrow\Cbb$.

 \begin{proof}
In the special case that each component of $\fg$ intersects the set $G$ of incoming marked points, the theorem is due to \cite[Thm. 3.5]{GZ3}. The general case follows from the special case and the propagation of conformal blocks and fusion products, as we explain below.

Choose distinct points $z_1,\dots,z_L$ of $\fg$, disjoint from $G'\cup G$, such that the propagation $\wtd\fg$ of $\fg$ at $z_1,\dots,z_L$ (with arbitrarily chosen local coordinates at $z_1,\dots,z_L$) satisfies the condition that each component intersects $G\cup\{z_1,\dots,z_L\}$. Let $\iota_L$ be the ordering of $\{z_1,\dots,z_L\}$ defined by $\iota_L(i)=z_i$. Let 
\begin{align*}
\wtd\daleth\in\ST^*_{\wtd\fg,\tipae'*\tipae*\iota_L}\big(\bbs_{\fg,\tipae',\tipae}(\Xbb)\otimes\Xbb\otimes\Vbb^{\otimes L}\big)
\end{align*}
be the propagation of $\daleth$ at $z_1,\dots,z_L$, cf. Def. \ref{lb37}. Then, we have a commutative diagram
\begin{equation*}
\begin{tikzcd}[column sep=2cm,row sep=large]
\ST_{\ff,\eps'*\eps}^*\big(\Wbb'\otimes \boxtimes_{\fg,\tipae',\tipae}(\Xbb)\big) \arrow[r,"\upphi\mapsto\upphi\circ_{p_\blt}\wtd\daleth"] \arrow[d,"="'] & \ST_{\ff\#^{\eps,\tipae'}_{p_\blt}\wtd\fg,\eps'*\tipae*\iota_L}^*(\Wbb'\otimes \Xbb\otimes\Vbb^{\otimes L}) \arrow[d,"\upchi\mapsto\upchi(-\otimes\idt^{\otimes L})"] \\
\ST_{\ff,\eps'*\eps}^*\big(\Wbb'\otimes \boxtimes_{\fg,\tipae',\tipae}(\Xbb)\big) \arrow[r,"\upphi\mapsto\upphi\circ_{p_\blt}\daleth"]           & \ST_{\ff\#^{\eps,\tipae'}_{p_\blt}\fg,\eps'*\tipae}^*(\Wbb'\otimes \Xbb)      
\end{tikzcd}
\end{equation*}
where the vertical arrow on the right is an isomorphism by \eqref{eq11}. 

By Thm. \ref{lb36}, $\big(\boxtimes_{\fg,\tipae',\tipae}(\Xbb),\wtd\daleth\big)$ is an $(\tipae',\tipae*\iota_L)$-fusion product of $\Xbb\otimes\Vbb^{\otimes L}$ along $\wtd\fg$. Therefore, by \cite[Thm. 3.5]{GZ3}, the top vertical arrow in the above diagram is an isomorphism. It follows that the bottom arrow is also an isomorphism.
 \end{proof}

\begin{rem}\label{SF2}
Let $(\boxtimes_{\fg,\tipae',\tipae}(\Xbb),\daleth)$ be an $(\tipae',\tipae)$-fusion product of $\Xbb$ along $\fg$. Then the canonical conformal block $\daleth:\bbs_{\fg,\tipae',\tipae}(\Xbb)\otimes\Xbb\rightarrow\Cbb$ is \textbf{partially injective}, meaning that if each component of $\fg$ intersects the set $G$ of incoming marked points, then for each $\xi\in\bbs_{\fg,\tipae',\tipae}(\Xbb)$, we have
\begin{align}
\daleth(\xi\otimes w)=0\text{ for all $w\in\Wbb$}\qquad\Longrightarrow\qquad \xi=0
\end{align}
See \cite[Rem. 3.17]{GZ2} for the explanation. 

Since this partial injectivity is used in the proof that the sewing-factorization map \eqref{eq19} is injective (see \cite[Subsec. 2.5.2]{GZ3}), we also refer to the injectivity of the map $\upphi\mapsto\upphi\circ_{p_\blt}\daleth$ as the \textbf{partial injectivity of the canonical conformal block $\daleth$}. (Note, however, that the injectivity of this map does not require the assumption that each component of $\fg$ intersects the set $G$ of incoming marked points.)  \hqed
\end{rem}

 \begin{rem}\label{lb42}
In the remainder of this article, we restrict to the case where $p_i=1$ and $\arg p_i=0$ for each $i$. In other words, we consider only the composition of conformal blocks, rather than the more general sewing.
Then, using the graphical calculus for the composition of conformal blocks as described in Rem. \ref{lb41}, we reformulate the sewing-factorization Thm. \ref{SF} as follows: 

Consider the fusion product
  \begin{align}\label{eq26}
	\vcenter{\hbox{{
			\includegraphics[height=1.4cm]{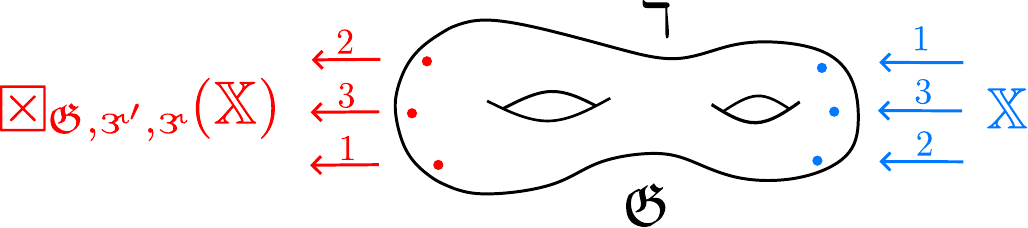}}}}
  \end{align}
Then for each
\begin{align*}
\upchi\in \ST^*\Bigg(\vcenter{\hbox{{
\includegraphics[height=1.95cm]{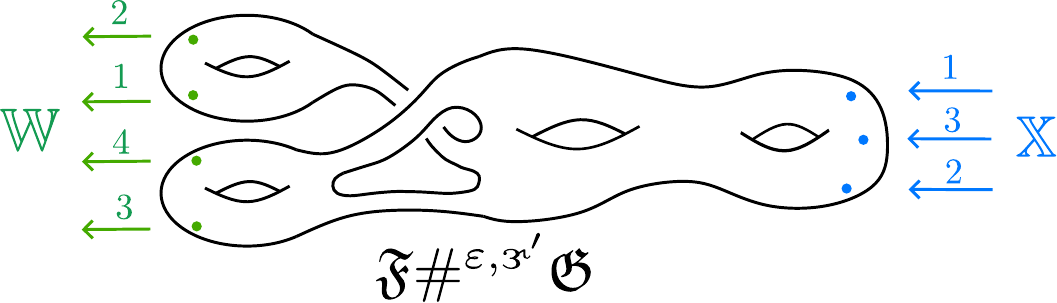}}}}\Bigg)
	\end{align*}
	there exists a unique 
	\begin{align*}
		\upphi\in \ST^*\Bigg(\vcenter{\hbox{{
		 \includegraphics[height=2cm]{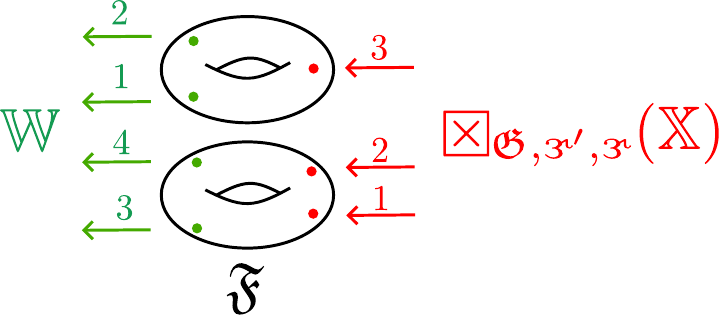}}}}~\Bigg)
	\end{align*}
such that 
\begin{align*}
\upchi=\vcenter{\hbox{{
\includegraphics[height=2.2cm]{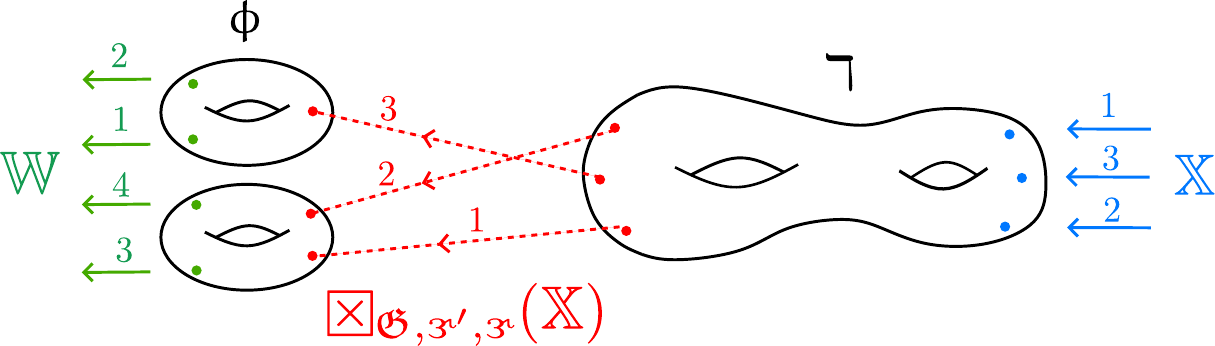}}}}
	\end{align*}
\hqed
 \end{rem}

\begin{rem}\label{SF1}
In this paper, we will apply Rem. \ref{lb42} mainly to the case that \eqref{eq26} is the fusion product $(\boxtimes_\fn\Cbb,\upomega)$ of $\Cbb$ along a standard $(2,0)$-sphere described in Prop. \ref{lb13}:
\begin{align*}
	\vcenter{\hbox{{
			\includegraphics[height=1cm]{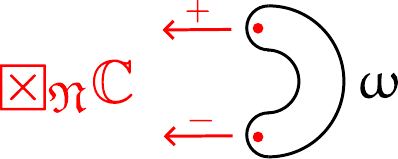}}}}
\end{align*}
Then the sewing-factorization Thm. \ref{SF} asserts that for each
\begin{align*}
\upchi\in \ST^*\Bigg(\vcenter{\hbox{{
\includegraphics[height=1.9cm]{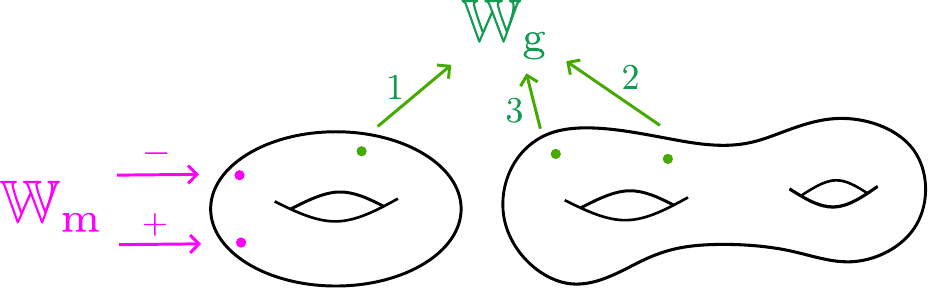}}}}~\Bigg)
	\end{align*}
	there exists a unique 
	\begin{align*}
		\upphi\in \ST^*\Bigg(\vcenter{\hbox{{
		 \includegraphics[height=1.95cm]{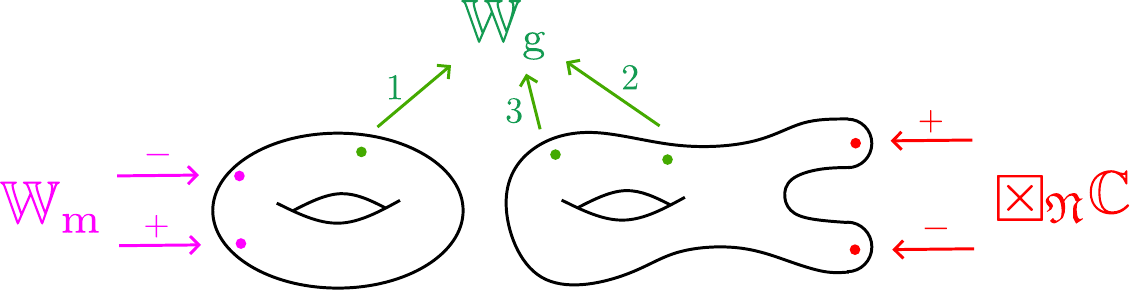}}}}~\Bigg)
	\end{align*}
such that 
\begin{align*}
\upchi=\vcenter{\hbox{{
\includegraphics[height=2.45cm]{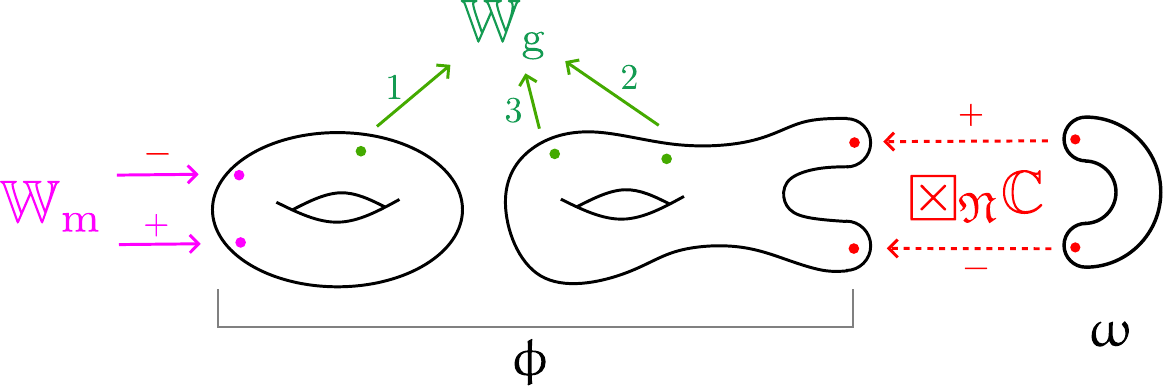}}}}
	\end{align*}
\end{rem}

\section{$\boxtimes_\fn\Cbb$ as an AUF algebra with involution $\Theta$}

Recall Def. \ref{lb29} for the default fusion product $(\boxtimes_\fn\Cbb,\upomega)$. In this chapter, we prove that $\boxtimes_\fn\Cbb$ has a natural structure of associative algebra, which is \textbf{almost unital and finite dimensional (AUF)} in the sense of \cite{GZ4}---that is, there is a family of mutually orthogonal idempotents $(e_i)_{i\in\fk I}$ of $\boxtimes_\fn\Cbb$ such that $\boxtimes_\fn\Cbb=\sum_{i,j}e_i(\boxtimes_\fn\Cbb) e_j$ where each summand is finite-dimensional. (Note that this sum must be a direct sum.) We will also explain why $\boxtimes_\fn\Cbb$ is the end $\int_{\Mbb\in\Mod(\Vbb)}\Mbb\otimes\Mbb'$.

In this chapter, all $2$-pointed spheres in the pictures are assumed to be standard (cf. Def. \ref{lb43}). Let $N\in\Zbb_+$.

\subsection{The actions $\Phi_{i,+}$ and $\Phi_{i,-}$ of $\boxtimes_\fn\Cbb$ on $\Wbb\in \Mod(\Vbb^{\otimes N})$}

Let $\Wbb,\Mbb\in \Mod(\Vbb^{\otimes N})$. In this section, we fix $1\leq i\leq N$ and consider the spaces of conformal blocks 
\begin{gather}
\ST^*\Bigg(\vcenter{\hbox{{
		\includegraphics[height=2.6cm]{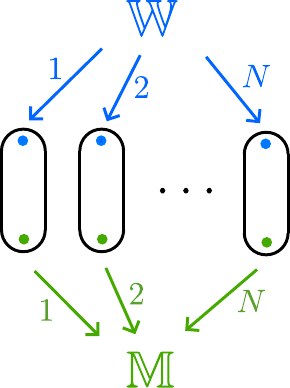}}}}~\Bigg)	\label{eq28}\\
\ST^*\Bigg(\vcenter{\hbox{{
			\includegraphics[height=2.8cm]{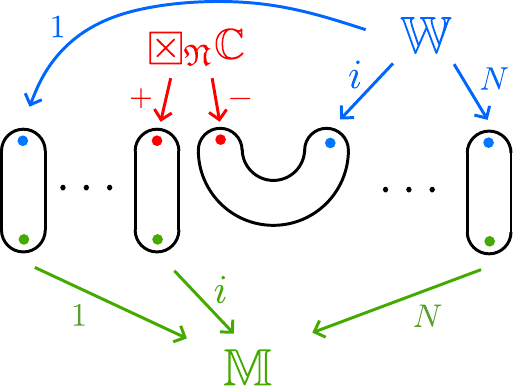}}}}~\Bigg)=	\ST^*\Bigg(\vcenter{\hbox{{
			\includegraphics[height=2.9cm]{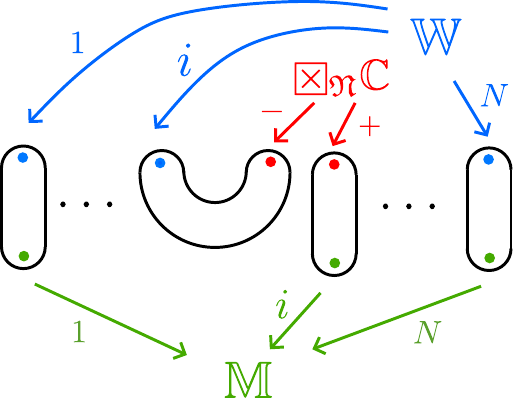}}}}~\Bigg)	\label{eq55}\\[1ex]
\ST^*\Bigg(\vcenter{\hbox{{
			\includegraphics[height=2.9cm]{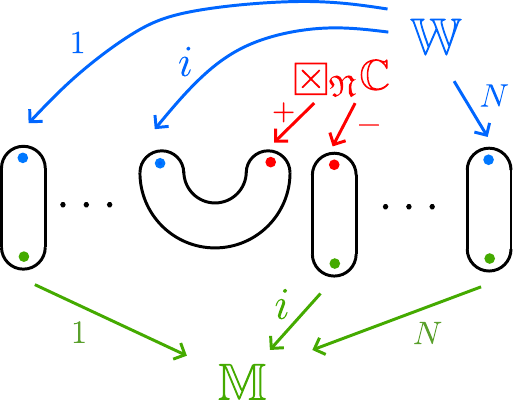}}}}~\Bigg)=	\ST^*\Bigg(\vcenter{\hbox{{
			\includegraphics[height=2.8cm]{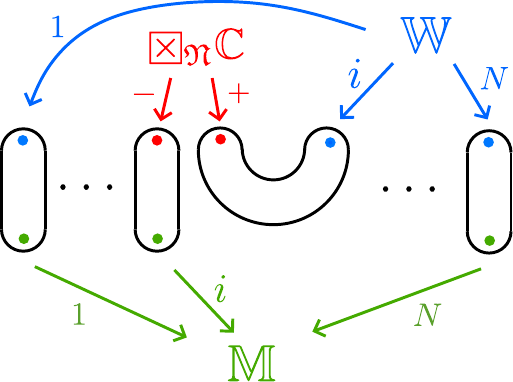}}}}~\Bigg)		\label{eq56}
\end{gather}

Let us give explicit algebraic descriptions of these conformal blocks.

\begin{rem}
Let $\upchi:\Wbb\otimes\Mbb'\rightarrow \Cbb$ be a linear map. Then $\upchi\in \eqref{eq28}$ if and only if for each $1\leq j\leq N,v\in\Vbb,w\in \Wbb$ and $m'\in\Mbb'$, the following relation holds in $\Cbb[[z^{\pm1}]]$:
\begin{align}\label{eq59}
\bigbk{\upchi,Y_j(v,z)w\otimes m'}=\bigbk{\upchi,w\otimes Y_j'(v,z) m'}
\end{align}
\end{rem}

We now show that each $\upchi\in\eqref{eq28}$, when viewed as a linear map $\Wbb\rightarrow\ovl\Mbb$, has range in $\Mbb$. Recall \eqref{eq115} for the meaning of $P_j(\lambda)$.

\begin{lm}\label{lb31}
Let $K\in\Zbb_+$. Choose $\Xbb\in \Mod(\Vbb^{\otimes K})$ and a finite subset $S\subset \Xbb$. Then for each $\lambda\in \Cbb$, there exists a polynomial $h(x)\in \Cbb[x]$ such that $P_j(\lambda)\tipaz=h(L_j(0))\tipaz$ for all $\tipaz\in S$ and $1\leq j\leq K$.
\end{lm}

\begin{proof}
It is well-known that if $T$ is a linear operator on a finite-dimensional $\Cbb$-vector space $\mc W$, then $\mc W$ is the direct sum of generalized eigenspaces of $T$, and the projection operator of $\mc W$ onto each generalized eigenspace is a polynomial of $T$.

Now, choose $\mu_\blt\in \Cbb^K$ such that $\Re (\mu_j)\geq\Re (\lambda)$ for all $j$, and that $S\subset \Xbb_{[\leq \mu_\blt]}$. Let $\mc W$ be the direct sum of $K$ copies of $\Xbb_{[\leq \mu_\blt]}$, and let $T=\diag_{j=1}^K \big(L_j(0)\big|_{\Xbb_{[\leq\mu_\blt]}}\big)$. Then $Q=\diag_{j=1}^K \big(P_j(\lambda)\big|_{\Xbb_{[\leq\mu_\blt]}}\big)$ is the projection of $\mc W$ onto the generalized eigenspace of $T$ with eigenvalue $\lambda$. Therefore, by the first paragraph, there exists a polynomial $h$ such that $Q=h(T)$.
\end{proof}

\begin{pp}\label{lb19}
Elements in \eqref{eq28} are precisely those elements of the form 
	\begin{align}\label{eq53}
		T^\flat:\Wbb\otimes \Mbb'\rightarrow \Cbb\qquad w\otimes  m'\mapsto \Lan T(w),m' \Ran
	\end{align}
	where $T\in \Hom_{\Vbb^{\otimes N}}(\Wbb,\Mbb)$. 
\end{pp}
\begin{proof}
Note that a linear map $T:\Wbb\rightarrow\Mbb$ belongs to $\Hom_{\Vbb^{\otimes N}}(\Wbb,\Mbb)$ iff for each $1\leq j\leq N,v\in\Vbb,w\in \Wbb,m'\in \Mbb'$, the following relation holds in $\Wbb[[z^{\pm1}]]$:
\begin{align}\label{eq106}
	\Lan T(Y_j(v,z)w),m'\Ran=\Lan T(w),Y_j'(v,z)m'\Ran
\end{align}
Thus, \eqref{eq53} belongs to $\eqref{eq28}$. Conversely, choose an element $\upchi\in \eqref{eq28}$. Set  
\begin{align*}
\upchi^\sharp:\Wbb\rightarrow (\Mbb')^*=\ovl{\Mbb}\qquad w\mapsto \upchi(w\otimes-)
\end{align*}
By \eqref{eq59}, 
\begin{align*}
\upchi\big(L_j(0)w\otimes m'\big)=	\upchi\big(w\otimes L_j(0) m'\big)
	\end{align*}
This, together with Lem. \ref{lb31} (applied to $\Wbb\oplus\Mbb'$), implies for each $j$ and $\lambda\in\Cbb$ that
	\begin{align*}
\Lan \upchi^\sharp (P_j(\lambda)w),m' \Ran=	\upchi\big(P_j(\lambda)w\otimes m'\big)=\upchi\big(w\otimes P_j(\lambda) m'\big)=	\Lan \upchi^\sharp (w),P_j(\lambda) m'\Ran
		\end{align*}
		Thus, $\upchi^\sharp (P_j(\lambda)w)=P_j(\lambda)\upchi^\sharp(w)$. Therefore, $\upchi^\sharp$ has range in $\Mbb$. Let $T=\upchi^\sharp$. By \eqref{eq59} and \eqref{eq106}, $T$ belongs to $\Hom_{\Vbb^{\otimes N}}(\Wbb,\Mbb)$. Clearly $T^\flat=\upchi$. This proves that $\upchi$ is of the form \eqref{eq53}.
\end{proof}

Next, we describe conformal blocks in \eqref{eq55} and \eqref{eq56}.

\begin{rem}\label{lb17}
Let $\upchi:\boxtimes_\fn\Cbb\otimes \Wbb\otimes \Mbb'\rightarrow \Cbb$ (resp. $\updelta:\Wbb\otimes \boxtimes_\fn\Cbb\otimes \Mbb'\rightarrow \Cbb$) be a linear map. Then $\upchi\in \eqref{eq55}$ (resp. $\updelta\in \eqref{eq56}$) iff for each $v\in \Vbb, \psi\in \boxtimes_\fn\Cbb,w\in \Wbb,m'\in \Mbb',j\ne i$, the following relations hold in $\Cbb[[z^{\pm1}]]$:
\begin{subequations}\label{eq116}
\begin{gather}
	\upchi\big(\psi\otimes Y_j(v,z)w\otimes m'\big)=\upchi\big(\psi\otimes w\otimes Y'_j(v,z)m'\big)\\
	\upchi\big(\psi\otimes Y_i(v,z)w\otimes m'\big)=\upchi\big(Y_-'(v,z)\psi\otimes w\otimes m'\big)\\
	\upchi\big(Y_+(v,z)\psi\otimes w\otimes m'\big)=\upchi\big(\psi\otimes w\otimes Y_i'(v,z)m'\big)
\end{gather}
\end{subequations}
resp.
\begin{subequations}\label{eq117}
\begin{gather}
	\updelta\big(Y_j(v,z)w\otimes\psi\otimes m'\big)=\updelta\big(w\otimes\psi\otimes Y'_j(v,z)m'\big)\\
	\updelta\big(Y_i(v,z)w\otimes\psi\otimes m'\big)=\updelta\big(w\otimes Y_+'(v,z)\psi\otimes m'\big)\\
	\updelta\big(w\otimes Y_-(v,z)\psi\otimes m'\big)=\updelta\big(w\otimes\psi\otimes Y_i'(v,z)m'\big)
\end{gather}
\end{subequations}
\end{rem}

\begin{pp}\label{lb44}
Elements in \eqref{eq55} (resp. \eqref{eq56}) are precisely those of the form
\begin{gather*}
T_+^\flat:\boxtimes_\fn\Cbb\otimes\Wbb\otimes\Mbb'\rightarrow\Cbb\qquad \psi\otimes w\otimes w'\mapsto\bk{T_+(\psi\otimes w),w'}
\end{gather*}
resp.
\begin{gather*}
T_-^\flat:\Wbb\otimes\boxtimes_\fn\Cbb\otimes\Mbb'\rightarrow\Cbb\qquad w\otimes\psi\otimes w'\mapsto\bk{T_-(w\otimes\psi),w'}
\end{gather*}
where
\begin{align*}
T_+:\boxtimes_\fn\Cbb\otimes\Wbb\rightarrow\Mbb\qquad\text{resp.}\qquad T_-:\Wbb\otimes\boxtimes_\fn\Cbb\rightarrow\Mbb
\end{align*}
is a linear map such that for all $w\in\Wbb,v\in\Vbb,\psi\in\boxtimes_\fn\Cbb$ and $j\neq i$, the following relations hold in $\Mbb[[z^{\pm 1}]]$:
\begin{subequations}\label{eq118}
\begin{gather}
T_+\big(\psi\otimes Y_j(v,z)w\big)=Y_j(v,z)T_+\big(\psi\otimes w\big)\\
T_+\big(\psi\otimes Y_i(v,z)w\big)=T_+\big(Y'_-(v,z)\psi\otimes w\big)\\
T_+\big(Y_+(v,z)\psi\otimes w\big)=Y_i\big(v,z)T_+(\psi\otimes w\big)\label{eq118c}
\end{gather}
\end{subequations}
resp.
\begin{subequations}\label{eq119}
\begin{gather}
T_-\big(Y_j(v,z)w\otimes \psi\big)=Y_j(v,z)T_-\big(w\otimes\psi\big)\\
T_-\big(Y_i(v,z)w\otimes\psi\big)=T_-\big(w\otimes Y'_+(v,z)\psi\big)\\
T_-\big(w\otimes Y_-(v,z)\psi\big)=Y_i(v,z)T_-\big(w\otimes\psi\big)
\end{gather}
\end{subequations}
\end{pp}

\begin{proof}
It is clear that \eqref{eq116} and \eqref{eq117} are equivalent to \eqref{eq118} and \eqref{eq119}, respectively. Therefore, the only remaining step is to prove the following: if $\upchi$ and $\updelta$ satisfy the descriptions in Rem. \ref{lb17}, then, when viewing $\upchi$ as a linear map $T_+=\upchi^\sharp:\boxtimes_\fn\Cbb\otimes\Wbb\rightarrow\ovl\Mbb$, and viewing $\updelta$ as a linear map $T_-=\updelta^\sharp:\Wbb\otimes\boxtimes_\fn\Cbb\rightarrow\ovl\Mbb$, the ranges of both maps lie in $\Mbb$.

By \eqref{eq116} and Lem. \ref{lb31} (applied to $\boxtimes_\fn\Cbb\oplus\Wbb\oplus\Mbb'$), we have
\begin{gather*}
	\upchi\big(\psi\otimes P_j(\lambda)w\otimes w'\big)=\upchi\big(\psi\otimes w\otimes P_j(\lambda)w'\big)\\
	\upchi\big(\psi\otimes P_i(\lambda)w\otimes w'\big)=\upchi\big(P_-(\lambda)\psi\otimes w\otimes w'\big)\\
	\upchi\big(P_+(\lambda)\psi\otimes w\otimes w'\big)=\upchi\big(\psi\otimes w\otimes P_i(\lambda)w'\big)
\end{gather*}
That is, for each $\psi\in\boxtimes_\fn\Cbb,w\in\Wbb,\lambda\in\Cbb$ and $j\neq i$, the map $T_+=\upchi^\sharp$ satisfies
\begin{subequations}\label{eq64}
\begin{gather}
T_+\big(\psi\otimes P_j(\lambda)w\big)=P_j(\lambda)T_+\big(\psi\otimes w\big)\label{eq64a}\\
	T_+\big(\psi\otimes P_i(\lambda)w\big)=T_+\big(P_-(\lambda)\psi\otimes w\big)\label{eq64b}\\
T_+\big(P_+(\lambda)\psi\otimes w\big)=P_i(\lambda)T_+\big(\psi\otimes w\big)\label{eq64c}
\end{gather}
\end{subequations}
Similarly, by \eqref{eq117} and Lem. \ref{lb31}, the map $T_-=\updelta^\sharp$ satisfies
\begin{subequations}\label{eq65}
\begin{gather}
	T_-\big(P_j(\lambda)w\otimes\psi\big)=P_j(\lambda)T_-\big(w\otimes\psi\big)\label{eq65a}\\
	T_-\big(P_i(\lambda)w\otimes\psi\big)=T_-\big(w\otimes P_+(\lambda)\psi\big)\label{eq65b}\\
T_-\big(w\otimes P_-(\lambda)\psi\big)=P_i(\lambda)T_-\big(w\otimes\psi\big)\label{eq65c}
\end{gather}
\end{subequations}
This proves that $T_\pm$ have ranges in $\Mbb$.
\end{proof}

\begin{cv}\label{lb45}
In the pictures of conformal blocks, $T_\pm$ denote the conformal blocks $T_\pm^\flat$; equivalently, $\upchi^\sharp$ and $\updelta^\sharp$ denote the conformal blocks $\upchi$ and $\updelta$, respectively. Similarly, $T\in\Hom_{\Vbb^{\otimes N}}(\Wbb,\Mbb)$ denotes the conformal block $T^\flat:\Wbb\otimes\Mbb'\rightarrow\Cbb$ in \eqref{eq28} sending $w\otimes m'$ to $\bk{Tw,m'}$.
\end{cv}

We now focus on the case that $\Mbb=\Wbb$. Recall that $1\leq i\leq N$ is fixed.

\begin{df}\label{lb49}
By Rem. \ref{SF1}, there exist unique
\begin{gather}
\Phi_{i,+}^\flat\in \ST^*\Bigg(\vcenter{\hbox{{
		\includegraphics[height=2.8cm]{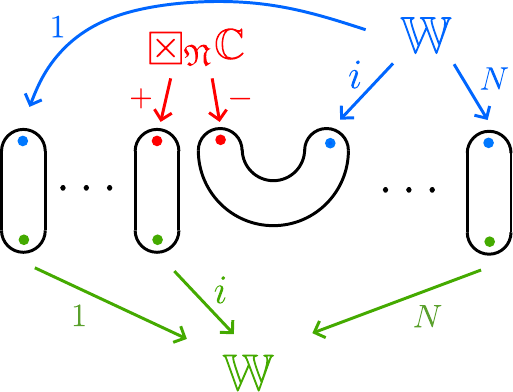}}}}~\Bigg) \qquad \Phi_{i,-}^\flat\in \ST^*\Bigg(\vcenter{\hbox{{
			\includegraphics[height=2.9cm]{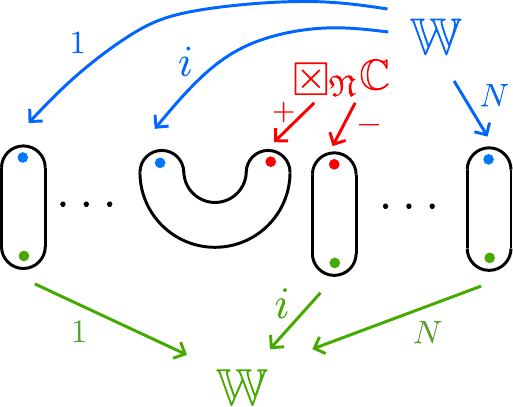}}}}~\Bigg)\label{eq54}
\end{gather}
such that when viewed as linear maps 
\begin{align}\label{eq120}
\Phi_{i,+}:\boxtimes_\fn\Cbb\otimes\Wbb\rightarrow\Wbb\qquad \Phi_{i,-}:\Wbb\otimes\boxtimes_\fn\Cbb\rightarrow\Wbb
\end{align}
(cf. Prop. \ref{lb44}), for each $w\in\Wbb,w'\in\Wbb'$ we have
\begin{subequations}\label{eq47}
\begin{align}\label{eq50}
\wick{\upomega(\c1 -)\cdot\big\langle\Phi_{i,+}(\c1 -\otimes w),w'\big\rangle}=\Lan w',w\Ran=\wick{\big\langle\Phi_{i,-}(w\otimes \c1 -),w'\big\rangle\cdot \upomega(\c1 -)}
 \end{align}
which is abbreviated to
\begin{align}\label{eq51}
	 \wick{\upomega(\c1 -)\cdot\Phi_{i,+}(\c1 -\otimes w)}=w=\wick{\Phi_{i,-}(w\otimes \c1 -)\cdot \upomega(\c1 -)}
  \end{align}
\end{subequations}
for each $w\in\Wbb$. Graphically, 
\begin{align}\label{eq48}
\begin{aligned}
&\vcenter{\hbox{{
		 \includegraphics[height=3cm]{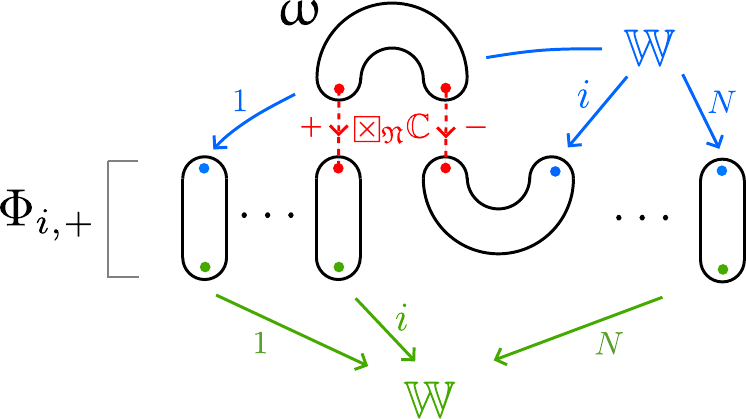}}}}~~~~=~~\vcenter{\hbox{{
			\includegraphics[height=2.8cm]{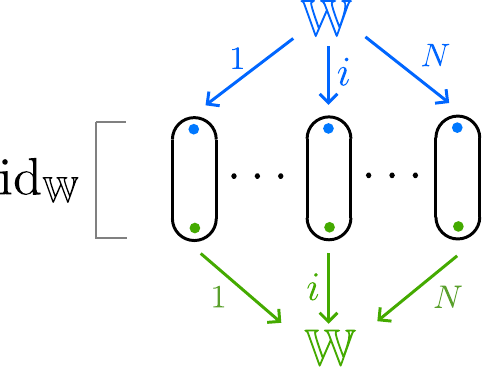}}}}\\[1ex]
=~~&\vcenter{\hbox{{
				   \includegraphics[height=3cm]{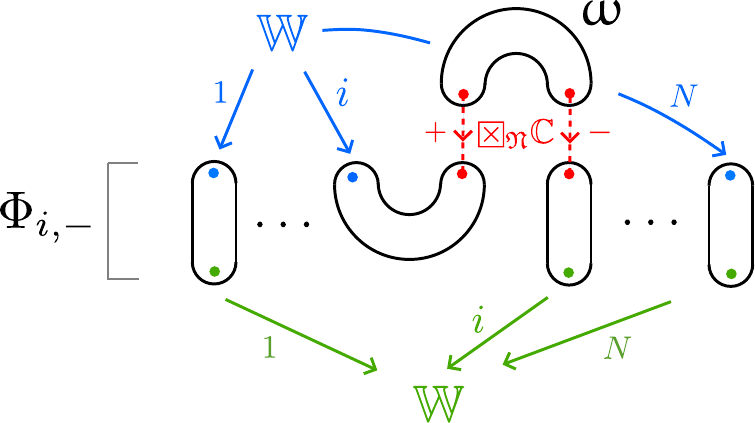}}}}
\end{aligned}
 \end{align}
Note that the sewing radii can be chosen to be admissible when all the sewing moduli are set to $1$ (cf. Subsec. \ref{lb46} for terminology). Therefore, the contractions in \eqref{eq47} converge absolutely (by Thm. \ref{lb47}).
\end{df}

\subsection{The canonical involution $\Theta$ of $\boxtimes_\fn\Cbb$}

We continue to fix $1\leq i\leq N$. In this section, we relate $\Phi_{i,+}$ and $\Phi_{i,-}$.

\begin{thm}\label{lb62}
The space of conformal blocks
\begin{align}\label{eq82}
\ST^*\Bigg(\vcenter{\hbox{{
				   \includegraphics[height=2.5cm]{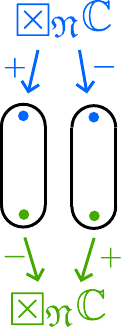}}}}~\Bigg)\quad=\quad
\ST^*\Bigg(\vcenter{\hbox{{
				   \includegraphics[height=2.5cm]{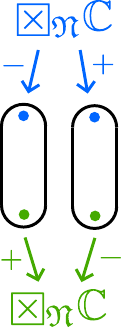}}}}~\Bigg)
\end{align}
consists of linear operators $T\in\End(\boxtimes_\fn\Cbb)$ satisfying
\begin{gather}\label{eq84}
TY_+(v)_n=Y_-(v)_nT\qquad TY_-(v)_n=Y_+(v)_nT\qquad\text{for all }v\in\Vbb,n\in\Zbb
\end{gather}
Moreover, there exists a unique $\Theta\in\eqref{eq82}$ whose transpose $\Theta^\tr\in\End(\bbs_\fn\Cbb)$ satisfies
\begin{align}\label{eq86}
\upomega=\upomega\circ\Theta^\tr
\end{align}
In addition, we have $\Theta^2=\id_{\boxtimes_\fn\Cbb}$, in particular, $\Theta$ is bijective. For each $\Wbb\in\Mod(\Vbb^{\otimes N})$ and $\psi\in\boxtimes_\fn\Cbb$, we have
\begin{align}\label{eq88}
\Phi_{i,+}(\psi\otimes w)=\Phi_{i,-}(w\otimes\Theta\psi)
\end{align}
\end{thm}

The map $\Theta$ is called the \textbf{canonical involution}  of $\boxtimes_\fn\Cbb$.

\begin{proof}
Let $\alpha:\{+,-\}\rightarrow\{+,-\}$ send $\pm$ to $\mp$. Recall Def. \ref{lb61} for the meaning of $\ovl\alpha(\boxtimes_\fn\Cbb)$. Then, by Prop. \ref{lb12}, we have
\begin{align}
\ST^*\Bigg(\vcenter{\hbox{{
				   \includegraphics[height=2.5cm]{fig24b.pdf}}}}~\Bigg)\quad=\quad
\ST^*\Bigg(\vcenter{\hbox{{
				   \includegraphics[height=2.5cm]{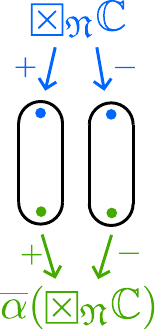}}}}~\Bigg)\quad\xlongequal{\text{Prop. \ref{lb19}}}\quad\Hom_{\Vbb^{\otimes 2}}(\boxtimes_\fn\Cbb,\ovl\alpha(\boxtimes_\fn\Cbb))
\end{align}
The latter hom space clearly consists of $T\in\End(\boxtimes_\fn\Cbb)$ satisfying \eqref{eq84}.

By Prop. \ref{lb12}, we have
\begin{align}
\ST^*\Bigg(\vcenter{\hbox{{
				   \includegraphics[height=1.6cm]{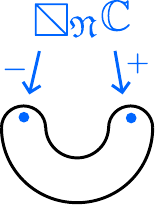}}}}~\Bigg)\quad=\quad
\ST^*\Bigg(\vcenter{\hbox{{
				   \includegraphics[height=1.6cm]{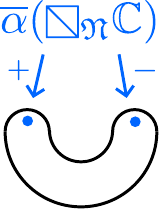}}}}~\Bigg)
\end{align}
Therefore, by Prop. \ref{lb13}, for each $\upphi$ belonging to the RHS above, there exists a unique $T\in\Hom_{\Vbb^{\otimes 2}}(\ovl\alpha(\bbs_\fn\Cbb),\bbs_\fn\Cbb)$ satisfying $\upphi=\upomega\circ T$. The existence and uniqueness of $\Theta$ satisfying \eqref{eq84} and \eqref{eq86} follow by letting $\upphi=\upomega$ and $\Theta=T^\tr$. Thus $\upomega=\upomega\circ(\Theta^\tr)^2$. Since $(\Theta^\tr)^2\in\End_{\Vbb^{\otimes 2}}(\bbs_\fn\Cbb)$, by the universal property for dual fusion products, we have $(\Theta^\tr)^2=\id$, and hence $\Theta^2=\id$.

By \eqref{eq84} and \eqref{eq86}, the map
\begin{align*}
\boxtimes_\fn\Cbb\otimes\Wbb\rightarrow\Wbb\qquad \psi\otimes w\mapsto\Phi_{i,-}(w\otimes\Theta\psi)
\end{align*}
satisfies the definition of $\Phi_{i,+}$ in Def. \ref{lb49}. This proves \eqref{eq88}. 
\end{proof}

\subsection{The  left and right actions $\Phi=\Phi_{+,+}$ and $\Psi=\Phi_{-,-}$ of $\boxtimes_\fn\Cbb$ on $\Wbb\in \Mod(\Vbb^{\otimes 2})$}
\label{sec2}
In this section, we assume $N=2$. Choose $\Wbb\in \Mod(\Vbb^{\otimes 2})$. 

\begin{df}
Let $\Phi=\Phi_{1,+}\equiv\Phi_{+,+}$ and $\Psi=\Phi_{2,-}\equiv\Phi_{-,-}$, that is, 
\begin{gather}
   \Phi=\Phi_{+,+}:\boxtimes_\fn\Cbb\otimes \Wbb\rightarrow \Wbb\qquad 
   \Psi=\Phi_{-,-}:\Wbb\otimes \boxtimes_\fn\Cbb\rightarrow \Wbb
\end{gather}
For each $\psi\in \boxtimes_\fn\Cbb$ and $w\in \Wbb$, write 
\begin{align}\label{eq38}
   \pmb{\psi\diamond_L w}:=\Phi(\psi\otimes w) \qquad \pmb{w\diamond_R \psi}:=\Psi(w\otimes \psi)
\end{align}
\end{df}

The figures representing the conformal blocks $\Phi,\Psi$ are
\begin{align}\label{eq123}
   \vcenter{\hbox{{
		\includegraphics[height=2.6cm]{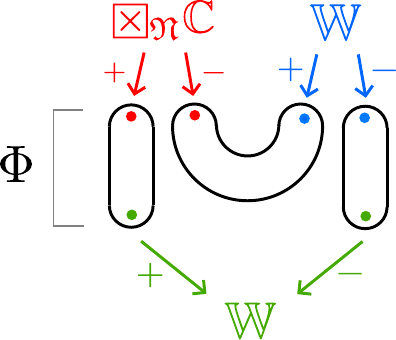}}}}\qquad \vcenter{\hbox{{
		   \includegraphics[height=2.6cm]{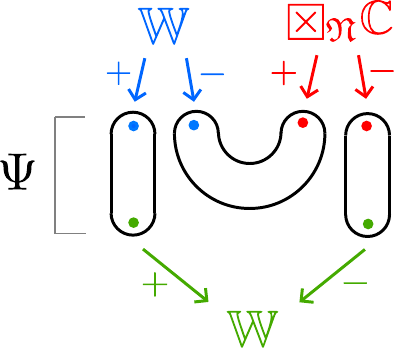}}}}
\end{align}
By Prop. \ref{lb44}, one can give an explicit and algebraic characterization of $\Phi$ and $\Psi$ being conformal blocks of the corresponding types. Specifically, for each $v\in \Vbb, \psi\in \boxtimes_\fn\Cbb,w\in \Wbb$, the following relations hold in $\Wbb[[z^{\pm1}]]$:
 \begin{gather*}
	\Phi(\psi\otimes Y_-(v,z)w)=Y_-(v,z)\Phi(\psi\otimes w)\\
\Phi(\psi\otimes Y_+(v,z)w)=\Phi(Y'_-(v,z)\psi\otimes w)\\
	\Phi(Y_+(v,z)\psi\otimes w)=Y_+(v,z)\Phi(\psi\otimes w)
 \end{gather*}
 resp. 
 \begin{gather*}
\Psi(Y_+(v,z)w\otimes \psi)=Y_+(v,z)\Psi(w\otimes\psi)\\
\Psi(Y_-(v,z)w\otimes \psi)=\Psi(w\otimes Y'_+(v,z)\psi)\\
	\Psi(w\otimes Y_-(v,z)\psi)=Y_-(v,z)\Psi(w\otimes \psi)
 \end{gather*}
By Def. \ref{lb49}, the conformal blocks $\Phi$ and $\Psi$ are determined by the fact that
\begin{gather}\label{eq58}
\wick{\upomega(\c1 -)\cdot\Phi(\c1 -\otimes w)}=w=\wick{\Psi(w\otimes \c1 -)\cdot \upomega(\c1 -)}
\end{gather}
holds for each $w\in\Wbb$. The picture for \eqref{eq58} is
\begin{gather}\label{eq34}
   \vcenter{\hbox{{
		\includegraphics[height=2.9cm]{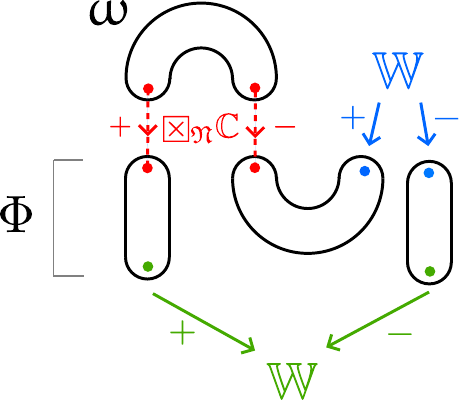}}}}\quad=\quad\vcenter{\hbox{{
		   \includegraphics[height=2.6cm]{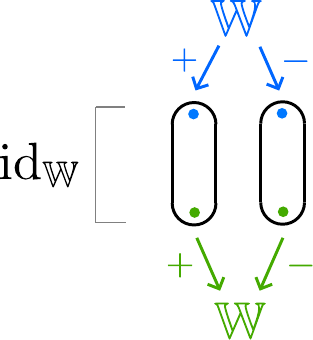}}}}\quad=\quad\vcenter{\hbox{{
			   \includegraphics[height=2.9cm]{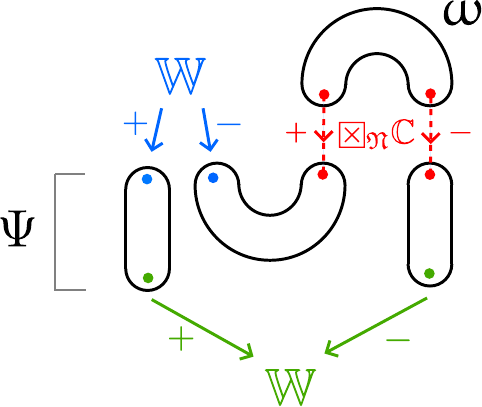}}}}
\end{gather}

\begin{comment}
By \eqref{eq64} and \eqref{eq65}, for each $\psi\in\boxtimes_\fn\Cbb,w\in\Wbb,\lambda\in\Cbb$, we have
 \begin{subequations}
 \begin{gather}
\Phi(\psi\otimes P_-(\lambda)w)=P_-(\lambda)\Phi(\psi\otimes w)\\
\Phi(\psi\otimes P_+(\lambda)w)=\Phi(P_-(\lambda)\psi\otimes w)\label\\
\Phi(P_+(\lambda)\psi\otimes w)=P_+(\lambda)\Phi(\psi\otimes w)
 \end{gather}
\end{subequations}
 resp. 
 \begin{subequations}
 \begin{gather}
\Psi(P_+(\lambda)w\otimes \psi)=P_+(\lambda)\Psi(\psi\otimes w)\\
\Psi(P_-(\lambda)w\otimes \psi)=\Psi(w\otimes P_+(\lambda)\psi)\\
	\Psi(w\otimes P_-(\lambda)\psi)=P_-(\lambda)\Psi(w\otimes \psi)
 \end{gather}
\end{subequations}
\end{comment}

\begin{pp}\label{lb16}
	For each $\psi_1,\psi_2\in \boxtimes_\fn\Cbb$ and $w\in \Wbb$, we have the associative law 
	\begin{align*}
		(\psi_1\diamond_L w)\diamond_R \psi_2=\psi_1\diamond_L(w\diamond_R \psi_2)
	\end{align*}
\end{pp}
Therefore, both sides can be denoted by
\begin{align*}
\psi_1\diamond w\diamond \psi_2
\end{align*}
\begin{proof}
We need to prove
	\begin{align}\label{eq37}
		\Psi\big(\Phi(\psi_1\otimes w)\otimes \psi_2\big)=\Phi\big(\psi_1\otimes \Psi(w\otimes \psi_2)\big)
	\end{align}
for all $\psi_1,\psi_2\in \boxtimes_\fn\Cbb$ and $w\in \Wbb$. Set 
	\begin{gather*}
		A:\boxtimes_\fn\Cbb\otimes \Wbb\otimes \boxtimes_\fn\Cbb\rightarrow \Wbb\qquad \psi_1\otimes w\otimes\psi_2\mapsto \Psi\big(\Phi(\psi_1\otimes w)\otimes \psi_2 \big)\\
		B:\boxtimes_\fn\Cbb\otimes \Wbb\otimes \boxtimes_\fn\Cbb\rightarrow \Wbb\qquad \psi_1\otimes w\otimes\psi_2\mapsto\Phi\big(\psi_1\otimes \Psi(w\otimes \psi_2)\big)
	\end{gather*}
In other words, $A$ and $B$ are defined by the contractions
   \begin{subequations}\label{eq122}
   \begin{gather}
	\label{eq32} A(\psi_1\otimes w\otimes \psi_2)=\wick{\Lan \Phi(\psi_1\otimes w),\c1 -\Ran \cdot \Psi(\c1 -\otimes \psi_2)}\\
	\label{eq33} B(\psi_1\otimes w\otimes \psi_2)=\wick{\Phi(\psi_1\otimes \c1 -)\cdot\Lan \c1 -,\Psi(w\otimes \psi_2)\Ran}
   \end{gather}
\end{subequations}
where the notation is similar to that in \eqref{eq51}. Thus, $A,B$  are obtained by composing conformal blocks, and are therefore themselves conformal blocks; see Fig. \ref{img1}.
\begin{figure}[h]
	\centering
\begin{gather*}
\vcenter{\hbox{{
		   \includegraphics[height=2.6cm]{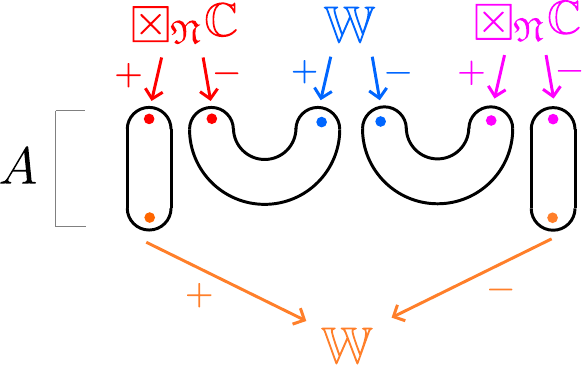}}}}	\quad=\quad\vcenter{\hbox{{
		\includegraphics[height=3.8cm]{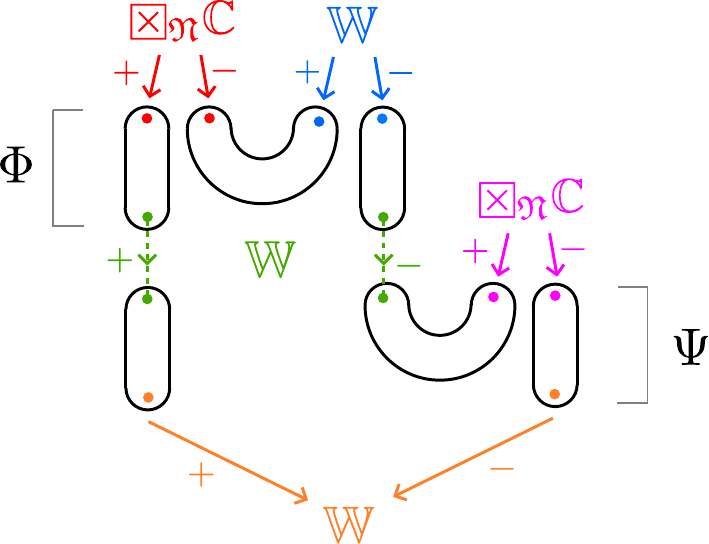}}}}\\[1ex]
\vcenter{\hbox{{
		   \includegraphics[height=2.6cm]{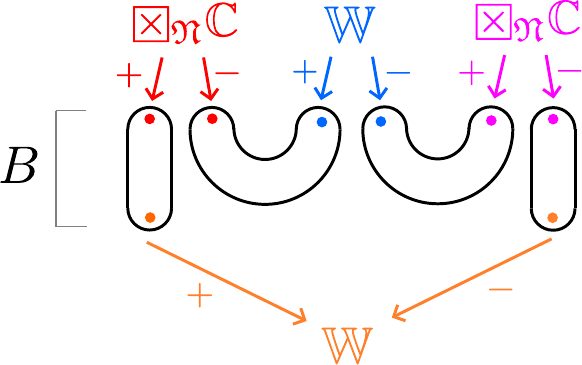}}}}	\quad=\quad	  \vcenter{\hbox{{
		\includegraphics[height=3.8cm]{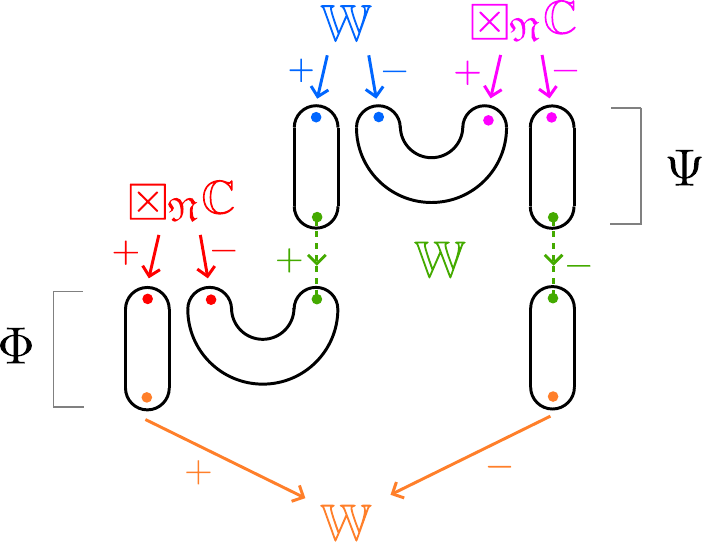}}}}
   \end{gather*}
\caption{~~The conformal blocks $A$ and $B$.}
	\label{img1}
\end{figure}

For each $w\in \Wbb$,
\begin{align}
&\wick{\upomega(\c1 -)A(\c1 -\otimes w\otimes \c1 -)\upomega(\c1 -)}\xlongequal{\eqref{eq122}}\wick{\upomega(\c1 -)\Lan \Phi(\c1 -\otimes w), \c1 -\Ran\cdot \Psi(\c1 -\otimes \c1 -)\upomega(\c1 -)}\nonumber\\
\xlongequal{\eqref{eq58}}&\wick{\Lan w,\c1 -\Ran\Psi(\c1 -\otimes \c1 -)\upomega(\c1 -)}\xlongequal{\eqref{eq58}}\wick{\Lan w,\c1 -\Ran\cdot \id_\Wbb(\c1-)}=w \label{eq35}
\end{align}
The picture for \eqref{eq35} is Fig. \ref{img2}. From this picture, it is evident that the sewing radii can be chosen to be admissible when all sewing moduli are set to $1$. Therefore, by Thm. \ref{lb47}, the contractions involved in each term of \eqref{eq35} are simultaneously converging absolutely. In particular, by Fubini's theorem for absolutely integrable functions, the order in which the contractions are performed does not affect the resulting values.\footnote{This reasoning for the commutativity of contractions will appear repeatedly in the remainder of the article. We will not refer to it explicitly each time.}
\begin{figure}[h]
	\centering
\begin{align*}
&\vcenter{\hbox{{
		   \includegraphics[height=4.1cm]{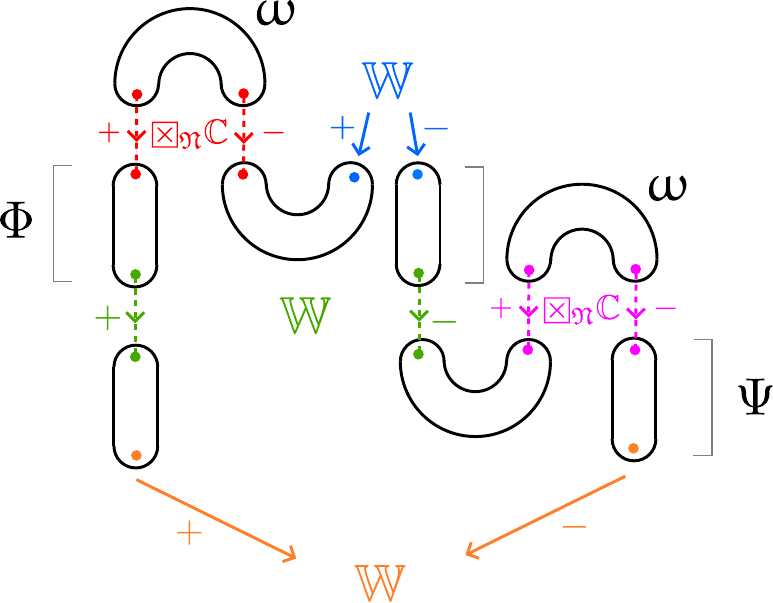}}}}\quad\xlongequal{\eqref{eq34}}\quad \vcenter{\hbox{{
		   \includegraphics[height=3.7cm]{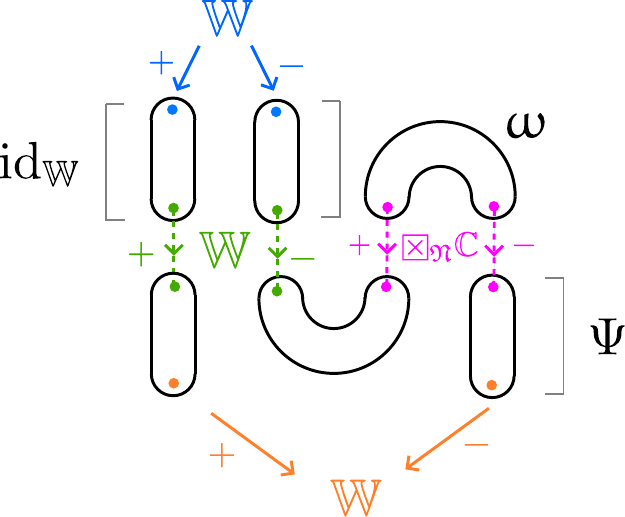}}}}\\[1ex]
\xlongequal{\eqref{eq34}}&\quad\vcenter{\hbox{{
		   \includegraphics[height=3.7cm]{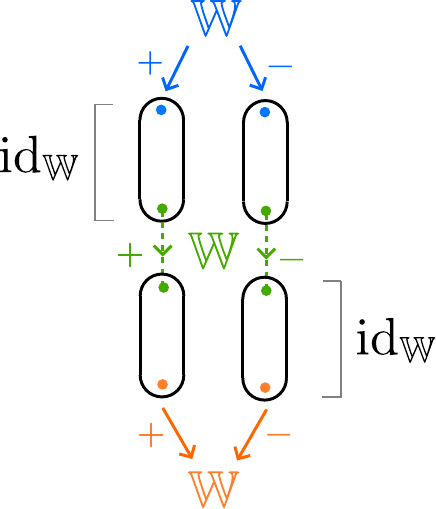}}}}\quad = \quad \vcenter{\hbox{{
		   \includegraphics[height=2.3cm]{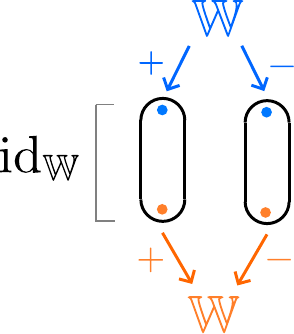}}}}
\end{align*}
\caption{~~The pictorial illustration of \eqref{eq35}.}
	\label{img2}
\end{figure}

By a similar argument, for each $w\in \Wbb$ we have
\begin{align*}
	\wick{\upomega(\c1 -)B(\c1 -\otimes w\otimes \c1 -)\upomega(\c1 -)}=w
\end{align*}
Applying twice the partial injectivity of the canonical conformal block $\upomega$ (cf. Rem. \ref{SF2}), we conclude that $A=B$. This proves \eqref{eq37}.
\end{proof}

\subsection{$\boxtimes_\fn\Cbb$ as an associative algebra over $\Cbb$}
In this section, we choose $\Wbb=\boxtimes_\fn\Cbb\in \Mod(\Vbb^{\otimes 2})$. Then \eqref{eq123} becomes
\begin{align}\label{eq39}
   \vcenter{\hbox{{
		\includegraphics[height=2.6cm]{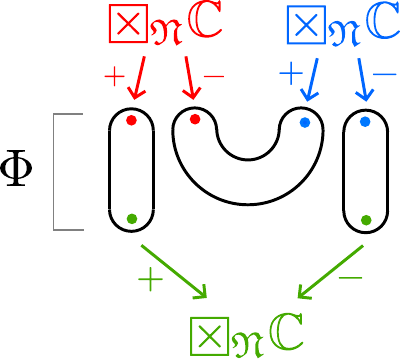}}}}\qquad \vcenter{\hbox{{
		   \includegraphics[height=2.6cm]{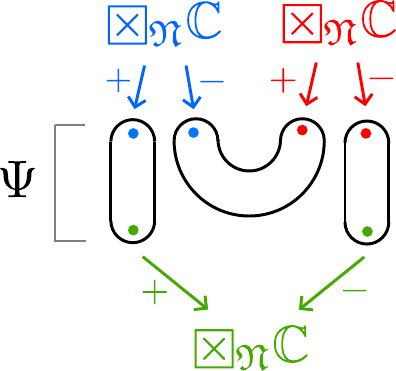}}}}
\end{align}

Following \eqref{eq38}, we have 
\begin{align}\label{eq40}
	\psi_1\diamond_L \psi_2:=\Phi(\psi_1\otimes \psi_2), \quad \psi_1\diamond_R \psi_2:=\Psi(\psi_1\otimes \psi_2).
 \end{align}
 for each $\psi_1,\psi_2\in \boxtimes_\fn\Cbb$.
\begin{pp}\label{lb15}
	For each $\psi_1,\psi_2\in \boxtimes_\fn\Cbb$, we have $\psi_1\diamond_L \psi_2=\psi_1\diamond_R \psi_2$. Therefore, we denote both $\diamond_L$ and $\diamond_R$ by $\diamond$.
\end{pp}

\begin{proof}
We compute for each $\psi'\in\bbs_\fn\Cbb$ that
	\begin{align}\label{eq126}
\Lan\psi',\wick{\upomega(\c1-)\Phi(\c1 -\otimes \c1-)\cdot \upomega(\c1 -)}\Ran\xlongequal{\eqref{eq58}}\Lan\psi',\wick{\id_{\boxtimes_\fn\Cbb}(\c1-)\cdot \upomega(\c1 -)}\Ran=\upomega(\psi')
	\end{align}
The picture of \eqref{eq126} is
\begin{align}\label{eq127}
		\vcenter{\hbox{{
			\includegraphics[height=2.9cm]{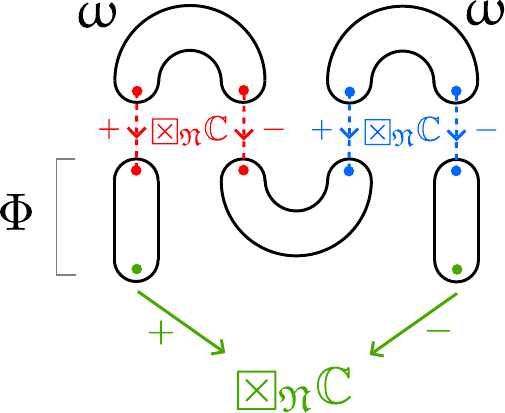}}}}\quad\xlongequal{\eqref{eq34}}\quad \vcenter{\hbox{{
			   \includegraphics[height=2.9cm]{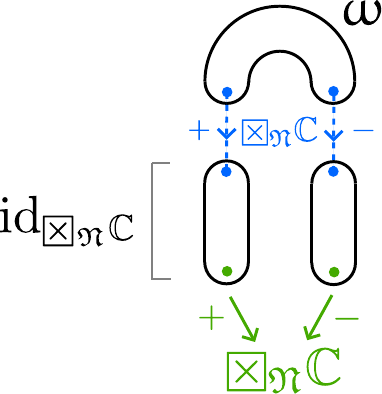}}}}\quad=\quad\vcenter{\hbox{{
				\includegraphics[height=1.6cm]{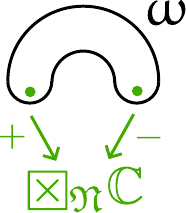}}}}
	\end{align}
In the pictures, we do not distinguish between $\upomega$ and $\upomega^\sharp$. By a similar argument, 
\begin{align*}
\Lan\psi',\wick{\upomega(\c1-)\Psi(\c1 -\otimes \c1-)\cdot \upomega(\c1 -)}\Ran=\upomega(\psi')
\end{align*}
Therefore, applying twice the partial injectivity of $\upomega$ (Rem. \ref{SF2}), we conclude $\Phi=\Psi$.
\end{proof}

\begin{co}\label{lb82}
The complex vector space $\boxtimes_\fn\Cbb$, together with the operation $\diamond$, defines a (not necessarily unital) associative $\Cbb$-algebra $(\boxtimes_\fn\Cbb,\diamond)$.
\end{co}
\begin{proof}
By Prop. \ref{lb16} and \ref{lb15}, $\diamond$ satisfies the associativity. This fact, together with the linearity of $\Phi$ and $\Psi$, proves that $(\boxtimes_\fn\Cbb,\diamond)$ is an associative $\Cbb$-algebra.
\end{proof}

\subsection{The $\boxtimes_\fn\Cbb$-module structures on $\Wbb\in\Mod(\Vbb^{\otimes N})$}

Fix $\Wbb\in\Mod(\Vbb^{\otimes N})$ and $1\leq i\leq N$.

\begin{thm}\label{lb50}
The linear map $\Phi_{i,+}:\boxtimes_\fn\Cbb\otimes\Wbb\rightarrow\Wbb$ defines a left $\boxtimes_\fn\Cbb$-module structure on $\Wbb$, and the linear map $\Phi_{i,-}:\Wbb\otimes\boxtimes_\fn\Cbb\rightarrow\Wbb$ defines a right $\boxtimes_\fn\Cbb$-module structure on $\Wbb$. In other words, for each $\psi_1,\psi_2\in\boxtimes_\fn\Cbb$ and $w\in\Wbb$, we have
\begin{gather}
\Phi_{i,+}\big((\psi_1\diamond\psi_2)\otimes w\big)=\Phi_{i,+}\big(\psi_1\otimes\Phi_{i,+}(\psi_2\otimes w)\big)\label{eq124}\\
\Phi_{i,-}\big(w\otimes(\psi_2\diamond\psi_1)\big)=\Phi_{i,-}\big(\Phi_{i,-}(w\otimes\psi_2)\otimes\psi_1\big)\label{eq125}
\end{gather}
\end{thm}

\begin{proof}
We only prove \eqref{eq124}, as \eqref{eq125} can be proved in a similar way. We draw \eqref{eq28} as
\begin{align}
\ST^*\Bigg(\vcenter{\hbox{{
			\includegraphics[height=2.5cm]{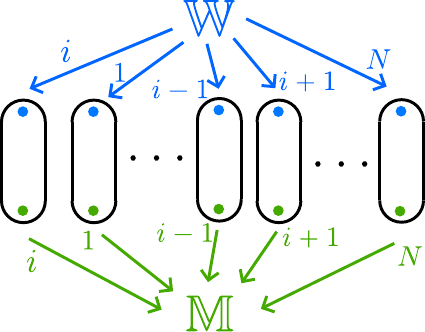}}}}~\Bigg)\xlongequal{\text{abbreviate}}\ST^*\Bigg(\vcenter{\hbox{{
			\includegraphics[height=2.5cm]{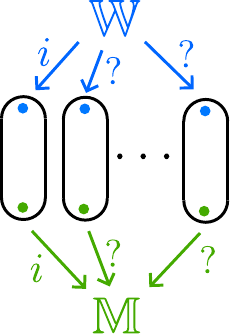}}}}~\Bigg)
\end{align}
That is, we rearrange of order of the spheres in \eqref{eq28} from $1,2,\dots,N$ to $i,1,\dots,i-1,i+1,\cdots,N$, and the labels $1,\dots,i-1,i+1,\dots,N$ on the respective arrows are abbreviated to the symbol $?$. Of course, in this proof we set $\Mbb=\Wbb$. Set 
\begin{gather*}
A:\boxtimes_\fn\Cbb\otimes \boxtimes_\fn\Cbb\otimes \Wbb\rightarrow \Wbb\qquad \psi_1\otimes\psi_2\otimes w\mapsto \Phi_{i,+}(\Phi_{+,+}(\psi_1\otimes\psi_2)\otimes w)\\
B:\boxtimes_\fn\Cbb\otimes \boxtimes_\fn\Cbb\otimes \Wbb\rightarrow \Wbb\qquad  \psi_1\otimes\psi_2\otimes w \mapsto \Phi_{i,+}(\psi_1\otimes \Phi_{i,+}(\psi_2\otimes w))
\end{gather*}
(Recall that $\Phi_{+,+}(\psi_1\otimes\psi_2)=\psi_1\diamond\psi_2$.) Since $A,B$ are defined by composing conformal blocks, they are themselves conformal blocks. See Fig. \ref{img3}.
\begin{figure}[h]
	\centering
\begin{gather*}
\vcenter{\hbox{{
		   \includegraphics[height=2.7cm]{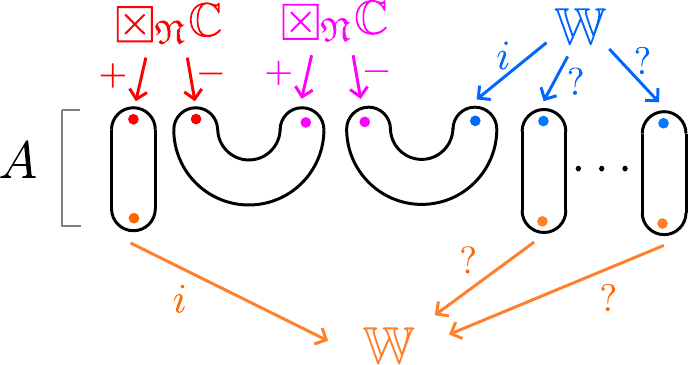}}}}	\quad=\quad\vcenter{\hbox{{
		\includegraphics[height=3.9cm]{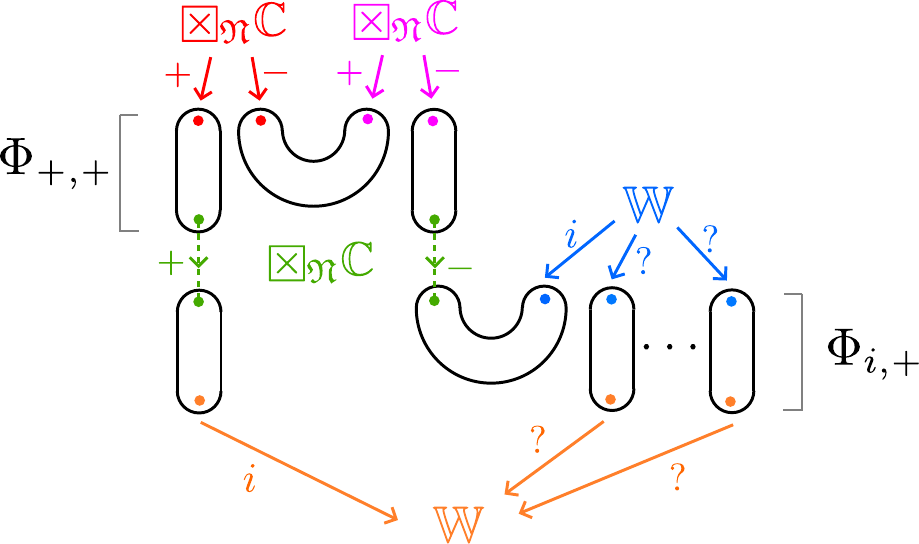}}}}\\[1ex]
\vcenter{\hbox{{
		   \includegraphics[height=2.7cm]{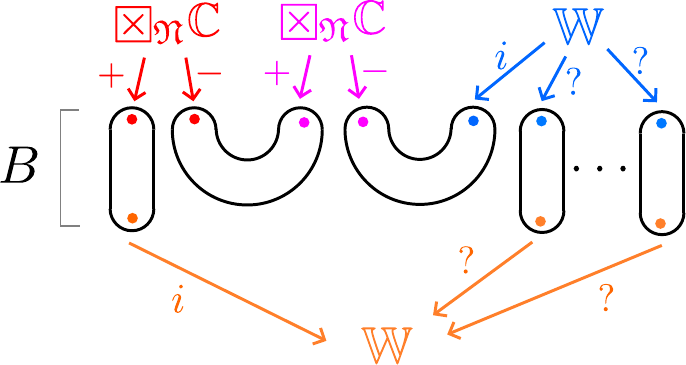}}}}	\quad=\quad	  \vcenter{\hbox{{
		\includegraphics[height=3.9cm]{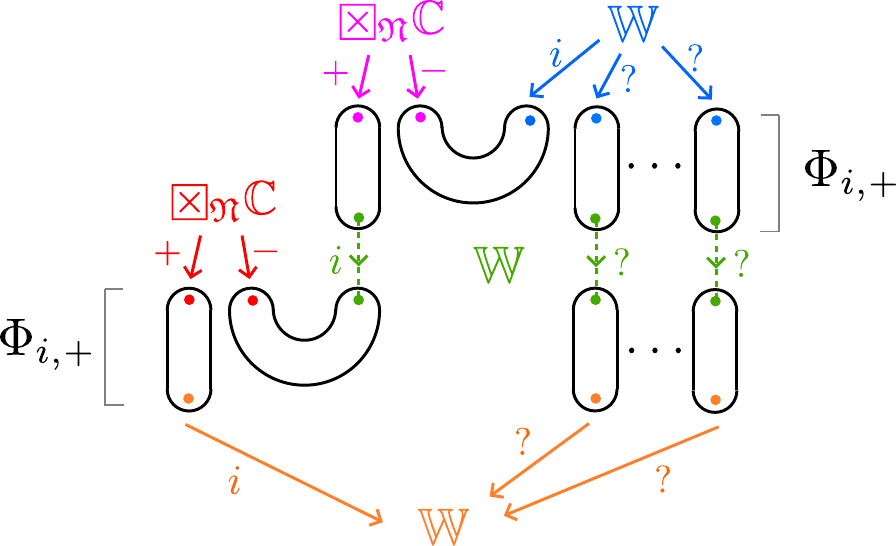}}}}
   \end{gather*}
\caption{~~The conformal blocks $A$ and $B$.}
	\label{img3}
\end{figure}

For each $w\in\Wbb$, we compute that
\begin{align}
&\wick{\upomega(\c1 -) A(\c1 -\otimes \c1 -\otimes w)\upomega(\c1 -)}=\wick{\upomega(\c1 -)\Phi_{i,+}(\Phi_{+,+}(\c1 -\otimes \c1 -)\otimes w)\upomega(\c1 -)}\nonumber\\
=&\wick{\upomega(\c1-)\langle\Phi_{+,+}(\c1-\otimes\c1-),\c2-\rangle\upomega(\c1-)\Phi_{i,+}(\c2-\otimes w) }\xlongequal{\eqref{eq126}}\wick{\upomega(\c1 -)\Phi_{i,+}(\c1 -\otimes w)}\xlongequal{\eqref{eq47}}w  \label{eq128}
	\end{align}
See Fig. \ref{img4} for the picture. We also compute that
\begin{align}
&\wick{\upomega(\c1 -) B(\c1 -\otimes \c1 -\otimes w)\upomega(\c1 -)}=\wick{\upomega(\c1 -)\Phi_{i,+}(\c1 -\otimes \Phi_{i,+}(\c1 -\otimes w))\upomega(\c1 -)}\nonumber\\
\xlongequal{\eqref{eq47}}&\wick{\upomega(\c1 -)\Phi_{i,+}(\c1 -\otimes w)}\xlongequal{\eqref{eq47}}w  \label{eq129}
\end{align}
See Fig. \ref{img5}. The figures show that the sewing radii can be chosen to be admissible when all the sewing moduli are set to $1$. Therefore, by Thm. \ref{lb47}, the contractions in \eqref{eq128} and \eqref{eq129} converge absolutely.

\begin{figure}[H]
	\centering
\begin{align*}
&\vcenter{\hbox{{
		   \includegraphics[height=4.3cm]{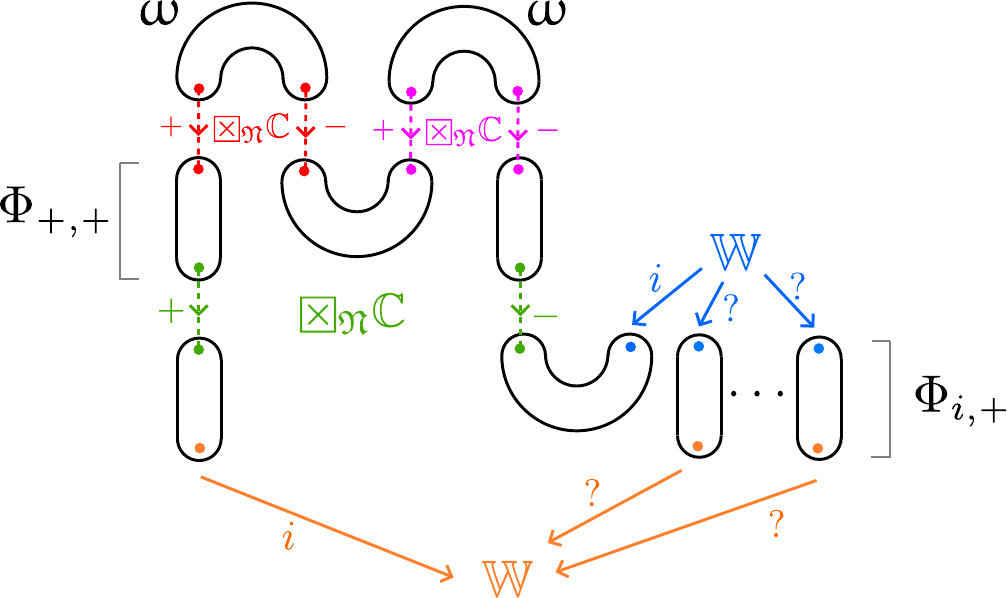}}}}\\[1ex]
&\xlongequal{\eqref{eq127}}\quad
\vcenter{\hbox{{
		   \includegraphics[height=3cm]{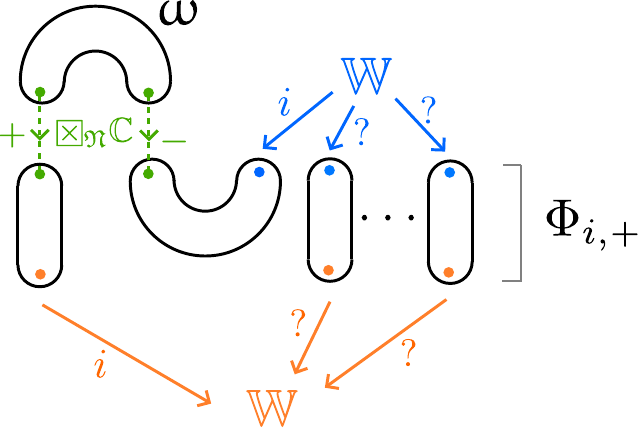}}}}
\quad\xlongequal{\eqref{eq48}}\quad \vcenter{\hbox{{
		   \includegraphics[height=2.4cm]{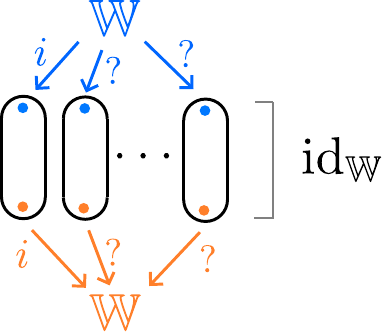}}}}
\end{align*}
\caption{~~The pictorial illustration of \eqref{eq128}.}
	\label{img4}
\end{figure}

\begin{figure}[H]
	\centering
\begin{align*}
&\vcenter{\hbox{{
		   \includegraphics[height=4.3cm]{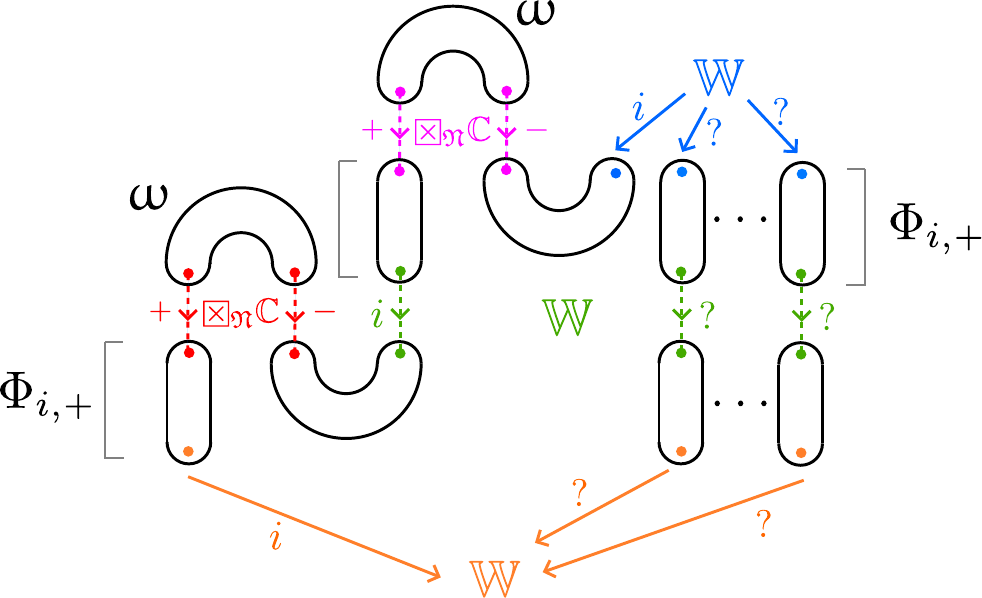}}}}
\quad\xlongequal{\eqref{eq48}}\quad\vcenter{\hbox{{
		   \includegraphics[height=3.8cm]{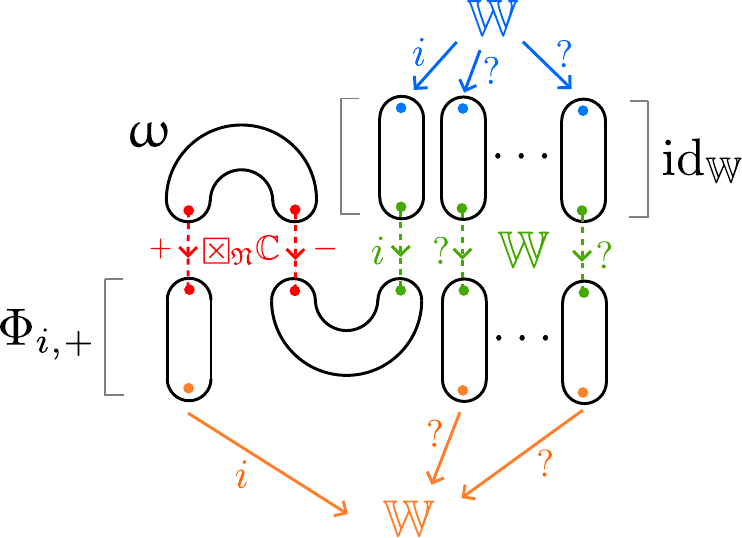}}}}\\[1ex]
&\xlongequal{\eqref{eq48}}\quad
\vcenter{\hbox{{
		   \includegraphics[height=3.8cm]{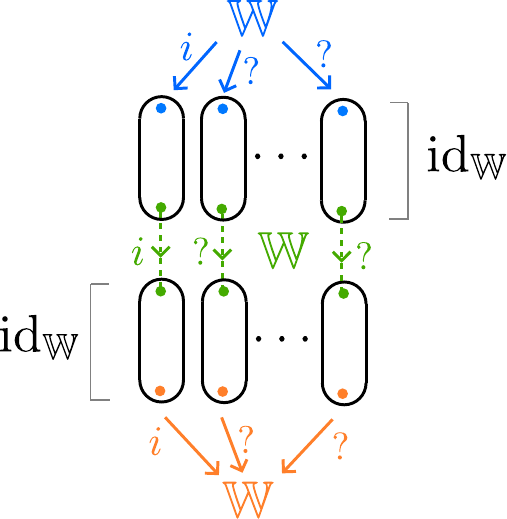}}}}
\quad=\quad \vcenter{\hbox{{
		   \includegraphics[height=2.4cm]{fig21c.pdf}}}}
\end{align*}
\caption{~~The pictorial illustration of \eqref{eq129}.}
	\label{img5}
\end{figure}
Applying twice the partial injectivity of the canonical conformal block $\upomega$ (cf. Rem. \ref{SF2}), we conclude $A=B$.
\end{proof}

\begin{co}\label{lb26}
Let $\Wbb\in \Mod(\Vbb^{\otimes 2})$. Then $\Wbb$ is a $\boxtimes_\fn\Cbb$-bimodule whose left and right module structures are defined by
\begin{gather*}
\boxtimes_\fn\Cbb\times \Wbb\rightarrow\Wbb\qquad (\psi,w)\mapsto \psi\diamond_L w\\
\Wbb\times\boxtimes_\fn\Cbb\rightarrow\Wbb \qquad (w,\psi)\mapsto w\diamond_R \psi
\end{gather*}
\end{co}

\begin{proof}
This follows immediately from Prop. \ref{lb16} and Thm. \ref{lb50}.
\end{proof}

\subsection{The AUF algebra $\boxtimes_\fn\Cbb$ with involution $\Theta$}

Fix $1\leq i\leq N$. In this section, we describe the actions of $\boxtimes_\fn\Cbb$ on $\Wbb\in\Mod(\Vbb^{\otimes N})$ in terms of vertex operators; see Thm. \ref{lb21}. This description will allow us to relate linear operators that intertwine the action of $\Vbb$ with those that intertwine the action of $\boxtimes_\fn\Cbb$. It will also be useful in showing that the associative algebra $\boxtimes_\fn\Cbb$ is AUF.

Recall that $\zeta$ is the standard coordinate of $\Cbb$.

\begin{df}\label{lb23}
For each $z\in\Cbb^\times=\Cbb\setminus\{0\}$, let $\fq_z$ be the propagation of the $(2,0)$-pointed sphere $\fn=\eqref{eq18}$ at $z$ with local coordinate $\zeta-z$, cf. Def. \ref{lb37}. Namely, $\fq_z$ is the unordered $(2,1)$-pointed sphere with local coordinates
	\begin{align*}
\fq_z=\big(\{\infty,0\};1/\zeta,\zeta\big|\Pbb^1\big|z;\zeta-z\big)
	\end{align*}
Choose orderings
   \begin{gather*}
\epsilon:\{+,-\}\xlongrightarrow{\simeq}\{\infty,0\}\qquad \epsilon(+)=\infty\qquad \epsilon(-)=0\\
	\iota_z:\{1\}\rightarrow \{z\}\qquad 1 \mapsto z
   \end{gather*}
where $\epsilon$ is the default ordering of $\{\infty,0\}$ (cf. Def. \ref{lb51}). By Def. \ref{lb37} and Thm. \ref{lb36}, 
\begin{align*}
\big(\boxtimes_\fn\Cbb,\aleph_z\big)
\end{align*}
is an $(\epsilon,\iota_z)$-fusion product of $\Vbb$ along $\fq_z$, where $\aleph_z$ is the unique element in
\begin{align*}
\ST^*_{\fq_z,\epsilon*\iota_z}(\bbs_\fn\Cbb\otimes\Vbb)=  
		\ST^*\Bigg(\vcenter{\hbox{{
			\includegraphics[height=2.1cm]{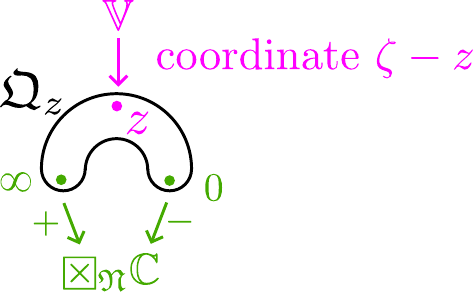}}}}~\Bigg)
\end{align*}
satisfying
\begin{align}\label{eq131}
\aleph_z:\bbs_\fn\Cbb\otimes\Vbb\rightarrow\Cbb\qquad \aleph_z(-\otimes\idt)=\upomega
\end{align}
Define 
\begin{align}\label{eq79}
	\aleph_z^\sharp:\Vbb\rightarrow \ovl{\boxtimes_\fn\Cbb}=(\bbs_\fn\Cbb)^*\qquad v\mapsto\aleph_z(-\otimes v)
\end{align}
where $\aleph_z(-\otimes v)$ denotes the linear map $\bbs_\fn\Cbb\rightarrow\Cbb$ sending each $\psi'\in\bbs_\fn\Cbb$ to $\aleph_z(\psi'\otimes v)$.
\end{df}

\begin{cv}\label{lb60}
As in Convention \ref{lb45}, we will not distinguish between $\aleph_z$ and $\aleph_z^\sharp$ in the graphical representations of conformal blocks. Moreover, only in this section, the local coordinates at $\infty$ and $0$ are always assumed to be $1/\zeta$ and $\zeta$, respectively.
\end{cv}

\begin{rem}
The map $\tipath:t\in\Pbb^1\mapsto 1/t\in\Pbb^1$ implements an isomorphism between $\fk Q_z$ and
\begin{align*}
\fk Q_{1/z}'=\big(\{\infty,0\};1/\zeta,\zeta\big|\Pbb^1\big|1/z;1/\zeta-z\big)
\end{align*}
Therefore, by Prop. \ref{lb8},
\begin{align*}
\big(\boxtimes_\fn\Cbb,\aleph_z\big)
\end{align*}
is an $(\tipath\circ\epsilon,\iota_{1/z})$-fusion product of $\Vbb$ along $\fk Q_{1/z}'$, where $\aleph_z$ belongs to
\begin{align*}
\ST^*\Bigg(\vcenter{\hbox{{
			\includegraphics[height=2.1cm]{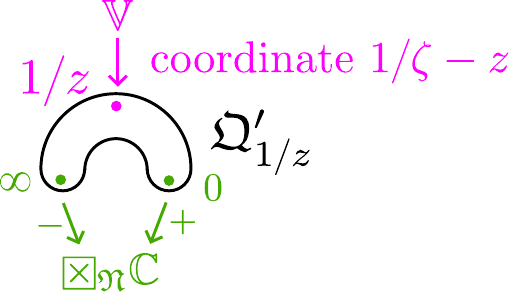}}}}~\Bigg)
\end{align*}
\end{rem}

\begin{lm}\label{lb24}
Fix $z\in\Cbb^\times$. Then for each $\lambda,\mu\in \Cbb$, the linear map
	\begin{gather}
		\label{eq81}P_+({\leq\lambda})\circ P_-({\leq\mu})\circ\aleph_z^\sharp:\Vbb\rightarrow \boxtimes_\fn\Cbb_{[\leq(\lambda,\mu)]}=\big(\bbs_\fn\Cbb_{[\leq(\lambda,\mu)]}\big)^*
	\end{gather}
	is surjective.
\end{lm}

Note that $P_+({\leq\lambda})$ commutes with $P_-({\leq\mu})$, and $P_+(\lambda)$ commutes with $P_-(\mu)$. (Recall \eqref{eq130} for the meanings of these projection operators.)

\begin{proof}
	Note that $\bbs_\fn\Cbb_{[\leq(\lambda,\mu)]}$ is finite dimensional. If \eqref{eq81} is not surjective, then there exists $0\ne \upphi\in \bbs_\fn\Cbb_{[\leq(\lambda,\mu)]}$ such that for all $v\in V$, 
	\begin{align*}
		\Lan P_+(\leq\lambda)P_-(\leq\mu)\aleph_z^\sharp(v),\upphi\Ran=0,
	\end{align*}
	which is equivalent to $\aleph_z(v\otimes \upphi)=0$. This contradicts the partial injectivity of the canonical conformal block $\aleph_z$ (cf. Rem. \ref{SF2}). Thus \eqref{eq81} must be surjective. 
\end{proof}

\begin{rem}\label{lb30}
Fix $z\in\Cbb^\times$. Then by Lem. \ref{lb24}, each $\psi\in \boxtimes_\fn\Cbb$ can be written as a (finite) linear combination of elements of the form $P_+(\lambda)P_-(\mu)\aleph_z^\sharp(v)$ where $\lambda,\mu\in \Cbb$ and $v\in \Vbb$.
\end{rem}

\begin{thm}\label{lb21}
Let $z\in\Cbb^\times$ and $\Wbb\in\Mod(\Vbb^{\otimes N})$. Then for each $\lambda,\mu\in \Cbb$, $w\in \Wbb$ and $v\in \Vbb$, we have 
\begin{subequations}\label{eq68}
\begin{gather}
\Phi_{i,+}\big(P_+(\lambda)P_-(\mu)\aleph_z^\sharp (v)\otimes w\big) =P_i(\lambda)Y_i(v,z)P_i(\mu)w  \label{eq68a}\\
\Phi_{i,-}\big(w\otimes P_+(\lambda)P_-(\mu)\aleph_z^\sharp (v)\big) =P_i(\mu)Y'_i(v,z)P_i(\lambda)w  \label{eq68b}
\end{gather}
\end{subequations}
\end{thm}

Clearly, these identities still hold if $\lambda$ and $\mu$ are replaced with $\leq\lambda$ and $\leq\mu$.

\begin{proof}
Step 1.	We claim that for each $w\in \Wbb,w'\in \Wbb',v\in \Vbb$ we have
\begin{subequations}\label{eq62}
\begin{gather}
\wick{\Lan w',\Phi_{i,+}(\c1 -\otimes w)\Ran \aleph_z(\c1 - \otimes v)}=\Lan w',Y_i(v,z)w \Ran\label{eq62a}\\
\wick{\Lan w',\Phi_{i,-}(w\otimes \c1 -)\Ran \aleph_z(\c1 - \otimes v)}=\Lan w',Y'_i(v,z)w \Ran\label{eq62b}
\end{gather}
\end{subequations}
We first prove \eqref{eq62a}. Note that \eqref{eq62a} is equivalent to Fig. \ref{img6}; in particular, both sides of \eqref{eq62a} define conformal blocks associated to geometric data on the respective side of Fig. \ref{img6}. When $v$ is the vacuum vector $\idt$, the right hand side of \eqref{eq62a} becomes $\bk{w',w}$ (i.e., the conformal block $\id_\Wbb$); by \eqref{eq47} (whose pictorial illustration is \eqref{eq48}) and \eqref{eq131}, the left hand side also becomes $\bk{w',w}$. Therefore, by the propagation of conformal blocks (cf. \cite[Cor. 2.44]{GZ1}), \eqref{eq62a} holds true.
\begin{figure}[h]
	\centering
\begin{align*}
		\vcenter{\hbox{{
			\includegraphics[height=3.6cm]{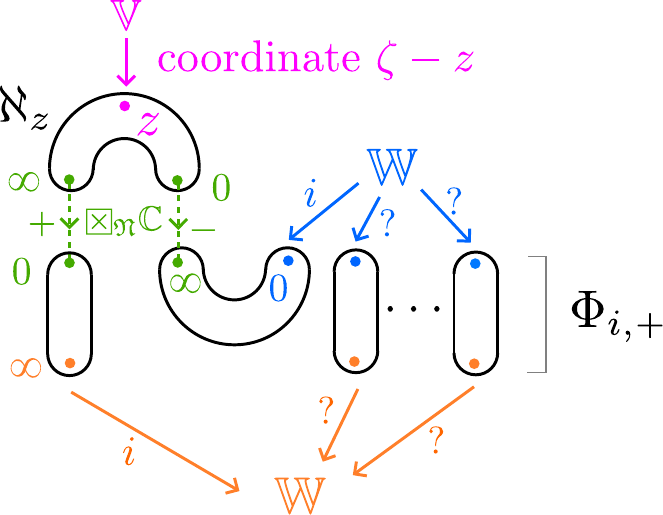}}}}\quad=\quad\vcenter{\hbox{{
				\includegraphics[height=2.4cm]{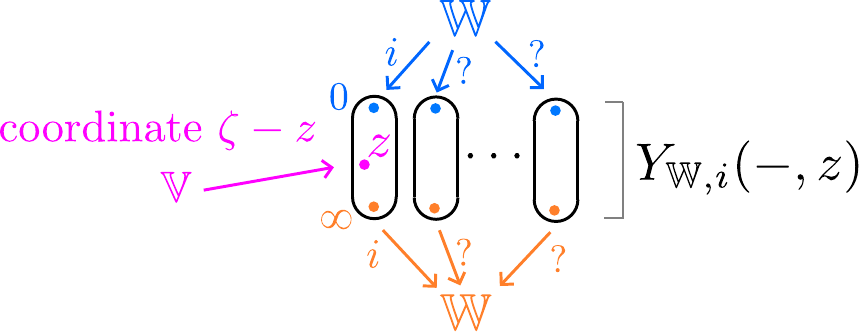}}}}
	 \end{align*}
\caption{~~The figure for \eqref{eq62a}}
	\label{img6}
\end{figure}

Similarly, \eqref{eq62b} is equivalent to Fig. \ref{img7}. Here, we need to check that the RHS of \eqref{eq62b} defines a conformal block associated to the geometric data given on the RHS of Fig. \ref{img7}. Once this is proved, by applying again the propagation of conformal blocks, we obtain \eqref{eq62b}. (Note again that in Fig. \ref{img6} and \ref{img7}, the sewing radii can be chosen to be admissible when all the sewing moduli are set to $1$. Therefore, the contractions converge absolutely.)
\begin{figure}[h]
	\centering
\begin{align*}
		\vcenter{\hbox{{
			\includegraphics[height=3.6cm]{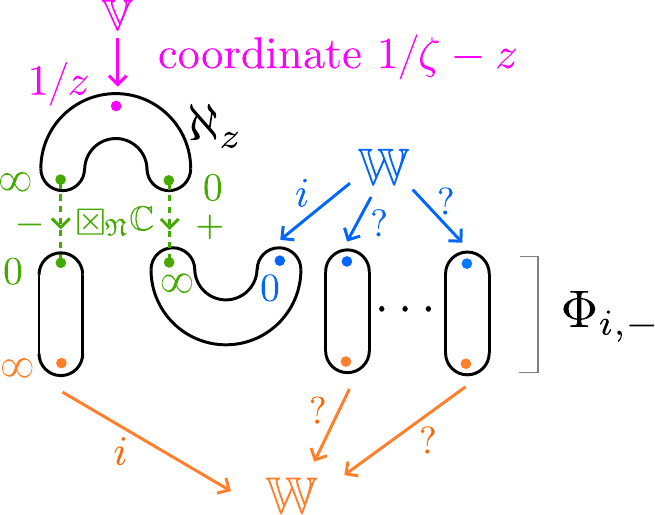}}}}\quad=\quad\vcenter{\hbox{{
				\includegraphics[height=2.4cm]{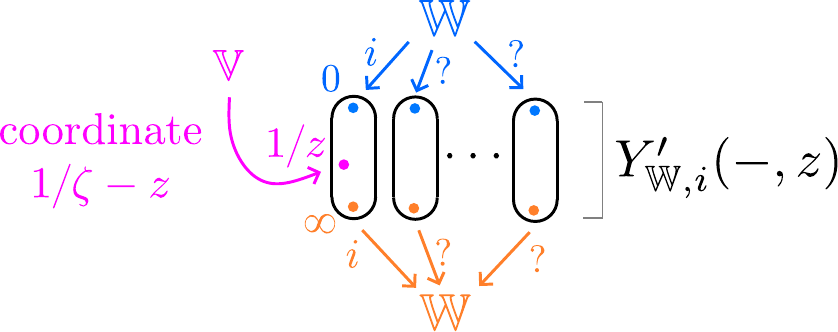}}}}
	 \end{align*}
\caption{~~The figure for \eqref{eq62b}}
	\label{img7}
\end{figure}

If the local coordinate at $1/z$ is $\zeta-1/z$ instead of $1/\zeta-z$, then 
\begin{gather*}
\Wbb\otimes\Vbb\otimes\Wbb'\rightarrow\Cbb\qquad w\otimes v\otimes w'\mapsto\bk{w',Y_i(v,1/z)w}
\end{gather*}
defines a conformal block. Now, the element $\upgamma_z\in\MG$ (cf. \eqref{eq132}) transforms $1/\zeta-z$ to $\zeta-1/z$, that is, $\upgamma_z\circ(1/\zeta-z)=\zeta-1/z$. By Prop. \ref{lb65}, the linear functional
\begin{gather*}
\Wbb\otimes\Vbb\otimes\Wbb'\rightarrow\Cbb\qquad w\otimes v\otimes w'\mapsto\bk{w',Y_i(\MU(\upgamma_z)v,1/z)w}=\bk{w',Y_i'(v,z)w}
\end{gather*}
defines a conformal block associated to the geometric data on the RHS of Fig. \ref{img7}.\\[-1ex]

Step 2. By the definition of contractions (cf. \eqref{eq57}) and the map $\aleph_z^\sharp$, we can write \eqref{eq62a} (equivalently, Fig. \ref{img6}) as 
	\begin{align}\label{eq63}
\sum_{\lambda,\mu\in \Cbb}\Lan w',\Phi_{i,+}(P_+(\lambda)P_-(\mu)\aleph_z^\sharp(v)\otimes w)\Ran=\Lan w',Y_i(v,z)w \Ran
	\end{align}
Therefore, for each $\lambda,\mu\in\Cbb$, we have
	\begin{align*}
&\Lan w',P_i(\lambda)Y_i(v,z)P_i(\mu)w\Ran=\Lan P_i(\lambda)w',Y_i(v,z)P_i(\mu)w\Ran\\
\xlongequal{\eqref{eq63}}&\sum_{\wtd\lambda,\wtd \mu \in \Cbb}\Lan  P_i(\lambda)w',\Phi_{i,+}(P_+(\wtd \lambda)P_-(\wtd \mu)\aleph_z^\sharp(v)\otimes P_i(\mu)w)\Ran\\
\xlongequal[\eqref{eq64c}]{\eqref{eq64b}} &\Lan w',\Phi_{i,+}(P_+(\lambda)P_-(\mu)\aleph_z^\sharp(v)\otimes w)\Ran
	\end{align*}
Since $w'\in \Wbb'$ is arbitrary, we obtain \eqref{eq68a}. Similarly, \eqref{eq62b} implies
\begin{align}\label{eq72}
\sum_{\lambda,\mu\in \Cbb}\Lan w',\Phi_{i,-}(w\otimes P_-(\mu)P_+(\lambda)\aleph_z^\sharp(v))\Ran=\Lan w',Y'_i(v,z)w \Ran
\end{align}
and hence
\begin{align*}
&\Lan w',P_i(\mu)Y'_i(v,z)P_i(\lambda)w\Ran=\Lan P_i(\mu)w',Y'_i(v,z)P_i(\lambda)w\Ran\\
\xlongequal{\eqref{eq72}}&\sum_{\wtd\lambda,\wtd \mu \in \Cbb}\Lan  P_i(\mu)w',\Phi_{i,-}(P_i(\lambda)w\otimes P_-(\wtd \mu)P_+(\wtd \lambda)\aleph_z^\sharp(v))\Ran\\
\xlongequal[\eqref{eq65c}]{\eqref{eq65b}} &\Lan w',\Phi_{i,-}(w\otimes P_-(\mu)P_+(\lambda)\aleph_z^\sharp(v))\Ran
\end{align*}
Eq. \eqref{eq68b} follows.
\end{proof}

\begin{co}\label{lb20}
Assume that $N=2$ and $\Wbb\in\Mod(\Vbb^{\otimes 2})$. Let $z\in\Cbb^\times$. For each $\lambda\in \Cbb$, let 
\begin{align}\label{eq85}
\upchi_\lambda=P_+(\lambda)P_-(\lambda)\aleph_z^\sharp (\idt)\qquad\in (\boxtimes_\fn\Cbb)_{[\lambda,\lambda]}
\end{align}
Then $\upchi_\lambda$ is independent of the choice of $z$. Moreover, for each $w\in \Wbb$, we have 
\begin{align}\label{eq70}
	\upchi_\lambda\diamond_{L} w=P_+(\lambda)w,\quad w\diamond_R \upchi_\lambda=P_-(\lambda)w.
\end{align}
\end{co}
\begin{proof}
\eqref{eq70} follows from Thm. \ref{lb21} by choosing $v=\idt$ and $\mu=\lambda$. By the definition of $\aleph_z$ and $\aleph_z^\sharp$ in Def. \ref{lb23}, the element $\aleph_z^\sharp(\idt)\in\ovl{\boxtimes_\fn\Cbb}=(\bbs_\fn\Cbb)^*$ is equal to $\upomega$. Therefore, $\aleph_z^\sharp(\idt)$ is independent of $z$. So is $\upchi_\lambda$. 
\end{proof}

\begin{co}\label{lb54}
$(\upchi_\lambda)_{\lambda\in \Cbb}$ is a family of mutually orthogonal idempotents in the algebra $\boxtimes_\fn\Cbb$, i.e., for each $\lambda,\mu\in \Cbb$, $\upchi_\lambda\diamond \upchi_\mu=\delta_{\lambda,\mu}\upchi_\mu$. Moreover, we have
\begin{gather*}
\dim \big(\upchi_\lambda \diamond \boxtimes_\fn\Cbb \diamond \upchi_\mu\big)<\infty
\end{gather*}
For each $\psi\in \boxtimes_\fn\Cbb$, we have
\begin{align}\label{eq61}
\psi=\sum_{\lambda,\mu\in \Cbb} \upchi_\lambda \diamond \psi\diamond \upchi_\mu
\end{align}
where RHS is a finite sum. Therefore, $(\boxtimes_\fn\Cbb,\diamond)$ is an AUF algebra in the sense of \cite{GZ4}. 
\end{co}

\begin{proof}
	By Cor. \ref{lb20} with $\Wbb=\boxtimes_\fn\Cbb$, for each $\lambda,\mu\in \Cbb$, we have 
	\begin{align*}
		\upchi_\lambda\diamond \upchi_\mu=P_+(\lambda)\upchi_\mu=\delta_{\lambda,\mu}\upchi_\mu
	\end{align*}
	and 
	\begin{align}\label{eq60}
		\upchi_\lambda\diamond \boxtimes_\fn\Cbb\diamond \upchi_\mu=\boxtimes_\fn\Cbb_{[\lambda,\mu]}
	\end{align}
	\eqref{eq60} implies that $\dim \big(\upchi_\lambda \diamond \boxtimes_\fn\Cbb \diamond \upchi_\mu\big)<\infty$. Moreover, we have a decomposition
	\begin{align*}
		\boxtimes_\fn\Cbb=\bigoplus_{\lambda,\mu\in \Cbb} (\boxtimes_\fn\Cbb)_{[\lambda,\mu]}\xlongequal{\eqref{eq60}}\bigoplus_{\lambda,\mu\in \Cbb} \upchi_\lambda\diamond \boxtimes_\fn\Cbb\diamond \upchi_\mu
	\end{align*}
	which implies \eqref{eq61}. 
\end{proof}

\begin{comment}
\begin{co}
Let $\Wbb_1,\Wbb_2\in \Mod(\Vbb^{\otimes 2})$ and $T:\Wbb_1\rightarrow \Wbb_2$ be a linear map. Then $T\in \Hom_{\Vbb^{\otimes 2}}(\Wbb_1,\Wbb_2)$ iff for each $\psi,\varphi\in \boxtimes_\fn\Cbb$ and $w\in \Wbb_1$, we have 
\begin{align}
	T(\psi\diamond_L w)=\psi\diamond_L T(w),\quad T(w\diamond_R \varphi)=T(w)\diamond_R \varphi
\end{align}
\end{co}
\begin{proof}
	
\end{proof}
\end{comment}

\begin{co}
The canonical involution $\Theta$ is an anti-automorphism of the associative algebra $\boxtimes_\fn\Cbb$. That is, for each $\psi_1,\psi_2\in\boxtimes_\fn\Cbb$, we have
\begin{align}\label{eq89}
\Theta\psi_1\diamond\Theta\psi_2=\Theta(\psi_2\diamond\psi_1)
\end{align}
\end{co}

\begin{proof}
By Thm. \ref{lb62}, $\Theta$ intertwines $\Phi_{i,+}$ and $\Phi_{i,-}$. Therefore, for each $w,\psi\in\boxtimes_\fn\Cbb$, we have
\begin{align*}
\Theta\psi\diamond w=\Phi_{+,-}(w\otimes\psi)
\end{align*}
Therefore
\begin{align*}
&(\Theta\psi_1\diamond\Theta\psi_2)\diamond w=\Theta\psi_1\diamond(\Theta\psi_2\diamond w)=\Theta\psi_1\diamond\Phi_{+,-}(w\otimes\psi_2)\\
=&\Phi_{+,-}(\Phi_{+,-}(w\otimes\psi_2)\otimes\psi_1)
\end{align*}
By Thm. \ref{lb50}, the last term above equals
\begin{align*}
\Phi_{+,-} (w\otimes(\psi_2\diamond\psi_1))=\Theta(\psi_2\diamond\psi_1)
\diamond w
\end{align*}
Since $\boxtimes_\fn\Cbb$ is AUF, there exists an idempotent $w$ such that
\begin{align*}
(\Theta\psi_1\diamond\Theta\psi_2)\diamond w=\Theta\psi_1\diamond\Theta\psi_2\qquad \Theta(\psi_2\diamond\psi_1)
\diamond w=\Theta(\psi_2\diamond\psi_1)
\end{align*}
This proves \eqref{eq89}. 
\end{proof}

\subsection{The linear isomorphism $\SF:\Mod(\Vbb)\xlongrightarrow{\simeq} \Coh(\boxtimes_\fn \Cbb)$}

In this section, we apply the previous results to the case where $N=1$. Recall from \eqref{eq87} that for each $\Mbb\in\Mod(\Vbb)$,
\begin{align*}
\End^0(\Mbb)\simeq \Mbb\otimes_\Cbb\Mbb'
\end{align*}
is a $\Cbb$-subalgebra of $\End(\Mbb)$. For the remainder of this article, we identify these two spaces when no confusion arises.

\begin{df}\label{lb83}
A left $\boxtimes_\fn\Cbb$-module $(\Mbb,\pi_\Mbb)$ is called \textbf{coherent} if it is finitely generated, and if for each $m\in\Mbb$ there exists $\psi\in\boxtimes_\fn\Cbb$ such that $\pi_\Mbb(\psi)m=m$. See \cite[Sec. 2]{GZ4} for equivalent descriptions of coherent left $\boxtimes_\fn\Cbb$-modules. The linear category of coherent left $\boxtimes_\fn\Cbb$-modules is denoted by $\pmb{\Coh(\boxtimes_\fn\Cbb)}$.
\end{df}

\begin{df}\label{lb57}
Let $\Mbb\in\Mod(\Vbb)$. Note that by Prop. \ref{lb19}, we have
\begin{align}
\Hom_{\Vbb^{\otimes2}}(\boxtimes_\fn\Cbb,\Mbb\otimes\Mbb')=\ST^*\Bigg(	\vcenter{\hbox{{
		\includegraphics[height=2.2cm]{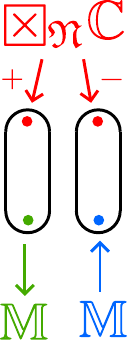}}}} ~\Bigg)=\ST^*\Bigg(	\vcenter{\hbox{{
		\includegraphics[height=2.2cm]{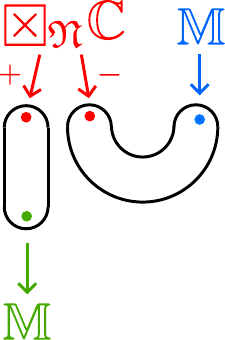}}}} ~\Bigg)
\end{align}
By the universal property of dual fusion products (cf. Def. \ref{lb2}), there is a unique
\begin{align*}
\pi_\Mbb\in\Hom_{\Vbb^{\otimes2}}(\boxtimes_\fn\Cbb,\Mbb\otimes\Mbb')
\end{align*}
whose transpose $\pi_\Mbb^\tr:\Mbb'\otimes\Mbb\rightarrow\bbs_\fn\Cbb$ composed with $\upomega:\bbs_\fn\Cbb\rightarrow\Cbb$ equals the evaluation pairing $\id_\Mbb^\flat:\Mbb'\otimes\Mbb\rightarrow\Cbb,m'\otimes m\mapsto\bk{m',m}$, that is,
\begin{align*}
\upomega\circ\pi_\Mbb^\tr=\id_\Mbb^\flat
\end{align*}
As conformal blocks, $\pi_\Mbb$ clearly equals the map $\Phi_{1,+}$ defined in Def. \ref{lb49} for $\Wbb=\Mbb$. More precisely, the map $\Phi_+\equiv \Phi_{1,+}$ satisfies
\begin{gather}
\Phi_+:\boxtimes_\fn\Cbb\otimes\Mbb\rightarrow\Mbb\qquad \psi\otimes m\mapsto\pi_\Mbb(\psi)m
\end{gather}
where $\pi_\Mbb(\psi)\in\End^0(\Mbb)$ is viewed as a linear operator on $\Mbb$.
\end{df}

\begin{rem}\label{lb25}
By Thm. \ref{lb21}, for each $m\in \Mbb,v\in \Vbb$ and $\lambda,\mu\in \Cbb$ and $z\in\Cbb^{\times}$, we have 
\begin{subequations}\label{eq83}
\begin{gather}
	\pi_\Mbb\big(P_+(\lambda)P_-(\mu)\aleph_z^\sharp(v)\big)m=P(\lambda)Y_\Mbb(v,z)P(\mu)m\label{eq83a}\\
\pi_\Mbb\big(P_+(\leq \lambda)P_-(\leq \mu)\aleph_z^\sharp(v)\big)m=P(\leq \lambda)Y_\Mbb(v,z)P(\leq \mu)m\label{eq83b}
\end{gather}
\end{subequations}
In particular, we have $\pi_\Mbb(\upchi_\lambda)m=P(\lambda)m$.
\end{rem}

\begin{pp}\label{lb28}
Let $\Mbb\in\Mod(\Vbb)$. Then $(\Mbb,\pi_\Mbb)$ is a coherent left $\boxtimes_\fn\Cbb$-module, i.e., $(\Mbb,\pi_\Mbb)\in \Coh(\boxtimes_\fn\Cbb)$.
\end{pp}

\begin{proof}
By Thm. \ref{lb50}, $(\Mbb,\pi_\Mbb)$ is a left $\boxtimes_\fn\Cbb$-module. Since $\Mbb$ is a finitely generated $\Vbb$-module, by Rem. \ref{lb25}, $\Mbb$ is a finitely generated left $\boxtimes_\fn\Cbb$-module. For each $m\in \Mbb$, choose $\lambda\in \Cbb$ such that $m\in \Mbb_{[\leq \lambda]}$. Let $\psi\in\boxtimes_\fn\Cbb$ be $\sum_{\mu\leq\lambda}P_+(\mu)P_-(\mu)\aleph_1^\sharp(\idt)$. By Rem. \ref{lb25}, we have $\pi_\Mbb(\psi)m=P({\leq\lambda)}m=m$. Therefore, $(\Mbb,\pi_\Mbb)$ is coherent.
\end{proof}

\begin{pp}\label{lb27}
Let $\Mbb_1,\Mbb_2\in \Mod(\Vbb)$ and $T:\Mbb_1\rightarrow \Mbb_2$ be a linear map. Then $T$ intertwines the actions of $\Vbb$ if and only if $T$ intertwines the actions of $\boxtimes_\fn\Cbb$. Therefore, we have
\begin{align*}
\Hom_\Vbb(\Mbb_1,\Mbb_2)=\Hom_{\boxtimes_\fn\Cbb}(\Mbb_1,\Mbb_2)
\end{align*}
\end{pp}

\begin{proof}
By Rem. \ref{lb30} and \ref{lb25}, $T$ belongs to $\Hom_{\boxtimes_\fn\Cbb}(\Mbb_1,\Mbb_2)$ iff
\begin{align*}
TP(\lambda)Y(v,z)P(\mu)m=P(\lambda)Y(v,z)P(\mu)Tm
\end{align*}
holds for each $\lambda,\mu\in\Cbb$, $v\in\Vbb$, $m\in\Mbb_1$, and $z\in\Cbb^{\times}$. This is clearly equivalent to
\begin{align}\label{eq73}
TP(\lambda)Y(v)_nP(\mu)m=P(\lambda)Y(v)_nP(\mu)Tm
\end{align}
for each $\lambda,\mu,v,m$, and $n\in\Zbb$. If this holds, then taking the sum over all $\lambda,\mu$, we obtain
\begin{align}\label{eq133}
TY(v)_nm=Y(v)_nTm
\end{align}
for all $v\in\Vbb,m\in\Mbb_1$ and $n\in\Zbb$, which implies $T\in\Hom_\Vbb(\Mbb_1,\Mbb_2)$. Conversely, assume that $T\in\Hom_\Vbb(\Mbb_1,\Mbb_2)$. Then \eqref{eq133} holds. By Lem. \ref{lb31}, $T$ intertwines the actions of $P(\lambda)$ and $P(\mu)$. Therefore \eqref{eq73} is true.
\end{proof}

\begin{thm}\label{lb55}
The linear functor
\begin{gather}
\begin{gathered}
\scr F:\Mod(\Vbb)\rightarrow \Coh(\boxtimes_\fn\Cbb)\\
(\Mbb,Y_\Mbb)\mapsto (\Mbb,\pi_\Mbb)\\
T\in\Hom_\Vbb(\Mbb_1,\Mbb_2)\mapsto T\in\Hom_{\boxtimes_\fn\Cbb}(\Mbb_1,\Mbb_2)
\end{gathered}
\end{gather}
is an isomorphism of $\Cbb$-linear categories. Consequently, $\Coh(\boxtimes_\fn\Cbb)$ is an abelian category since $\Mod(\Vbb)$ is so.
\end{thm}

It follows that $\boxtimes_\fn\Cbb$ is \textbf{strongly AUF} in the sense of \cite{GZ4}. That is, the AUF algebra $\boxtimes_\fn\Cbb$ has finitely many irreducibles (cf. \cite[Cor. 6.4]{GZ4})---equivalently, $\Coh(\boxtimes_\fn\Cbb)$ has a projective generator (cf. \cite[Prop. 7.8]{GZ4}).

\begin{proof}
Step 1. Note that if $(\Mbb_1,Y_{\Mbb_1})$ and $(\Mbb_2,Y_{\Mbb_2})$ are sent by $\scr F$ to the same object, then clearly $\Mbb_1=\Mbb_2$ as vector spaces. Moreover, the identity map on $\Mbb_1$ intertwines the actions of $\boxtimes_\fn\Cbb$, and hence intertwines the actions of $\Vbb$ by Prop. \ref{lb27}. This proves $Y_{\Mbb_1}=Y_{\Mbb_2}$. 

According to the above paragraph, if we can prove that $\SF$ is surjective, namely, each object $(\Mbb,\pi_\Mbb)$ of $\Coh(\boxtimes_\fn\Cbb)$ equals $\SF(\Mbb,Y_\Mbb)$ for some $(\Mbb,Y_\Mbb)\in\Mod(\Vbb)$, then this object $(\Mbb,Y_\Mbb)$ is unique. Therefore, the functor $\scr G:\Coh(\boxtimes\Cbb)\rightarrow\Mod(\Vbb)$ sending each $(\Mbb,\pi_\Mbb)$ to $(\Mbb,Y_\Mbb)$ and sending each morphism to itself must be the inverse functor of $\SF$. This will finish the proof that $\SF$ is an isomorphism of linear categories. 

To prove that $\SF$ is surjective, it suffices to prove that it is essentially surjective. Indeed, suppose the latter is true. Choose any $(\Mbb,\pi_\Mbb)\in\Coh(\boxtimes_\fn\Cbb)$. Then there exists $(\wtd\Mbb,Y_{\wtd\Mbb})\in\Mod(\Vbb)$ and an isomorphism of left $\boxtimes_\fn\Cbb$-modules $T:(\Mbb,\pi_\Mbb)\rightarrow(\wtd\Mbb,\pi_{\wtd\Mbb})$, where $(\wtd\Mbb,\pi_{\wtd\Mbb})=\SF(\wtd\Mbb,Y_{\wtd\Mbb})$. Let $Y_\Mbb(-,z)=T^{-1}Y_{\wtd\Mbb}(-,z)T$. Then $(\Mbb,Y_\Mbb)\in\Mod(\Vbb)$ and $\SF(\Mbb,Y_\Mbb)=(\Mbb,\pi_\Mbb)$.\\[-1ex]

Step 2. Let us prove that $\SF$ is essentially surjective. Let $\scr C$ be the set of equivalence classes of objects in the range of $\SF$. It is clear that $\scr C$ is closed under finite direct sums. $\scr C$ is also closed under taking quotients. More precisely: Suppose that $(\Mbb,\pi_\Mbb)=\SF(\Mbb,Y_\Mbb)$ where $(\Mbb,Y_\Mbb)\in\Mod(\Vbb)$, and that $\Xbb\subset\Mbb$ is a linear subspace invariant under the left action of $\boxtimes_\fn\Cbb$. Then $\Mbb/\Xbb$ belongs to the range of $\SF$. To see this, we note that by Rem. \ref{lb25}, $\Xbb$ is $\Vbb$-invariant, so $Y_\Mbb$ descends to a vertex operator $Y_{\Mbb/\Xbb}$ on $\Mbb/\Xbb$, making $(\Mbb/\Xbb,Y_{\Mbb/\Xbb})$ an object of $\Mod(\Vbb)$. By Rem. \ref{lb25} again, it is clear that $\SF(\Mbb/\Xbb,Y_{\Mbb/\Xbb})$ is the quotient of $(\Mbb,\pi_\Mbb)$ by $\Xbb$.

By \cite[Def. 2.3]{GZ4}, any object in $\Coh(\boxtimes_\fn\Cbb)$ is isomorphic to a quotient of a finite direct sum of objects of the form $\boxtimes_\fn\Cbb\diamond e$ where $e\in\boxtimes_\fn\Cbb$ is an idempotent. Moreover, by \eqref{eq61}, $e$ is a subidempotent of $f$ (i.e. $e\diamond f=f\diamond e=e$) where $f=\sum_{\mu\in E}\upchi_\mu$ with $E$ a finite subset of $\Cbb$. Therefore, by the above paragraph, it suffices to show that for each $\mu\in\Cbb$, the left $\boxtimes_\fn\Cbb$-module $\Xbb:=\boxtimes_\fn\Cbb\diamond \upchi_\mu$ (whose module structure $\wtd\pi_\Xbb$ is defined by the left action of $\boxtimes_\fn\Cbb$, i.e., $\wtd\pi_\Xbb(\psi)(w\diamond\upchi_\mu)=\psi\diamond w\diamond\upchi_\mu$ for each $\psi,w\in\boxtimes_\fn\Cbb$) belongs to the range of $\SF$.

By Cor. \ref{lb54}, we have
\begin{align*}
\Xbb=\bigoplus_{\lambda\in\Cbb}(\boxtimes_\fn\Cbb)_{[\lambda,\mu]}
\end{align*}
as vector spaces. This shows that $\Xbb$ is a grading-restricted generalized $\Vbb$-module with vertex operator $Y_\Xbb$ defined to be the restriction of $Y_{\boxtimes_\fn\Cbb,+}$ to $\Xbb$. (In fact, for any $\Wbb\in\Mod(\Vbb^{\otimes 2})$, $\big(\bigoplus_\lambda\Wbb_{[\lambda,\mu]},Y_{\Wbb,+}\big)$ is clearly an object of $\Mod(\Vbb)$.) Let $(\Xbb,\pi_\Xbb)$ be $\SF(\Xbb,Y_\Xbb)$. Let us prove that $\wtd\pi_\Xbb=\pi_\Xbb$. This will finish the proof that $(\Xbb,\wtd\pi_\Xbb)$ belongs to the range of $\SF$.

By Rem. \ref{lb30}, it suffices to prove that $\wtd\pi_\Xbb(\psi)=\pi_\Xbb(\psi)$ where $\psi=P_+(\lambda)P_-(\kappa)\aleph_1^\sharp(v)$ for some $\lambda,\kappa\in\Cbb$ and $v\in\Vbb$. Choose any $w\in\boxtimes_\fn\Cbb$. Then
\begin{align*}
&\wtd\pi_\Xbb(\psi)(w\diamond \upchi_\mu)=\psi\diamond w\diamond \upchi_\mu=(\psi\diamond w)\diamond \upchi_\mu\\
\xlongequal{\eqref{eq68}}&(P_+(\lambda)Y_{\boxtimes_\fn\Cbb,+}(v,1)P_+(\kappa)w)\diamond\upchi_\mu\xlongequal{\eqref{eq70}}P_-(\mu)P_+(\lambda)Y_{\boxtimes_\fn\Cbb,+}(v,1)P_+(\kappa)w
\end{align*}
On the other hand,
\begin{align*}
&\pi_\Xbb(\psi)(w\diamond \upchi_\mu)\xlongequal{\eqref{eq83}}P(\lambda)Y_\Xbb(v,1)P(\kappa)(w\diamond\upchi_\mu)=P_+(\lambda)Y_{\boxtimes_\fn\Cbb,+}(v,1)P_+(\kappa)(w\diamond\upchi_\mu)\\
&\xlongequal{\eqref{eq70}}P_+(\lambda)Y_{\boxtimes_\fn\Cbb,+}(v,1)P_+(\kappa)P_-(\mu)w=P_-(\mu)P_+(\lambda)Y_{\boxtimes_\fn\Cbb,+}(v,1)P_+(\kappa)w
\end{align*}
This proves $\wtd\pi_\Xbb(\psi)=\pi_\Xbb(\psi)$.
\end{proof}

\begin{co}\label{lb58}
Let $(\Gbb,Y_\Gbb)$ be a \textbf{generator} of $\Mod(\Vbb)$, that is, each object of $\Mod(\Vbb)$ has an epimorphism from a finite direct sum of $\Gbb$. Then $(\Gbb,\pi_\Gbb)$ is a generator of $\Coh(\boxtimes_\fn\Cbb)$, and 
\begin{align}\label{eq75}
	\pi_\Gbb:\boxtimes_\fn\Cbb\rightarrow \Gbb\otimes \Gbb'=\End^0(\Gbb)
\end{align} 
is an injective homomorphism of associative $\Cbb$-algebras. Moreover, for each $z\in\Cbb^\times$, we have
\begin{align}\label{eq77}
\pi_\Gbb(\boxtimes_\fn\Cbb)=\Span_\Cbb\big\{P(\lambda)Y_\Gbb(v,z)P(\mu)\big|_\Gbb:\lambda,\mu\in\Cbb,v\in\Vbb\big\}
\end{align}
\end{co}

%\pi_\Gbb^*:\SLF(\End_B^0(\Gbb))\xrightarrow{\simeq}\SLF (\boxtimes_\fn\Cbb)\quad f\mapsto f\circ \pi_\Gbb

\begin{proof}
By Thm. \ref{lb55}, $(\Gbb,\pi_\Gbb)$ is a generator of $\Coh(\boxtimes_\fn\Cbb)$. Therefore, for each $\lambda\in\Cbb$, since $\boxtimes_\fn\Cbb\diamond\upchi_\lambda$ belongs to $\Coh(\boxtimes_\fn\Cbb)$, it has an epimorphism from a direct sum of $\Gbb$. Therefore, if $x\in\boxtimes_\fn\Cbb$ acts as zero on $\Gbb$, then for each finite set $E\subset\Cbb$, $x$ acts as zero on $\bigoplus_{\lambda\in E}\boxtimes_\fn\Cbb\diamond\upchi_\lambda$. Note that this space can be written as $(\boxtimes_\fn\Cbb)\diamond e_E$ where $e_E=\sum_{\lambda\in E}\upchi_\lambda$ is an idempotent of $\boxtimes_\fn\Cbb$. Since there exists $E$ such that this space contains $x$, we have $x=x\diamond e_E$ and $e_E\in(\boxtimes_\fn\Cbb)\diamond e_E$, and hence $x=0$. This proves the injectivity of \eqref{eq75}. Eq. \eqref{eq77} follows immediately from Rem. \ref{lb30} and \ref{lb25}.
\end{proof}

\begin{co}\label{lb56}
Let $(\Gbb,Y_\Gbb)$ be a projective generator of $\Mod(\Vbb)$. Then $(\Gbb,\pi_\Gbb)$ is a projective generator of $\Coh(\boxtimes_\fn\Cbb)$. Moreover, if we let $B$ be $\End_\Vbb(\Gbb)^\opp$, the opposite algebra of $\End_\Vbb(\Gbb)$, then \eqref{eq75} restricts to an isomorphism of associative $\Cbb$-algebras
\begin{align}\label{eq76}
\pi_\Gbb:\boxtimes_\fn\Cbb\xlongrightarrow{\simeq} \End_B^0(\Gbb)
\end{align}
which is also an isomorphism in $\Mod(\Vbb^{\otimes 2})$. 
\end{co}

Recall \eqref{eq74} for the meaning of $\End_B^0(\Gbb)$.

\begin{proof}
By Thm. \ref{lb55}, $(\Gbb,\pi_\Gbb)$ is a projective generator of $\Coh(\boxtimes_\fn\Cbb)$. Note that \eqref{eq75} is simultaneously an injective homomorphism of associative algebras and (by Def. \ref{lb57}) an injective homomorphism of $\Vbb^{\otimes 2}$-modules. By \cite[Thm. 11.7]{GZ4}, the range of \eqref{eq75} is $\End_B^0(\Gbb)$. This proves that \eqref{eq75} restricts to an isomorphism of $\Cbb$-algebras \eqref{eq76}. In particular, \eqref{eq76} is bijective. Since $\End_B^0(\Gbb)$ is clearly an $\Vbb^{\otimes 2}$-submodule of $\Gbb\otimes\Gbb'\in\Mod(\Vbb^{\otimes 2})$, we have $\End_B^0(\Gbb)\in\Mod(\Vbb^{\otimes 2})$. It follows that \eqref{eq76} is also an isomorphism in $\Mod(\Vbb^{\otimes 2})$.
\end{proof}

\subsection{An alternative proof of the equivalence $\boxtimes_\fn\Cbb\simeq\int_{\Mbb\in\Mod(\Vbb)}\Mbb\otimes_\Cbb\Mbb'$}\label{lb81}

In this section, we use the results developed in the preceding sections to give an alternative proof of the isomorphism $\boxtimes_\fn\Cbb\simeq\int_{\Mbb\in\Mod(\Vbb)}\Mbb\otimes_\Cbb\Mbb'$, originally established in \cite[Sec. 0.6]{GZ3} with the help of \cite{FSS20}.

We first review the definition of ends and coend in the context of VOAs.

\begin{df}\label{lb75}
Let $N\in\Nbb$, and let $\scr D$ be a category. Suppose that $F:\Mod(\Vbb^{\otimes N})\times \Mod(\Vbb^{\otimes N})\rightarrow\scr D$ is a covariant bi-functor and $A,B\in\scr D$. A family of morphisms
\begin{align}
\varphi_\Wbb:F(\Wbb', \Wbb)\rightarrow A\quad \text{resp.}\quad \psi_\Wbb:B\rightarrow F(\Wbb', \Wbb)
\end{align}
for all $\Wbb\in\Mod(\Vbb^{\otimes N})$ (with contragredient module $\Wbb'$) is called \textbf{dinatural} if for any $\Mbb\in\Mod(\Vbb^{\otimes N})$ and $T\in\Hom_{\Vbb^{\otimes N}}(\Mbb,\Wbb)$ (with transpose $T^\tr$), the diagram
\begin{equation*}
\begin{tikzcd}[row sep=large, column sep=huge]
F(\Wbb',\Mbb) \arrow[r,"{F(T^\tr,\id_\Mbb)}"] \arrow[d,"{F(\id_{\Wbb'},T)}"'] & F(\Mbb',\Mbb) \arrow[d,"\varphi_\Mbb"] \\
F(\Wbb',\Wbb) \arrow[r,"\varphi_\Wbb"]           & A         
\end{tikzcd}
\quad\text{resp.}\quad
	\begin{tikzcd}[row sep=large, column sep=huge]
	B \arrow[r,"{\psi_\Wbb}"] \arrow[d,"{\psi_\Mbb}"'] & F(\Wbb',\Wbb) \arrow[d,"{F(T^t,\id_\Wbb)}"] \\
F(\Mbb',\Mbb) \arrow[r,"{F(\id_{\Mbb'},T)}"]           & F(\Mbb',\Wbb)      
	\end{tikzcd}
\end{equation*}
commutes. A dinatural transformation $\varphi_\Wbb:F(\Wbb', \Wbb)\rightarrow A$ (resp. $\psi_\Wbb:B\rightarrow F(\Wbb', \Wbb)$) is called a \textbf{coend} (resp. an \textbf{end}) in $\scr D$ if it satisfies the universal property that for any $\wtd A\in\scr D$ (resp. $\wtd B\in \scr D$) and dinatural transformation $\wtd \varphi_\Wbb:F(\Wbb',\Wbb)\rightarrow \wtd A$ (resp $\wtd \psi_\Wbb:\wtd B\rightarrow F(\Wbb',\Wbb)$) for all $\Wbb\in\Mod(\Vbb^{\otimes N})$, there is a unique $\Gamma\in\Hom_{\scr D}(A,\wtd A)$ (resp. $\Psi\in \Hom_{\scr D}(\wtd B,B)$) such that $\wtd\varphi_\Wbb=\Gamma\circ\varphi_\Wbb$ (resp. $\wtd\psi_\Wbb=\psi_\Wbb\circ \Psi$) holds for all $\Wbb$. In that case, we write
\begin{align*}
A=\int^{\Wbb\in\Mod(\Vbb^{\otimes N})}F(\Wbb',\Wbb)\quad\text{resp.}\quad B=\int_{\Wbb\in\Mod(\Vbb^{\otimes N})}F(\Wbb',\Wbb)
\end{align*}
\end{df}

From the universal property, if (co)ends exist, then they are unique up to unique isomorphisms. In this paper, we mainly consider the covariant bi-functor 
\begin{align*}
F:\Mod(\Vbb)\times \Mod(\Vbb)\rightarrow \Mod(\Vbb^{\otimes 2})\quad F(\Mbb,\Wbb)=\Wbb\otimes_\Cbb \Mbb
\end{align*}
%In the following, we give two proves that $\boxtimes_\fn\Cbb$ is an end.

\begin{thm}\label{end}
The dinatural transform $\pi_\Mbb:\boxtimes_\fn\Cbb\rightarrow \End^0(\Mbb)=\Mbb\otimes\Mbb'$ (for all $\Mbb\in \Mod(\Vbb)$) is an end in $\Mod(\Vbb^{\otimes 2})$. In short, we have 
\begin{align*}
	\boxtimes_\fn\Cbb\simeq \int_{\Mbb\in \Mod(\Vbb)}\Mbb\otimes\Mbb'\qquad\text{as $\Vbb^{\otimes 2}$-modules}
\end{align*} 
\end{thm}

The dinaturality of $(\pi_\Mbb)_{\Mbb\in\Mod(\Vbb)}$, which means that $\pi_\Wbb(-)\circ T=T\circ\pi_\Mbb(-)$ holds for each $\Mbb,\Wbb\in\Mod(\Vbb)$ and $T\in\Hom_\Vbb(\Mbb,\Wbb)$, is obvious due to Prop. \ref{lb27}.

\begin{proof}
Let us check the universal property. Let $(\psi_\Mbb:\Abb \rightarrow \Mbb\otimes \Mbb')_{\Mbb\in\Mod(\Vbb)}$ be a dinatural transform in $\Mod(\Vbb^{\otimes 2})$, where $\Abb\in\Mod(\Vbb^{\otimes2})$. Choose a projective generator $\Gbb\in \Mod(\Vbb)$. By the dinaturality, for any $T\in\End_\Vbb(\Gbb)$, we have $(\id_\Gbb\otimes T^\tr)\psi_\Gbb=(T\otimes\id_{\Gbb'})\psi_\Gbb$. Therefore, each element in the range of $\psi_\Gbb:\Abb\rightarrow \Gbb\otimes \Gbb'=\End^0(\Gbb)$, as a linear operator on $\Gbb$, commutes with each $T\in\End_\Vbb(\Gbb)$. Therefore, $\psi_\Gbb(\Abb)$ is contained in $\End_B^0(\Gbb)$ where $B=\End_\Vbb(\Gbb)^\opp$. By Cor. \ref{lb56}, there is a unique $\Vbb^{\otimes2}$-module morphism $\Psi$ satisfying
\begin{align*}
\Psi:\Abb\rightarrow\boxtimes_\fn\Cbb\qquad\psi_\Gbb=\pi_\Gbb\circ\Psi
\end{align*}
Indeed, one sets $\Psi=\pi_\Gbb^{-1}\circ\psi_\Gbb$.

To prove the existence part of the universal property, we need to prove $\psi_\Mbb=\pi_\Mbb\circ\Psi$ for all $\Mbb\in \Mod(\Vbb)$, not just for $\Gbb$. Let $\wtd\psi_\Mbb=\pi_\Mbb\circ\Psi$. Then $(\wtd\psi_\Mbb:\Abb\rightarrow\Mbb\otimes\Mbb')_{\Mbb\in\Mod(\Vbb)}$ is also a dinatural transform. Moreover, we know that the dinatural transforms $(\psi_\Mbb)_{\Mbb\in\Mod(\Vbb)}$ and $(\wtd\psi_\Mbb)_{\Mbb\in\Mod(\Vbb)}$ agree when $\Mbb=\Gbb$. Thus, they agree on any finite direct sum of $\Gbb$. For any $\Mbb\in\Mod(\Vbb)$, there exist $n\in\Nbb_+$ and an epimorphism $T:\Xbb\rightarrow\Mbb$ where $\Xbb=\Gbb^{\oplus n}$. By the dinaturality, we have commuting diagrams
\begin{equation*}
	\begin{tikzcd}[row sep=large, column sep=huge]
\Abb \arrow[r,"{\psi_\Mbb}"] \arrow[d,"{\psi_\Xbb}"'] & \Mbb\otimes\Mbb' \arrow[d,"{\id_\Mbb\otimes T^\tr}"] \\
\Xbb\otimes\Xbb' \arrow[r,"{T\otimes\id_{\Xbb'}}"]           & \Mbb\otimes\Xbb'   
	\end{tikzcd}
\qquad
	\begin{tikzcd}[row sep=large, column sep=huge]
\Abb \arrow[r,"{\wtd\psi_\Mbb}"] \arrow[d,"{\wtd\psi_\Xbb}"'] & \Mbb\otimes\Mbb' \arrow[d,"{\id_\Mbb\otimes T^\tr}"] \\
\Xbb\otimes\Xbb' \arrow[r,"{T\otimes\id_{\Xbb'}}"]           & \Mbb\otimes\Xbb'   
	\end{tikzcd}
\end{equation*}
where $\psi_\Xbb=\wtd\psi_\Xbb$. Thus $(\id_\Mbb\otimes T^\tr)\psi_\Mbb=(\id_\Mbb\otimes T^\tr)\wtd\psi_\Mbb$. Since $T^\tr$ is injective and hence $\id_\Mbb\otimes T^\tr$ is injective, we conclude $\psi_\Mbb=\wtd\psi_\Mbb$.

Finally, we check the uniqueness. Suppose that $\Psi':\Abb\rightarrow\boxtimes_\fn\Cbb$ is another morphism satisfying $\psi_\Mbb=\pi_\Mbb\circ\Psi'$ for all $\Mbb\in\Mod(\Vbb)$. Then $\psi_\Gbb=\pi_\Gbb\circ\Psi'$, and hence $\pi_\Gbb\circ\Psi=\pi_\Gbb\circ\Psi'$. By Cor. \ref{lb58}, $\pi_\Gbb$ is injective. Therefore $\Psi=\Psi'$.
\end{proof}

\begin{co}\label{coend}
	The dinatural transform $\pi_\Mbb^t:\Mbb'\otimes \Mbb\rightarrow \bbs_\fn\Cbb$ (for all $\Mbb\in \Mod(\Vbb)$) is a coend in $\Mod(\Vbb^{\otimes 2})$. In short, we have
\begin{align*}
\bbs_\fn\Cbb\simeq \int^{\Mbb\in \Mod(\Vbb)}\Mbb'\otimes \Mbb\qquad\text{as $\Vbb^{\otimes 2}$-modules}
\end{align*} 
\end{co}

\begin{proof}
It follows from Thm. \ref{end} by taking transpose and reversing the arrows.
\end{proof}

\subsection{Symmetric linear functionals on $\Wbb\in\Mod(\Vbb^{\otimes 2})$}

In this section, we assume $N=2$ and fix $\Wbb\in\Mod(\Vbb^{\otimes 2})$. Recall that by Cor. \ref{lb26}, $\Wbb$ is a $\boxtimes_\fn\Cbb$-bimodule. Let $\fk C=(C\big| \{z,\tipaz\};\eta_z,\eta_\tipaz)$ be a standard $2$-pointed sphere. Recall Rem. \ref{lb59} for the meaning of $\ST^*_{\fk C}(\Wbb)$. Recall the canonical involution $\Theta$ defined in Thm. \ref{lb62}.

\begin{df}\label{lb53}
Let $\Wbb\in\Mod(\Vbb^{\otimes 2})$. A linear map $\upphi:\Wbb\rightarrow\Cbb$ is called a \textbf{symmetric linear functional} if
\begin{align*}
\upphi(\psi\diamond_L w)=\upphi(w\diamond_R \psi)\qquad\text{for all } \psi\in \boxtimes_\fn\Cbb,w\in \Wbb
\end{align*}
The space of symmetric linear functionals on $\Wbb$ is denoted by $\pmb{\SLF(\Wbb)}$.
\end{df}

\begin{thm}\label{lb48}
Assume $N=2$, and let $\Wbb\in \Mod(\Vbb^{\otimes 2})$, viewed as a $\boxtimes_\fn\Cbb$-bimodule. Let $\upphi:\Wbb\rightarrow \Cbb$ be a linear map. Then $\upphi\in \SLF(\Wbb)$ if and only if $\upphi\in \ST_{\fk C}^*(\Wbb)$ (i.e., $\upphi$ satisfies \eqref{eq69} for all $v\in\Vbb$ and $w\in\Wbb$). In short, we have
\begin{align*}
\SLF(\Wbb)=\ST^*_{\fk C}(\Wbb)
\end{align*}
\end{thm}

\begin{proof}
By Rem. \ref{lb30}, any element in $\boxtimes_\fn\Cbb$ is a linear combination of elements of the form $P_+(\lambda)P_-(\mu)\aleph_z^\sharp(v)$, where $\lambda,\mu \in \Cbb$, $z\in\Cbb^\times$, and $v\in \Vbb$. Thus, $\upphi\in \SLF(\Wbb)$ iff 
\begin{align}\label{eq96}
\bk{\upphi,P_+(\lambda)P_-(\mu)\aleph_z^\sharp(v)\diamond_L w}=\bk{\upphi,w\diamond_R P_+(\lambda)P_-(\mu)\aleph_z^\sharp(v)}
\end{align}
for each $v\in\Vbb$, $w\in\Wbb$, $\lambda,\mu\in\Cbb$, and $z\in\Cbb^\times$. By Thm. \ref{lb21}, \eqref{eq96} is equivalent to 
\begin{align*}
\bk{\upphi,P_+(\lambda)Y_+(v,z)P_+(\mu)w}=\bk{\upphi,P_-(\mu)Y_-'(v,z)P_-(\lambda)w}
\end{align*}
for all $v,w,\lambda,\mu,z$, and hence equivalent to
\begin{align}\label{eq71}
\bk{\upphi,P_+(\lambda)Y_+(v)_nP_+(\mu)w}=\bk{\upphi,P_-(\mu)Y_-'(v)_nP_-(\lambda)w}
\end{align}
for all $v,w,\lambda,\mu$ and $n\in\Zbb$, where $Y_-'(v,z)=\sum_n Y_-'(v)_nz^{-n-1}$.

We want to show that condition \eqref{eq71} is equivalent to \eqref{eq69}, i.e., 
\begin{align}\label{eq97}
\bk{\upphi,Y_+(v)_nw}=\bk{\upphi,Y_-'(v)_nw}
\end{align}
holds for all $v$ and $n\in\Zbb$. If \eqref{eq71} holds, by taking sum over all $\lambda,\mu$, we obtain \eqref{eq97}. Conversely, suppose that \eqref{eq97} holds. Then $\bk{\upphi,P_+(\kappa)w}=\bk{\upphi,P_-(\kappa)w}$ holds for each $\kappa\in\Cbb$ due to Lem. \ref{lb31}, and hence
\begin{align*}
&\bk{\upphi,P_+(\lambda)Y_+(v)_nP_+(\mu)w}=\bk{\upphi,P_-(\lambda)Y_+(v)_nP_+(\mu)w}\\
=&\bk{\upphi,Y_+(v)_nP_-(\lambda)P_+(\mu)w}=\bk{\upphi,Y'_-(v)_nP_-(\lambda)P_+(\mu)w}\\
=&\bk{\upphi,P_+(\mu)Y'_-(v)_nP_-(\lambda)w}=\bk{\upphi,P_-(\mu)Y'_-(v)_nP_-(\lambda)w}
\end{align*}
This proves \eqref{eq71}.
\end{proof}

By Prop. \ref{lb13}, regardless of which ordering we choose for $\{z,\tipaz\}$, the pair $(\boxtimes_\fn\Cbb,\upomega)$ is a fusion product of $\Cbb$ along $\fk C$. Therefore, we have a linear isomorphism
\begin{gather}\label{eq100}
\Hom_{\Vbb^{\otimes 2}}(\Wbb,\bbs_\fn\Cbb)\xlongrightarrow{\simeq}\ST^*_{\fk C}(\Wbb)\qquad T\mapsto \upomega\circ T
\end{gather}
Recall that $\ovl{\bbs_\fn\Cbb}=(\boxtimes_\fn\Cbb)^*$. Therefore, for each $w\in\Wbb$ and $T\in\Hom_{\Vbb^{\otimes 2}}(\Wbb,\bbs_\fn\Cbb)$, the element $Tw$ is a linear functional on $\boxtimes_\fn\Cbb$.

\begin{thm}\label{lb64}
Choose $\upxi\in\ST^*_{\fk C}(\Wbb)$, and let $T\in\Hom_{\Vbb^{\otimes 2}}(\Wbb,\bbs_\fn\Cbb)$ be the unique morphism satisfying $\upxi=\upomega\circ T$. Then for each $w\in\Wbb$, and $\psi\in\boxtimes_\fn\Cbb$, we have
\begin{align}\label{eq80}
\bk{T(w),\Theta\psi}=\upxi(\psi\diamond_Lw)=\upxi(w\diamond_R\psi)
\end{align}
\end{thm}

\begin{proof}
The second equality in \eqref{eq80} is obvious. Let us prove the first one. 

Define $A=\Theta^\tr\circ T:\Wbb\rightarrow\bbs_\fn\Cbb$,
which clearly belongs to
\begin{align}\label{eq90}
\ST^*\Bigg(\vcenter{\hbox{{
				   \includegraphics[height=2.4cm]{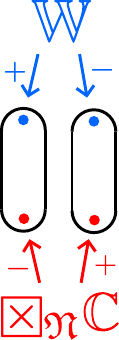}}}}~\Bigg)
\end{align}
Define $B=\upxi\circ\Phi_{+,+}:\boxtimes_\fn\Cbb\otimes\Wbb\rightarrow\Cbb$, which is a composition of conformal blocks as indicated by Fig. \ref{img8}. Therefore, $B$ belongs to \eqref{eq90}.

\begin{figure}[h]
	\centering
\begin{align*}
B\quad=\quad\vcenter{\hbox{{
				   \includegraphics[height=3.3cm]{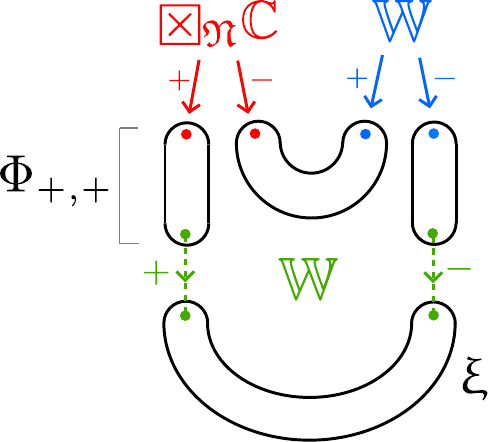}}}}
\end{align*}
\caption{~~The conformal block $B$.}
	\label{img8}
\end{figure}

Since $\upomega=\upomega\circ\Theta^\tr$ (cf. Thm. \ref{lb62}), we have $\upomega\circ A=\upomega\circ\Theta^\tr\circ T=\upxi$. This is a relation of conformal blocks, as indicated in Fig. \ref{img9}.

\begin{figure}[h]
	\centering
\begin{align*}
\quad\vcenter{\hbox{{
				   \includegraphics[height=3cm]{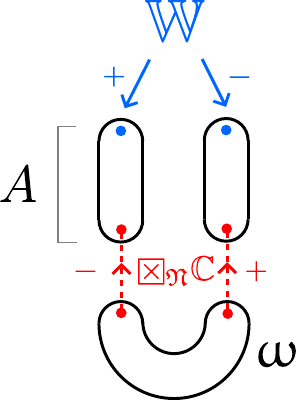}}}}
\quad=\quad
\quad\vcenter{\hbox{{
				   \includegraphics[height=1.7cm]{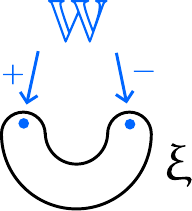}}}}
\end{align*}
\caption{~~The relation $\upomega\circ A=\upxi$.}
	\label{img9}
\end{figure}
On the other hand, for each $w\in\Wbb$, we have
\begin{align}\label{eq91}
\wick{\upomega(\c1-)B(\c1-\otimes w)}=\wick{\upxi\circ(\upomega(\c1-)\Phi_{+,+}(\c1-\otimes w))}\xlongequal{\eqref{eq47}}\upxi(w)
\end{align}
See Fig. \ref{img10}. Therefore, by the partial injectivity of $\upomega$ (cf. Rem. \ref{SF2}), we see that $A$ and $B$ correspond to the same conformal block in \eqref{eq90}. More precisely: $\bk{Aw,\psi}=B(\psi\otimes w)$ holds for all $w\in\Wbb,\psi\in\boxtimes_\fn\Cbb$. This proves \eqref{eq80}.
\begin{figure}[h]
	\centering
\begin{align*}
\quad\vcenter{\hbox{{
				   \includegraphics[height=3.5cm]{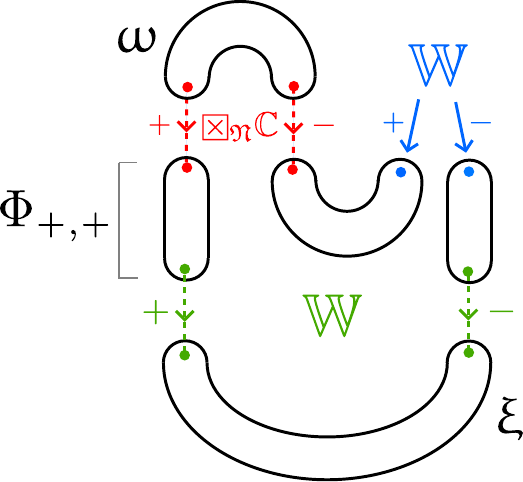}}}}
\quad\xlongequal{\eqref{eq48}}\quad
\quad\vcenter{\hbox{{
				   \includegraphics[height=3cm]{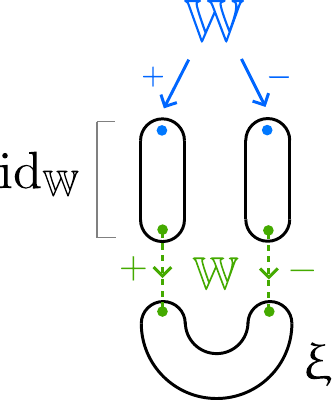}}}}
\quad=\quad
\quad\vcenter{\hbox{{
				   \includegraphics[height=1.7cm]{fig25d.pdf}}}}
\end{align*}
\caption{~~The pictorial illustration of \eqref{eq91}.}
	\label{img10}
\end{figure}
\end{proof}

We now consider the case where $\Wbb=\boxtimes_\fn\Cbb$. Then \eqref{eq100} becomes the linear isomorphism
\begin{align}\label{eq101}
\Hom_{\Vbb^{\otimes 2}}(\boxtimes_\fn\Cbb,\bbs_\fn\Cbb)\xlongrightarrow{\simeq}\SLF(\boxtimes_\fn\Cbb)=\ST^*_{\fk C}(\boxtimes_\fn\Cbb)\qquad T\mapsto \upomega\circ T
\end{align}

\begin{co}\label{lb63}
Choose $\upxi$ in $\SLF(\boxtimes_\fn\Cbb)$, and let $T\in \Hom_{\Vbb^{\otimes 2}}\big(\boxtimes_\fn\Cbb,\bbs_\fn\Cbb \big)$ be the unique morphism satisfying $\upxi=\upomega\circ T$. Then the following are equivalent:
\begin{enumerate}[label=(\alph*)]
\item $T$ is injective.
\item $T$ is surjective.
\item $T$ is an isomorphism of $\Vbb^{\otimes2}$-modules.
\item The symmetric linear functional $\upxi$ is \textbf{non-degenerate}, that is, the following map is injective:
\begin{align}
\boxtimes_\fn\Cbb\rightarrow(\boxtimes_\fn\Cbb)^*\qquad \psi\mapsto\upxi(\psi\diamond-)
\end{align}
\end{enumerate}
\end{co}

\begin{proof}
For each $\lambda,\mu\in\Cbb$, the morphism $T$ restricts to
	\begin{gather}\label{eq92}
		T:(\boxtimes_\fn\Cbb)_{[\lambda,\mu]}\rightarrow (\bbs_\fn\Cbb)_{[\lambda,\mu]}
	\end{gather}
whose domain and codomain have the same finite dimension (because $(\bbs_\fn\Cbb)_{[\lambda,\mu]}$ is the dual space of $(\boxtimes_\fn\Cbb)_{[\lambda,\mu]}$). Therefore, $T$ is injective iff $T$ is surjective iff $T$ is bijective. This proves the equivalence of (a), (b), and (c). By Thm. \ref{lb64}, the map \eqref{eq92} sends each $w\in\boxtimes_\fn\Cbb$ to $\Theta^\tr\circ T(w)$. Therefore, (a) and (d) are equivalent.
\end{proof}

\section{Torus conformal blocks and pseudo-$q$-traces}

\subsection{The space $\ST^*_{\ft_{z,q}}(\Vbb)$ of vacuum torus conformal blocks}

\subsubsection{The geometric setting}\label{lb67}

As usual, we let $\zeta$ be the standard coordinate of $\Cbb$. We fix a standard $2$-pointed sphere $\fk C$, viewed as an $(0,2)$-pointed sphere. Fix $z\in\Cbb^\times$, and fix $\tau\in\mbb H$ where $\mbb H\subset\Cbb$ is the (open) upper half plane. Let $q=e^{2\im\pi\tau}$ with $\arg q=2\pi\Re(\tau)$. Let $\fq_{z,q}$ be the $(1,2)$-pointed sphere with local coordinates described by
\begin{align*}
\fq_{z,q}=\big(\{\infty,0\};1/q\zeta,\zeta\big|\Pbb^1\big|z;\zeta-z\big)
\end{align*}
Since $0<|q|<1$, $\fc$ can be composed with $\fq_{z,q}$, because the sewing radii can be chosen to be admissible when the sewing moduli are set to $1$. Let
\begin{align}
\tipae:\{+,-\}\xlongrightarrow{\simeq}\{\infty,0\}\qquad \tipae(+)=\infty\qquad \tipae(-)=0
\end{align} 
be the default ordering of the outgoing marked points of $\fq_{z,q}$. Let $\eps$ be an arbitrary ordering of the marked points of $\fc$. Let $\fk T_{z,q}$ be the $1$-pointed torus with local coordinates defined by
\begin{align*}
\fk T_{z,q}=\fc\#^{\eps,\tipae}\fq_{z,q}
\end{align*}
That is,
\begin{align*}
\vcenter{\hbox{{\includegraphics[height=1cm]{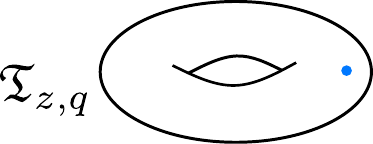}}}}\quad=\quad \vcenter{\hbox{{\includegraphics[height=1.4cm]{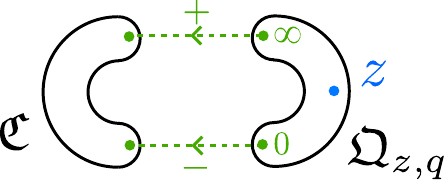}}}}
\end{align*}
Clearly, $\fk T_{z,q}$ is isomorphic to
\begin{align*}
\Big(\Cbb/(\Zbb+\tau\Zbb)\Big|\frac{\log z}{2\im\pi};e^{2\im\pi\zeta}-z \Big)
\end{align*}

\subsubsection{The sewing-factorization theorem for vacuum torus conformal blocks}

Recall from Rem. \ref{lb59} that $\ST^*_\fc(\boxtimes_\fn\Cbb)$ is the space of conformal blocks associated to $\boxtimes_\fn\Cbb$ and $\fc$, where the choice of ordering is irrelevant. Recall from Thm. \ref{lb48} that
\begin{align*}
\SLF(\boxtimes_\fn\Cbb)=\ST^*_{\fk C}(\boxtimes_\fn\Cbb)
\end{align*}
Let $\ST^*_{\ft_{z,q}}(\Vbb)$ be the space of conformal blocks associated to $\Vbb$ and $\ft_{z,q}$ via the unique ordering of the marked point of $\ft_{z,q}$. 

\begin{rem}
Let $\iota_z$ be the unique ordering of $\{z\}$. As noted in Def. \ref{lb23}, $(\boxtimes_\fn\Cbb,\aleph_z)$ is an $(\tipae,\iota_z)$-fusion product of $\Vbb$ along $\fq_z=\fq_{z,1}$. Let
\begin{align*}
\aleph_{z,q}=\aleph_z\circ(q^{L_+(0)}\otimes\id_\Vbb):\bbs_\fn\Cbb\otimes\Vbb\rightarrow\Cbb
\end{align*}
whose corresponding map $\Vbb\rightarrow\ovl{\boxtimes_\fn\Cbb}$ is
\begin{align*}
\aleph^\sharp_{z,q}=q^{L_+(0)}\circ\aleph^\sharp_z:\Vbb\rightarrow\ovl{\boxtimes_\fn\Cbb}
\end{align*}
Then by Prop. \ref{lb66},  $(\boxtimes_\fn\Cbb,\aleph_{z,q})$ is an $(\tipae,\iota_z)$-fusion product of $\Vbb$ along $\fq_{z,q}$.
\end{rem}

\begin{thm}\label{lb69}
We have a linear isomorphism
\begin{gather}\label{eq94}
\SLF(\boxtimes_\fn\Cbb)\xlongrightarrow{\simeq}\ST^*_{\fk T_{z,q}}(\Vbb)\qquad\upxi\mapsto \upxi\circ \aleph^\sharp_{z,q}
\end{gather}
where the RHS converges absolutely in the sense of \eqref{eq104}.
\end{thm}

\begin{proof}
This is a special case of the sewing-factorization Thm. \ref{SF}.
\end{proof}

The picture for the sewing-factorization isomorphism \eqref{eq94} is
\begin{align*}
\vcenter{\hbox{{\includegraphics[height=1.8cm]{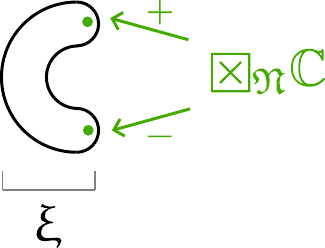}}}}\qquad\longmapsto\qquad \vcenter{\hbox{{\includegraphics[height=1.9cm]{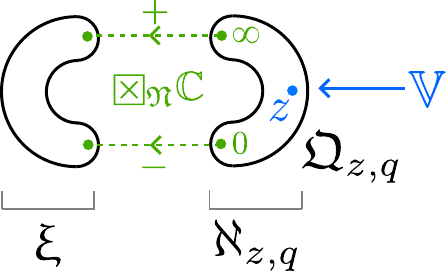}}}}
\end{align*}

\subsection{The isomorphism $\SLF(\End_\Vbb(\Gbb)^\opp)\simeq\ST^*_{\ft_{z,q}}(\Vbb)$ via pseudo-$q$-traces}\label{lb79}

Let $(\Gbb,Y_\Gbb)$ be a projective generator of $\Mod(\Vbb)$. By Cor. \ref{lb56}, $(\Gbb,\pi_\Gbb)$ is a projective generator of $\Mod(\boxtimes_\fn\Cbb)$. Let
\begin{align*}
B=\End_\Vbb(\Gbb)^\opp\xlongequal{\text{Prop.\ref{lb27}}}\End_{\boxtimes_\fn\Cbb}(\Gbb)^\opp
\end{align*}
which is a finite-dimensional unital $\Cbb$-algebra. Then $\Gbb$ is a $\boxtimes_\fn\Cbb$-$B$ bimodule.

\subsubsection{Pseudotraces}
By Prop. 9.1 and Thm. 9.4 of \cite{GZ4}, or by \cite[Lem. 5.3]{GR-Verlinde}, $\Gbb$ is a projective right $B$-module. In particular, each graded subspace $\Gbb_{[\lambda]}$ is a finite-dimensional projective right $B$-module. It follows that the right $B$-module $\Gbb$ has a \textbf{left coordinate system}, i.e., a collection of morphisms of right $B$-modules
\begin{align}\label{eq98}
	\alpha_i\in \Hom_B(B,\Gbb)\quad \wch\alpha^i\in \Hom_B(\Gbb,B),\quad \text{where }i\in I
\end{align}
satisfying the following conditions:
\begin{enumerate}[label=(\alph*)]
\item  For each $\xi\in \Gbb$, we have $\wch\alpha^i(\xi)=0$ for all but finitely many $i\in I$, and  $\sum_{i\in I}\alpha_i\circ\wch\alpha^i(\xi)=\xi$.
\item For each $x\in \End_B^0(\Gbb)$, we have $x\circ\alpha_i=0$ and $\wch\alpha^i\circ x=0$ for all but finitely many $i\in I$.
\end{enumerate}

\begin{df}
For each $\phi\in\SLF(B)$, the associated (left) \textbf{pseudotrace} $\Tr^\phi:\End^0_B(\Gbb)\rightarrow\Cbb$ is defined by sending each $x\in\End^0_B(\Gbb)$ to
\begin{align}
\Tr^\phi(x)=\sum_{i\in I}\phi\big(\wch\alpha^i\circ x\circ\alpha_i(1_B)\big)
\end{align}
Then $\Tr^\phi$ is independent of the choice of left coordinate systems, and $\Tr^\phi\in\SLF(\End^0_B(\Gbb))$. See \cite[Sec. 4]{GZ4} for details.
\end{df}

\begin{thm}\label{lb70}
We have a linear isomorphism
\begin{align}\label{eq95}
\SLF(\End_\Vbb(\Gbb)^\opp)\xlongrightarrow{\simeq}\SLF(\boxtimes_\fn\Cbb)\qquad\phi\mapsto \Tr^\phi\circ\pi_\Gbb
\end{align}
Moreover, $\phi$ is non-degenerate if and only if $\Tr^\phi\circ\pi_\Gbb$ is non-degenerate.
\end{thm}

\begin{proof}
This is due to Thm. 9.4 and 10.4 of \cite{GZ4}.
\end{proof}

\begin{rem}\label{lb73}
Recall from \eqref{eq101} the bijection between $T\in\Hom_{\Vbb^{\otimes 2}}(\boxtimes_\fn\Cbb,\bbs_\fn\Cbb)$ and $\upxi\in\SLF(\boxtimes_\fn\Cbb)$ related by $\upxi=\upomega\circ T$. Therefore, by Thm. \ref{lb70}, there is a bijection between  $\phi\in\SLF(\End_\Vbb(\Gbb)^\opp)$ and
\begin{align*}
T^\phi\in\Hom_{\Vbb^{\otimes2}}(\boxtimes_\fn\Cbb,\bbs_\fn\Cbb)\simeq \Hom_{\Vbb^{\otimes2}}\Big(\int_{\Mbb\in\Mod(\Vbb)}\Mbb\otimes\Mbb',\int^{\Mbb\in\Mod(\Vbb)}\Mbb'\otimes\Mbb\Big)
\end{align*}
(cf. Thm. \ref{end} and Cor. \ref{coend} for the last equivalence) related by
\begin{align}
\Tr^\phi\circ\pi_\Gbb=\upomega\circ T^\phi
\end{align}
Moreover, by Cor. \ref{lb63}, $\phi$ is non-degenerate iff $T^\phi$ is an isomorphism of $\Vbb^{\otimes 2}$-modules.
\end{rem}

\begin{rem}\label{lb74}
In the special case that $\Vbb$ is \textbf{strongly-finite} (i.e. the $C_2$-cofinite VOA $\Vbb=\bigoplus_{n\in\Nbb}\Vbb(n)$ satisfies $\Vbb\simeq\Vbb'$ and $\dim\Vbb(0)=1$), suppose that the conjectured rigidity of $\Mod(\Vbb)$ holds. Then by \cite{McR21-rational}, $\Mod(\Vbb)$ is a factorizable finite ribbon category. Therefore, by \cite{GR-modified-trace}, there is a distinguished (up to scalar multiplication) element $\phi\in\SLF(\End_\Vbb(\Gbb)^\opp)$, called the \textbf{modified trace}, which is non-degenerate \cite[Prop. 4.2]{GR-modified-trace}. Therefore, by Rem. \ref{lb73}, $\phi$ gives rise to an distinguished $\Vbb^{\otimes 2}$-module isomorphism
\begin{align*}
T^\phi:\boxtimes_\fn\Cbb\xrightarrow{\simeq}\bbs_\fn\Cbb
\end{align*}
In particular, $\boxtimes_\fn\Cbb$ is self-dual. (The self-dualness of the end $\boxtimes_\fn\Cbb$ also follows (more directly) from \cite{Shi-unimodular}.)
\end{rem}

\subsubsection{The isomorphism $\SLF(\End_\Vbb(\Gbb)^\opp)\simeq\ST^*_{\ft_{z,q}}(\Vbb)$}

Recall $z\in\Cbb^\times$, $q=e^{2\im\pi\tau}$ (where $\tau\in\mbb H$), and the $1$-pointed torus $\ft_{z,q}$ in Subsec. \ref{lb67}. We are now ready to prove the following conjecture by Gainutdinov-Runkel, cf. \cite[Conjecture 5.8]{GR-Verlinde}. Note that each element of $\ST^*_{\ft_{z,q}}(\Vbb)$ is a linear functional on $\Vbb$.

\begin{thm}\label{lb68}
Let $(\Gbb,Y_\Gbb)$ be a projective generator of $\Mod(\Vbb)$. Then we have a linear isomorphism
\begin{gather}\label{eq99}
\SLF (\End_\Vbb(\Gbb)^\opp)\xlongrightarrow{\simeq} \ST_{\fk T_{z,q}}^*(\Vbb)\qquad\phi\mapsto \Tr^\phi\big(Y_\Gbb(-,z)q^{L(0)}\big)
\end{gather}
where the RHS is understood as follows: for each $v\in\Vbb$,
\begin{align*}
\Tr^\phi\big(Y_\Gbb(v,z)q^{L(0)}\big):=\sum_{\lambda\in\Cbb}\Tr^\phi\big(P(\lambda)Y_\Gbb(v,z)q^{L(0)}P(\lambda)\big)
\end{align*}
and the series on the RHS converges absolutely.
\end{thm}

The construction of $\Tr^\phi\big(Y_\Gbb(-,z)q^{L(0)}\big)$ from $\phi$ is called the \textbf{pseudo-$q$-trace construction}.

\begin{proof}
Let us show that the composition of the isomorphisms \eqref{eq94} and \eqref{eq95} agrees with \eqref{eq99}. We compute that
\begin{align*}
&\bk{\eqref{eq94}\circ\eqref{eq95}(\phi),v}=(\Tr^\phi\circ\pi_\Gbb)\circ\aleph^\sharp_{z,q}(v)\xlongequal{\text{Thm. \ref{lb47}}}\sum_{\lambda,\mu\in\Cbb}\Tr^\phi\circ\pi_\Gbb\big(P(\lambda,\mu)\aleph^\sharp_{z,q}(v)\big)\\
=&\sum_{\lambda,\mu\in\Cbb}\Tr^\phi\circ\pi_\Gbb\big(P_+(\lambda)P_-(\mu)\aleph^\sharp_{z,q}(v)\big)=\sum_{\lambda,\mu\in\Cbb}\Tr^\phi\circ\pi_\Gbb\big(P_+(\lambda)P_-(\mu)q^{L_+(0)}\aleph^\sharp_z(v)\big)
\end{align*}
where the sums converge absolutely. Recall from Def. \ref{lb57} that for each $\psi\in\boxtimes_\fn\Cbb$ and $w\in\Gbb$, we have $\Phi_+(\psi\otimes w)=\pi_\Gbb(\psi)w$. By \eqref{eq118c}, we have $\Phi_+(q^{L_+(0)}\psi\otimes w)=q^{L(0)}\Phi_+(\psi\otimes w)$, and hence $\pi_\Gbb(q^{L_+(0)}\psi)=q^{L(0)}\pi_\Gbb(\psi)$. Therefore,
\begin{align*}
&\bk{\eqref{eq94}\circ\eqref{eq95}(\phi),v}=\sum_{\lambda,\mu\in\Cbb}\Tr^\phi\Big( q^{L(0)}\pi_\Gbb\big(P_+(\lambda)P_-(\mu)\aleph^\sharp_z(v)\big)\Big)\\
\xlongequal{\eqref{eq83a}}&\sum_{\lambda,\mu\in\Cbb}\Tr^\phi\big(q^{L(0)} P(\lambda)Y_\Gbb(v,z)P(\mu)\big)=\sum_{\lambda,\mu\in\Cbb}\Tr^\phi\big(P(\lambda)q^{L(0)} P(\lambda)Y_\Gbb(v,z)P(\mu)\big)
\end{align*}
Since $\Tr^\phi$ a symmetric linear functional on $\End^0_B(\Gbb)$, the last term above equals
\begin{align*}
&\sum_{\lambda,\mu\in\Cbb}\Tr^\phi\big(P(\lambda)Y_\Gbb(v,z)P(\mu)P(\lambda)q^{L(0)} P(\lambda)\big)=\sum_{\lambda\in\Cbb}\Tr^\phi\big(P(\lambda)Y_\Gbb(v,z)q^{L(0)} P(\lambda)\big)
\end{align*}
This finishes the proof.
\end{proof}

For any associative $\Cbb$-algebra $A$, let $Z(A)$ be its center.

\begin{co}
Assume that $\Vbb$ is strongly-finite, and that the conjectured rigidity of $\Mod(\Vbb)$ holds. Then for each projective generator $\Gbb$ of $\Mod(\Vbb)$, and for each non-degenerate $\phi\in\SLF(\End_\Vbb(\Gbb)^\opp)$ (cf. Rem. \ref{lb74} for the existence), we have a linear isomorphism
\begin{align}
Z(\End_\Vbb(\Gbb)^\opp)\xlongrightarrow{\simeq} \ST_{\fk T_{z,q}}^*(\Vbb)\qquad x\mapsto \Tr^{\phi_x}\big(Y_\Gbb(-,z)q^{L(0)}\big)
\end{align}
where $\phi_x\in\SLF(\End_\Vbb(\Gbb)^\opp)$ is defined by sending each $y\in\End_\Vbb(\Gbb)^\opp$ to $\phi(xy)$.
\end{co}

\begin{proof}
This follows immediately from Thm. \ref{lb68} and the easy fact that for any finite-dimensional $\Cbb$-algebra $A$ and a fixed non-degenerate $\phi\in\SLF(A)$, the map
\begin{align*}
Z(A)\rightarrow\SLF(A)\qquad x\mapsto\phi_x
\end{align*}
is a linear isomorphism.
\end{proof}

%%%%%%%%%%%%%%%%%%%%%%%%%%%%%%%%%%%%%%%%%%%%%%%%%%%%%%%%%%%%%%%%
%  References
%%%%%%%%%%%%%%%%%%%%%%%%%%%%%%%%%%%%%%%%%%%%%%%%%%%%%%%%%%%%%%%%
\footnotesize
	\bibliographystyle{alpha}
    \bibliography{voa}

\begin{thebibliography}{ALSW21}

\bibitem[ALSW21]{ALSW21}
Robert Allen, Simon Lentner, Christoph Schweigert, and Simon Wood.
\newblock Duality structures for module categories of vertex operator algebras
  and the {Feigin} {Fuchs} boson.
\newblock Preprint, {arXiv}:2107.05718 [math.{QA}] (2021), 2021.

\bibitem[AN03]{AN03-finite-dimensional}
Toshiyuki Abe and Kiyokazu Nagatomo.
\newblock Finiteness of conformal blocks over compact {Riemann} surfaces.
\newblock {\em Osaka J. Math.}, 40(2):375--391, 2003.

\bibitem[AN13]{AN-pseudo-trace}
Yusuke Arike and Kiyokazu Nagatomo.
\newblock Some remarks on pseudo-trace functions for orbifold models associated
  with symplectic fermions.
\newblock {\em Int. J. Math.}, 24(2):1350008, 29, 2013.

\bibitem[DGK24]{DGK3-morita}
Chiara Damiolini, Angela Gibney, and Daniel Krashen.
\newblock Morita equivalences for {Zhu}'s algebra.
\newblock Preprint, {arXiv}:2403.11855 [math.{RT}] (2024), 2024.

\bibitem[DGK25]{DGK2}
Chiara Damiolini, Angela Gibney, and Daniel Krashen.
\newblock Conformal blocks on smoothings via mode transition algebras.
\newblock {\em Comm. Math. Phys.}, 406(6):Paper No. 131, 58, 2025.

\bibitem[DGT24]{DGT2}
Chiara Damiolini, Angela Gibney, and Nicola Tarasca.
\newblock On factorization and vector bundles of conformal blocks from vertex
  algebras.
\newblock {\em Ann. Sci. \'Ec. Norm. Sup\'er. (4)}, 57(1):241--292, 2024.

\bibitem[DSPS19]{DSPS19-balanced}
Christopher~L. Douglas, Christopher Schommer-Pries, and Noah Snyder.
\newblock The balanced tensor product of module categories.
\newblock {\em Kyoto J. Math.}, 59(1):167--179, 2019.

\bibitem[DW25]{DW-modular-functor}
Chiara Damiolini and Lukas Woike.
\newblock Modular functors from conformal blocks of rational vertex operator
  algebras.
\newblock Preprint, {arXiv}:2507.05845 [math.{QA}] (2025), 2025.

\bibitem[FBZ04]{FB04}
Edward Frenkel and David Ben-Zvi.
\newblock {\em Vertex algebras and algebraic curves}, volume~88 of {\em
  Mathematical Surveys and Monographs}.
\newblock American Mathematical Society, Providence, RI, second edition, 2004.

\bibitem[Fio16]{Fio-genus-1}
Francesco Fiordalisi.
\newblock Logarithmic intertwining operators and genus-one correlation
  functions.
\newblock {\em Commun. Contemp. Math.}, 18(6):46, 2016.
\newblock Id/No 1650026.

\bibitem[FS17]{FS-coends-CFT}
J\"urgen Fuchs and Christoph Schweigert.
\newblock Coends in conformal field theory.
\newblock In {\em Lie algebras, vertex operator algebras, and related topics},
  volume 695 of {\em Contemp. Math.}, pages 65--81. Amer. Math. Soc.,
  Providence, RI, 2017.

\bibitem[FSS20]{FSS20}
J\"urgen Fuchs, Gregor Schaumann, and Christoph Schweigert.
\newblock Eilenberg-{W}atts calculus for finite categories and a bimodule
  {R}adford {$S^4$} theorem.
\newblock {\em Trans. Amer. Math. Soc.}, 373(1):1--40, 2020.

\bibitem[GR19]{GR-Verlinde}
Azat~M. Gainutdinov and Ingo Runkel.
\newblock The non-semisimple {V}erlinde formula and pseudo-trace functions.
\newblock {\em J. Pure Appl. Algebra}, 223(2):660--690, 2019.

\bibitem[GR20]{GR-modified-trace}
Azat~M. Gainutdinov and Ingo Runkel.
\newblock Projective objects and the modified trace in factorisable finite
  tensor categories.
\newblock {\em Compos. Math.}, 156(4):770--821, 2020.

\bibitem[Gui24]{Gui-sewingconvergence}
Bin Gui.
\newblock Convergence of sewing conformal blocks.
\newblock {\em Commun. Contemp. Math.}, 26(3):65, 2024.
\newblock Id/No 2350007.

\bibitem[GZ23]{GZ1}
Bin Gui and Hao Zhang.
\newblock Analytic conformal blocks of {$C_2$}-cofinite vertex operator
  algebras {I}: Propagation and dual fusion products.
\newblock arXiv:2305.10180, 2023.

\bibitem[GZ24]{GZ2}
Bin Gui and Hao Zhang.
\newblock Analytic conformal blocks of {$C_2$}-cofinite vertex operator
  algebras {II}: Convergence of sewing and higher genus pseudo-{$q$}-traces.
\newblock arXiv:2411.07707, 2024.

\bibitem[GZ25a]{GZ3}
Bin Gui and Hao Zhang.
\newblock Analytic conformal blocks of {$C_2$}-cofinite vertex operator
  algebras {III}: The sewing-factorization theorems.
\newblock arXiv:2503.23995, 2025.

\bibitem[GZ25b]{GZ4}
Bin Gui and Hao Zhang.
\newblock Pseudotraces on almost unital and finite-dimensional algebras.
\newblock arXiv:2508.00431, 2025.

\bibitem[HLZ12a]{HLZ2}
Yi-Zhi Huang, James Lepowsky, and Lin Zhang.
\newblock Logarithmic tensor category theory, ii: Logarithmic formal calculus
  and properties of logarithmic intertwining operators.
\newblock arXiv:1012.4196, 2012.

\bibitem[HLZ12b]{HLZ3}
Yi-Zhi Huang, James Lepowsky, and Lin Zhang.
\newblock Logarithmic tensor category theory, iii: Intertwining maps and tensor
  product bifunctors.
\newblock arXiv:1012.4197, 2012.

\bibitem[HLZ12c]{HLZ4}
Yi-Zhi Huang, James Lepowsky, and Lin Zhang.
\newblock Logarithmic tensor category theory, iv: Constructions of tensor
  product bifunctors and the compatibility conditions.
\newblock arXiv:1012.4198, 2012.

\bibitem[HLZ12d]{HLZ5}
Yi-Zhi Huang, James Lepowsky, and Lin Zhang.
\newblock Logarithmic tensor category theory, v: Convergence condition for
  intertwining maps and the corresponding compatibility condition.
\newblock arXiv:1012.4199, 2012.

\bibitem[HLZ12e]{HLZ6}
Yi-Zhi Huang, James Lepowsky, and Lin Zhang.
\newblock Logarithmic tensor category theory, vi: Expansion condition,
  associativity of logarithmic intertwining operators, and the associativity
  isomorphisms.
\newblock arXiv:1012.4202, 2012.

\bibitem[HLZ12f]{HLZ7}
Yi-Zhi Huang, James Lepowsky, and Lin Zhang.
\newblock Logarithmic tensor category theory, vii: Convergence and extension
  properties and applications to expansion for intertwining maps.
\newblock arXiv:1110.1929, 2012.

\bibitem[HLZ12g]{HLZ8}
Yi-Zhi Huang, James Lepowsky, and Lin Zhang.
\newblock Logarithmic tensor category theory, viii: Braided tensor category
  structure on categories of generalized modules for a conformal vertex
  algebra.
\newblock arXiv:1110.1931, 2012.

\bibitem[HLZ14]{HLZ1}
Yi-Zhi Huang, James Lepowsky, and Lin Zhang.
\newblock Logarithmic tensor category theory for generalized modules for a
  conformal vertex algebra, {I}: introduction and strongly graded algebras and
  their generalized modules.
\newblock In {\em Conformal field theories and tensor categories}, Math. Lect.
  Peking Univ., pages 169--248. Springer, Heidelberg, 2014.

\bibitem[HR24]{HR24-MF}
Aaron Hofer and Ingo Runkel.
\newblock Modular functors from non-semisimple 3d {TFTs}.
\newblock Preprint, {arXiv}:2405.18038 [math.{QA}] (2024), 2024.

\bibitem[Hua95]{Hua-tensor-4}
Yi-Zhi Huang.
\newblock A theory of tensor products for module categories for a vertex
  operator algebra. {IV}.
\newblock {\em J. Pure Appl. Algebra}, 100(1-3):173--216, 1995.

\bibitem[Hua05a]{Hua-differential-genus-0}
Yi-Zhi Huang.
\newblock Differential equations and intertwining operators.
\newblock {\em Commun. Contemp. Math.}, 7(3):375--400, 2005.

\bibitem[Hua05b]{Hua-differential-genus-1}
Yi-Zhi Huang.
\newblock Differential equations, duality and modular invariance.
\newblock {\em Commun. Contemp. Math.}, 7(5):649--706, 2005.

\bibitem[Hua09]{Hua-projectivecover}
Yi-Zhi Huang.
\newblock Cofiniteness conditions, projective covers and the logarithmic tensor
  product theory.
\newblock {\em J. Pure Appl. Algebra}, 213(4):458--475, 2009.

\bibitem[Hua22]{Hua22-Ass-IO}
Yi-Zhi Huang.
\newblock Associative algebras and intertwining operators.
\newblock {\em Comm. Math. Phys.}, 396(1):1--44, 2022.

\bibitem[Hua24a]{Hua-associative}
Yi-Zhi Huang.
\newblock Associative algebras and the representation theory of
  grading-restricted vertex algebras.
\newblock {\em Commun. Contemp. Math.}, 26(6):Paper No. 2350036, 46, 2024.

\bibitem[Hua24b]{Hua-modular-C2}
Yi-Zhi Huang.
\newblock Modular invariance of (logarithmic) intertwining operators.
\newblock {\em Comm. Math. Phys.}, 405(5):Paper No. 131, 82, 2024.

\bibitem[Li01a]{Li-regular-Zhu}
Haisheng Li.
\newblock The regular representation, {Zhu}'s {{\(A(V)\)}}-theory, and induced
  modules.
\newblock {\em J. Algebra}, 238(1):159--193, 2001.

\bibitem[Li01b]{Li-regular-AnV}
Haisheng Li.
\newblock The regular representations and the {{\(A_n(V)\)}}-algebras.
\newblock In {\em Proceedings on Moonshine and related topics. Proceedings of
  the workshop, Montr\'eal, Canada, May 1999. Dedicated to the memory of
  Chih-Han Sah}, pages 99--116. Providence, RI: American Mathematical Society
  (AMS), 2001.

\bibitem[Li02]{Li-regular-rep}
Haisheng Li.
\newblock Regular representations of vertex operator algebras.
\newblock {\em Commun. Contemp. Math.}, 4(4):639--683, 2002.

\bibitem[Lyu95]{Lyu95-Invariants}
Volodymyr~V. Lyubashenko.
\newblock Invariants of {$3$}-manifolds and projective representations of
  mapping class groups via quantum groups at roots of unity.
\newblock {\em Comm. Math. Phys.}, 172(3):467--516, 1995.

\bibitem[Lyu96]{Lyu96-Ribbon}
V.~Lyubashenko.
\newblock Ribbon abelian categories as modular categories.
\newblock {\em J. Knot Theory Ramifications}, 5(3):311--403, 1996.

\bibitem[McR21]{McR21-rational}
Robert McRae.
\newblock On rationality for ${C_2}$-cofinite vertex operator algebras.
\newblock Preprint, {arXiv}:2108.01898 [math.{QA}] (2021), 2021.

\bibitem[McR23]{McR-deligne}
Robert McRae.
\newblock Deligne tensor products of categories of modules for vertex operator
  algebras.
\newblock 2304.14023v1, 2023.

\bibitem[Miy04]{Miy-modular-invariance}
Masahiko Miyamoto.
\newblock Modular invariance of vertex operator algebras satisfying
  {$C_2$}-cofiniteness.
\newblock {\em Duke Math. J.}, 122(1):51--91, 2004.

\bibitem[MNT10]{MNT10}
Atsushi Matsuo, Kiyokazu Nagatomo, and Akihiro Tsuchiya.
\newblock Quasi-finite algebras graded by {H}amiltonian and vertex operator
  algebras.
\newblock In {\em Moonshine: the first quarter century and beyond}, volume 372
  of {\em London Math. Soc. Lecture Note Ser.}, pages 282--329. Cambridge Univ.
  Press, Cambridge, 2010.

\bibitem[Mor22]{Moriwaki22-CB}
Yuto Moriwaki.
\newblock Vertex operator algebra and parenthesized braid operad.
\newblock Preprint, {arXiv}:2209.10443, 2022.

\bibitem[NT05]{NT-P1_conformal_blocks}
Kiyokazu Nagatomo and Akihiro Tsuchiya.
\newblock Conformal field theories associated to regular chiral vertex operator
  algebras. {I}: {Theories} over the projective line.
\newblock {\em Duke Math. J.}, 128(3):393--471, 2005.

\bibitem[Seg88]{Segal-CFT1}
G.~B. Segal.
\newblock The definition of conformal field theory.
\newblock Differential geometrical methods in theoretical physics, {Proc}. 16th
  {Int}. {Conf}., {NATO} {Adv}. {Res}. {Workshop}, {Como}/{Italy} 1987, {NATO}
  {ASI} {Ser}., {Ser}. {C} 250, 165-171 (1988)., 1988.

\bibitem[Seg04]{Segal-CFT2}
Graeme Segal.
\newblock The definition of conformal field theory.
\newblock In {\em Topology, geometry and quantum field theory. Proceedings of
  the 2002 Oxford symposium in honour of the 60th birthday of Graeme Segal,
  Oxford, UK, June 24--29, 2002}, pages 421--577. Cambridge: Cambridge
  University Press, 2004.

\bibitem[Shi17]{Shi-unimodular}
Kenichi Shimizu.
\newblock On unimodular finite tensor categories.
\newblock {\em Int. Math. Res. Not. IMRN}, (1):277--322, 2017.

\bibitem[Zhu96]{Zhu-modular-invariance}
Yongchang Zhu.
\newblock Modular invariance of characters of vertex operator algebras.
\newblock {\em J. Amer. Math. Soc.}, 9(1):237--302, 1996.

\end{thebibliography}

\noindent {\small \sc Yau Mathematical Sciences Center, Tsinghua University, Beijing, China.}

\noindent {\textit{E-mail}}: binguimath@gmail.com\qquad bingui@tsinghua.edu.cn\\

\noindent {\small \sc Yau Mathematical Sciences Center and Department of Mathematics, Tsinghua University, Beijing, China.}

\noindent {\textit{E-mail}}: zhanghao1999math@gmail.com \qquad h-zhang21@mails.tsinghua.edu.cn
\end{document}